\documentclass{amsart}
\usepackage{amssymb, graphics}
\usepackage[all]{xy}

\newtheorem{lemma}{Lemma}[section]
\newtheorem{proposition}[lemma]{Proposition}
\newtheorem{corollary}[lemma]{Corollary}
\newtheorem{theorem}[lemma]{Theorem}
\newtheorem{theorem_int}{Theorem} 
\newtheorem{example}[lemma]{Example}
\newtheorem{definition}[lemma]{Definition}
\newtheorem{remark}[lemma]{Remark}

\newtheorem*{Acknowledgement}{Acknowledgement}

\newcommand\cf{cf\@. }
\newcommand\ie{i\@.e\@. }
\newcommand\eg{e\@.g\@. }

\newcommand\pa{ \partial}

\newcommand\bbC{\mathbb C}

\newcommand\bbN{\mathbb N}
\newcommand\bbP{\mathbb P}

\newcommand\bbR{\mathbb R}
\newcommand\bbS{\mathbb S}

\renewcommand\Re{\operatorname{Re}}
\renewcommand\Im{\operatorname{Im}}

\newcommand\tX{\widetilde{X}}

\newcommand\bX{\overline{X}}

\newcommand\hX{\widehat{X}}
\newcommand\hH{\widehat{H}}
\newcommand\hE{\widehat{E}}
\newcommand\tD{\widetilde{D}}
\newcommand\bD{\overline{D}}

\newcommand\hD{\widehat{D}}
\newcommand\CI{\mathcal{C}^{\infty}}
\newcommand\U{\operatorname{U}}
\newcommand\KE{\operatorname{KE}}
\newcommand\fc{\operatorname{fc}}
\newcommand\Diff{\operatorname{Diff}}
\newcommand\Ric{\operatorname{Ric}}

\newcommand\av{\operatorname{av}}
\newcommand\ord{\operatorname{ord}}
\newcommand\cC{\mathcal{C}}
\newcommand\cA{\mathcal{A}}
\newcommand\cE{\mathcal{E}}

\newcommand\cF{\mathcal{F}}
\newcommand\db{\overline{\pa}}
\newcommand\tH{\widetilde{H}}
\newcommand\cV{\mathcal{V}}
\newcommand\cU{\mathcal{U}}
\newcommand\cI{\mathcal{I}}
\newcommand\pfc{\operatorname{fc}}
\newcommand\End{\operatorname{End}}
\newcommand\depth{\operatorname{depth}}
\newcommand\tV{\widetilde{V}}
\newcommand\tU{\widetilde{U}}
\newcommand\tPhi{\widetilde{\Phi}}
\newcommand\hPhi{\widehat{\Phi}}
\newcommand\tb{\widetilde{b}}
\newcommand\tc{\widetilde{c}}
\newcommand\bv{\overline{v}}

\newcommand\bb{\overline{b}}
\newcommand\pr{\operatorname{pr}}
\newcommand\st{\operatorname{st}}
\newcommand\phg{\operatorname{phg}}
\newcommand\Spec{\operatorname{Spec}}
\newcommand\Id{\operatorname{Id}}
\newcommand\rank{\operatorname{rk}}
\newcommand\hu{\widehat{u}}
\newcommand\homega{\widehat{\omega}}
\newcommand\hDelta{\widehat{\Delta}}
\newcommand\halpha{\widehat{\alpha}}
\newcommand\tbeta{\widetilde{\beta}}
\newcommand\tphi{\widetilde{\phi}}
\newcommand\tcU{\widetilde{\mathcal{U}}}

\begin{document}
\title[asymptotics of complete K\"ahler metrics of finite volume]
{Asymptotics of complete K\"ahler metrics of finite volume on quasiprojective manifolds }

\author{Fr\'ed\'eric Rochon}
\address{Department of mathematics, Australian National University\newline
Current address: Universit\'e du Qu\'ebec \`a Montr\'eal}
\email{rochon.frederic@uqam.ca}
\author{Zhou Zhang}
\address{The school of Mathematics and Statistics, University of Sydney}
\email{zhangou@maths.usyd.edu.au}
%\dedicatory{\datverp}
\begin{abstract}   
Let $X$ be a quasiprojective manifold given by the complement of a divisor $\bD$ with normal crossings in a smooth projective manifold $\bX$.  Using a natural compactification of $X$ by a manifold with corners $\tX$, we describe the full asymptotic behavior at infinity of certain complete K\"ahler metrics of finite volume on $X$.  When these metrics evolve according to the Ricci flow, we prove that such asymptotic behaviors persist at later times by showing that  the associated potential function is smooth up to the boundary on the compactification $\tX$.  However, when the divisor $\bD$ is smooth with $K_{\bX}+[\bD]>0$ so that the Ricci flow converges to a K\"ahler-Einstein metric, we show that this K\"ahler-Einstein metric has a rather different asymptotic behavior at infinity, since its associated potential function is polyhomogeneous with, in general, some logarithmic terms occurring in its  expansion at the boundary.  
\end{abstract}
\maketitle

\tableofcontents

\section*{Introduction}

On non-compact complete Riemannian manifolds, the study of the spectral properties of the associated Laplacian usually requires a very good understanding of the asymptotic behavior of the metric at infinity, for instance on conformally compact manifolds \cite{Mazzeo-Melrose} or on manifolds with infinite cylindrical ends \cite{MelroseAPS}.  In complex geometry, an important example of this phenomenon is given by strictly pseudoconvex domains \cite{EMM}, where the natural complete metrics to consider are the Bergman metric and the K\"ahler-Einstein metric of Cheng and Yau \cite{Cheng-Yau}.  Indeed, the potential functions used to define these metrics are typically not smooth up to the boundary, as their expansions there also involve logarithmic terms, see \cite{Fefferman74} in the case of the Bergman metric and \cite{Fefferman76}, \cite{Cheng-Yau}, \cite{Lee-Melrose} in the case of the K\"ahler-Einstein metric of Cheng and Yau.  The detailed description of the possible terms occurring in the expansion at the boundary is crucial in the resolvent construction of the associated Laplacian given in \cite{EMM}.  

The focus of the present paper is on the asymptotic behavior at infinity of complete K\"ahler metrics on another important class of non-compact complex manifolds:  quasiprojective manifolds.  We will restrict our attention to quasiprojective manifolds $X$ taking the form $X=\bX\setminus \bD$, where $\bX$  
is a smooth projective manifold of dimension $n$ and $\bD$ is a divisor with normal crossings, that is, the irreducible components $\bD_1, \ldots, \bD_\ell$ of $\bD$ are smooth and intersect transversely.  Let $L\to \bX$ be a positive holomorphic line bundle and $h_L$ be a choice of Hermitian metric inducing a positive curvature form.  For $i=1,\ldots,\ell$, let also $s_i \in H^0 (\bX; [\bD_i])$ be a choice of section such that $s_i^{-1}(0)= \bD_i$ and let $\|\cdot\|_{\bD_i}$ be   
a choice of Hermitian metric for $[\bD_i]$.  From an idea of Carlson and Griffiths \cite[Proposition~2.1]{Carlson-Griffiths}, we then know that  for $\epsilon>0$ sufficiently small, the $(1,1)$-form
\begin{equation}
\begin{aligned}
\omega &=\sqrt{-1} \Theta_{L} + \sqrt{-1}\ \db\pa \log\left( \prod_{i=1}^{\ell} (-\log \epsilon \|s_{i}\|^{2}_{\bD_{i}})^{2} \right) \\
 &= \sqrt{-1}\Theta_{L} + 2\sqrt{-1}\sum_{i=1}^{\ell} \left(\frac{ \Theta_{\bD_{i}}}{\log\epsilon \|s_{i}\|_{\bD_{i}}^{2} }\right)  \\
 & \quad \quad \quad + 2\sqrt{-1}\sum_{i=1}^{\ell} \left( \frac{ (\pa \log \epsilon \|s_{i}\|^{2}_{\bD_{i}})\wedge (\db \log \epsilon \|s_{i}\|^{2}_{\bD_{i}}) }{( \log \epsilon \|s_{i}\|^{2}_{\bD_{i}})^{2}} 
 \right).\end{aligned}
\label{int.1}\end{equation}
is the K\"ahler form of a complete K\"ahler metric $g_{\omega}$ of finite volume on $X$.  The metric $g_{\omega}$ is the prototypical example of an asymptotically tame polyfibred cusp metric, a notion introduced in Definition~\ref{fc.27} below.  When $\dim_{\bbC}X=1$ and $\bD=\{p_1,\ldots,p_\ell\}$ is a set of distinct points, the metric $g_{\omega}$ asymptotically looks like the Poincar\'e metric $\frac{  |d\zeta|^2}{ |\zeta|^2 (\log |\zeta|)^2 }$ in local holomorphic coordinates near $p_i$, that is, the metric $g_{\omega}$ asymptotically looks like a cusp near each point of $\bD$.  When $\dim_{\bbC}X>1$, a similar phenomenon occurs, the metric $g_{\omega}$ approaching the Poincar\'e metric in the direction transversal to the irreducible component $\bD_i$.  When the divisor $\bD$ is smooth, polyfibred cusp metrics correspond to a special example of $d$-metrics, a notion introduced in \cite{Vaillant}.  In general, asymptotically tame polyfibred cusp metrics descend well to the divisors to give metrics of the same type, a feature that allows to proceed by induction to study their properties.        

When $K_{\bX}+[\bD]>0$, we can take $L=K_{\bX}+[\bD]$ to be our positive holomorphic line bundle.  From the work of Yau \cite{Yau1978}, Cheng and Yau \cite{Cheng-Yau}, Kobayashi \cite{Kobayashi}, Tsuji \cite{Tsuji}, Tian and Yau \cite{Tian-Yau} and Bando \cite{Bando}, we know that there exists a K\"ahler-Einstein metric $g_{\KE}$ bi-Lipschitz to $g_{\omega}$.  The K\"ahler form of the  K\"ahler-Einstein metric $g_{\KE}$ is of the form $\omega_{\KE}= \omega+ \sqrt{-1}\ \pa\db u$ with the function $u$  obtained by solving the complex Monge-Amp\`ere equation 
 \begin{equation}
    \log \left( \frac{ (\omega+ \sqrt{-1} \pa \db u)^{n} }{ \omega^n } \right) -u = F,
\label{int.2}\end{equation}
for some appropriate function $F$.  Alternatively, the K\"ahler-Einstein metric can be obtained by using the Ricci flow.   Indeed, as shown in \cite{Chau2004} or \cite[Example~6.18]{Lott-Zhang}, the Ricci flow with initial metric $g_{\omega}$ exists for all time and converges to the K\"ahler-Einstein metric $g_{\KE}$.  

Given its explicit description, the asymptotic behavior at infinity of the prototypical K\"ahler metric $g_{\omega}$ is certainly well-understood.  Studying the asymptotic behavior at infinity of the K\"ahler-Einstein metric $g_{\KE}$ is thus reduced to understanding the asymptotic behavior of the potential function $u$.  The first result in that direction was obtained by Schumacher, who showed in \cite{Schumacher} (see also the subsequent work \cite{Wu06}) that when the divisor $\bD$ is smooth, the K\"ahler-Einstein metric $g_{\KE}$ asymptotically approaches the corresponding K\"ahler-Einstein metric on $\bD$ in the directions tangent to $\bD$, while it approaches the Poincar\'e metric in the directions transversal to $\bD$.  
In \cite{Lott-Zhang}, Lott and the second author introduced the notion of standard spatial asymptotics for complete K\"ahler metrics on $X$ of finite volume and showed via the use of \'etale groupoids that such metrics continue to have standard spatial asymptotics at later time when they evolve according to the Ricci flow.  These results can be understood as giving the $0$th order term of the asymptotic behavior of the metric at infinity.  

By analogy with what happens for the Bergman metric or the K\"ahler-Einstein metric of Cheng and Yau on strictly pseudoconvex domains, a compactification of $X$ by a manifold with boundary, or more generally, by a manifold with corners, would be needed to give a proper description of higher order terms in the asymptotic behavior of the metric.  This is the starting point of this paper.  When the divisor $\bD$ is smooth, we obtain such a compactification by blowing up $\bD$ in $\bX$ in the sense of Melrose \cite{MelroseAPS},
\begin{equation}
   \hX= [\bX;\bD]= \bX\setminus \bD \bigsqcup S (N\bD),
\label{int.3}\end{equation} 
where $S(N\bD)$ is the unit normal bundle of $\bD$ in $\bX$.  The set $\hX$ is naturally a manifold with boundary $\pa \hX= S(N\bD)$.  In particular, its boundary has an induced circle fibration $\hPhi: \pa \hX\to \bD$.  If $\rho\in \CI(\hX)$ is a choice of boundary defining function for $\pa \hX$, that is, $\rho^{-1}(0)= \pa \hX$, $\rho$ is positive on $\hX\setminus \pa\hX$ and the differential $d\rho$ is nowhere zero on $\pa \hX$, then the compactification we are looking for, which we call the \textbf{logarithmic compactification} $\tX$ of $X$, is obtained by declaring it homeomorphic to $\hX$ with ring of smooth functions $\CI(\tX)$ generated by $\CI(\hX)$ and the function $x= \frac{-1}{\log \rho}$, assuming without loss of generality that $\rho$ is always less than $1$.  The function $x\in \CI(\tX)$ is then a boundary defining function for $\tX$.  The reason for considering $\tX$ instead of $\hX$ is that the works of Schumacher \cite{Schumacher} and Wu \cite{Wu06} strongly suggest it is with respect to the function $x$ that the asymptotic behavior of the K\"ahler-Einstein metric $g_{\KE}$ should be described.

When $\bD$ has normal crossings, we proceed in a similar way by successively blowing up in the sense of Melrose each irreducible component of $\bD$,
\begin{equation}
  \hX= [\bX;\bD_1,\ldots, \bD_{\ell}].
\label{int.4}\end{equation}
The space $\hX$ is then naturally a manifold with corners with a boundary hypersurface $\hH_i$ associated to each irreducible component $\bD_i$ of $\bD$.  The logarithmic compactification $\tX$ of $X$ is then obtained by suitably enlarging the ring of smooth functions of $\hX$.  Similarly, when $D_i= \bD_i \setminus \left( \bigcup_{j\ne i} \bD_i\cap \bD_j\right) $ is non-compact, we can consider its logarithmic compactification $\tD_i$.  It is such that the boundary hypersurface  $\tH_i\subset \tX$ associated to $\bD_i$ comes with a natural circle fibration $\tPhi_i: \tH_i\to \tD_i$.  

On $X$, we can consider two natural spaces of smooth functions, one being the restriction of $\CI(\tX)$ to $X$, and the other being the Cheng-Yau H\"older ring $\CI_{\fc}(X)$ of bounded smooth functions having their covariant derivatives with respect to the metric $g_{\omega}$ bounded on $X$.  Since neither of them is contained in the other, we can also consider the space given by their intersection,
\begin{equation}
      \CI_{\fc}(\tX)= \CI(\tX)\cap \CI_{\fc}(X).  
\label{int.4}\end{equation}
Since a function in $\CI_{\fc}(\tX)$ is in particular in $\CI(\tX)$, it has a Taylor series at each boundary hypersurface $\tH_i$ of $\tX$.  The key fact motivating the introduction of the space $\CI_{\fc}(\tX)$ is that requiring the function to be in $\CI_{\fc}(\tX)$ forces the Taylor series of $f$ at $\tH_i$ to be of the form
\begin{equation}
       f\sim \sum_{k=0}^{\infty} \tPhi_i^{*}(a_i) x_i^k, \quad a_i\in \CI_{\fc}(\tD_i),
\label{int.5}\end{equation}
where $x_i$ is a choice of boundary defining function for $\tH_i$ in $\tX$.  As a consequence, a function $f\in \CI_{\fc}(\tX)$ has a well-defined restriction to $\tD_i$ and its full asymptotic behavior at infinity is completely described by its Taylor series at each boundary hypersurface of $\tX$.  The space $\CI_{\fc}(\tX)$ provides the right framework to describe the asymptotic behavior of complete K\"ahler metrics on $X$ bi-Lipschitz to $g_\omega$.  

Using the logarithmic compactification $\tX$ and the space $\CI_{\fc}(\tX)$, our first result concerns the evolution of the asymptotic behavior of K\"ahler metrics under the Ricci flow, answering a question of \cite{Lott-Zhang} (see Theorem~\ref{ha.1} below for a more precise statement).   

\begin{theorem_int}
If $g_{\omega}$ is an asymptotically tame polyfibred cusp K\"ahler metric on $X$ and $\widetilde\omega_t= \omega_t + \sqrt{-1}\pa\db u(t,\cdot)$, with 
$\omega_t= -\Ric(\omega) + e^{-t}(\omega+ \Ric(\omega))$, is the solution to the normalized Ricci flow for $t\in [0,T)$ with
\[
          \frac{\pa u}{\pa t}= \log \left(  \frac{ (\omega_t+ \sqrt{-1}\pa\db u)^n}{\omega_0^n} \right) -u, \quad u(0,\cdot)=0,
\] 
then $g_{\widetilde{\omega}_t}$ is an asymptotically tame polyfibred cusp K\"ahler metric and $u(t,\cdot)\in \CI_{\fc}(\tX)$ for all $t\in [0,T)$.  
\label{int.6}\end{theorem_int}

When $\dim_{\bbC}X=1$, this result was obtained in \cite{Albin-Aldana-Rochon}.  We refer also to \cite{Bahuaud} for a related result for conformally compact metrics. Our general strategy to prove the result when $\dim_{\bbC}X>1$ is similar to the one of \cite{Albin-Aldana-Rochon}, namely, we restrict the evolution equation of the potential function to $\tD_i$ and solve it to get a candidate $u_i$ for what should be the restriction of $u$ to $\tD_i$.  Using a suitable decay estimate proved using a barrier function and the maximum principle (see Proposition~\ref{de.3} below), we then check $u_i$ is indeed the restriction of $u$ to $u_i$.  We can then proceed recursively in the same fashion to build up the whole Taylor series of $u$ at $\tD_i$ for each $i$ and show $u(t,\cdot)\in \CI_{\fc}(\tX)$.  At first, we can only show that this is the case for $t<\tau$ for some small $\tau$.  However, since we can get a uniform lower bound on $\tau$ for each compact subinterval of $[0,T)$, we can repeat the argument finitely many times for the theorem to be true on any compact subinterval of $[0,T)$, establishing the result.      

Our second result consists in generalizing the work of Schumacher \cite{Schumacher} to the case where $\bD$ has normal crossings, namely, we show that near each irreducible component $\bD_i$ of $\bD$, the K\"ahler-Einstein metric $g_{\KE}$ approaches the Poincar\'e metric in the directions transversal to $D_i$ and the corresponding K\"ahler-Einstein metric $g_{\KE,i}$ on $D_i$ in the directions tangent to $D_i$.  This can be formulated in terms of the potential function solving the Monge-Amp\`ere equation \eqref{int.2} (see  Theorem~\ref{ake.10} below for a more precise statement).      
\begin{theorem_int}
Suppose $K_{\bX}+[\bD]>0$ and let $\omega$ be the K\"ahler form of \eqref{int.1} with $L=K_{\bX}+[\bD]$.  Let $u$ be the solution to the complex Monge-Amp\`ere equation \eqref{int.2} so that $\omega_{\KE}= \omega+\sqrt{-1}\pa\db u$ is the K\"ahler form of the K\"ahler-Einstein metric $g_{\KE}$ bi-Lipschitz to $g_{\omega}$.  If $\omega$ has standard spatial asymptotics associated to $\{\omega_I\}$ and $\{c_i\}$ with $c_i=1$ for all $i\in\{1,\ldots,\ell\}$,  then there exists $\nu>0$ such that for
each $i\in \{1,\dots,\ell\}$,  
   \[
                        u -\tPhi_i^{*}( u_i)  \in x_i^{\nu}\CI_{\fc}(X)
   \]
in a collar neighborhood of $\tH_i$,    
 where $x_i$ is a boundary defining function for $\tH_i$ and  
  $u_i$ is such that $\omega_i+\sqrt{-1}\pa\db u_i$ is the K\"ahler form of the K\"ahler-Einstein metric on $D_i$ bi-Lipschitz to $g_{\omega_i}$.    
 \label{int.7}\end{theorem_int}

 From Theorem~\ref{int.6}, we would naively expect the K\"ahler-Einstein metric  to be also an asymptotically tame polyfibred cusp K\"ahler metric.  When $\dim_{\bbC}X=1$, this is indeed the case as described in \cite{Albin-Aldana-Rochon}.  However, when $\dim_\bbC X>1$ and the divisor $\bD$ is smooth, this is no longer the case as our next result shows (see Theorem~\ref{ake.20} below for a more precise statement).  The same phenomenon occurs for the Ricci flow of conformally compact metrics, see \cite{Fefferman_Graham2012} and \cite{Bahuaud} in that context.   
\begin{theorem_int}
Suppose the divisor $\bD$ is smooth and $K_{\bX}+[\bD]>0$.  Let $u$ be the solution to the complex Monge-Amp\`ere equation \eqref{int.2} so that $\omega_{\KE}= \omega+\sqrt{-1}\pa\db u$ is the K\"ahler form of the K\"ahler-Einstein metric $g_{\KE}$ bi-Lipschitz to $g_{\omega}$.  Then there exists an index set $E\subset [0,\infty)\times \bbN_0$ such that $u$ has an asymptotic expansion at $\pa \tX$ of the form
\[
      u \sim \sum_{(z,k)\in E} \tPhi^{*}(a_{z,k}) x^z (\log x)^{k}, \quad a_{z,k}\in \CI(\bD).
\]
Moreover, the index set $E$ is such that 
\[
   (z,k)\in E, \; z\le 1 \quad \Longrightarrow \quad (z,k)\in \{(0,0), (1,0), (1,1)\}.  
\]
\label{int.8}\end{theorem_int}

This result refines the asymptotic expansion of \cite{Wu06}.  Notice that the asymptotic expansion in \cite{Wu06} predicts no $x\log x$ term.  We were informed by Damin Wu that this will be addressed in an erratum to \cite{Wu06} that should appear soon.  
     
In \cite{Lee-Melrose}, the main ingredient is a boundary regularity result for a corresponding linear elliptic equation.
For edge singularities, the approach of Lee and Melrose was subsequently systematized by Mazzeo in \cite{MazzeoEdge}, \cite{Mazzeo1991} and this was used recently in \cite{Jeffres-Mazzeo-Rubinstein} to establish the polyhomogeneity of incomplete K\"ahler-Einstein metrics with an edge singularity along a divisor.   In our case, we use a boundary regularity result which is adapted to the geometry of the problem  (see Theorem~\ref{br.31} below).  Our strategy to obtain the latter relies on an observation (Lemma~\ref{br.6}) that allows us to deduce the result from a corresponding boundary regularity for elliptic $b$-operators \cite{MelroseAPS}.   

In contrast to the result of Lee and Melrose \cite{Lee-Melrose}, notice that non-integer powers of the boundary defining function might also occur in the asymptotic expansion of the potential function.  The first possible logarithmic term is of the form $x\log x$.  Unless such a term does not occur, the optimal H\"older regularity of the potential function on the logarithmic compactification $\tX$ is therefore easily seen to be
\[
 u\in \cC^{0,\delta}(\tX), \quad  \mbox{for all} \; 0<\delta<1.
 \]

In Theorem~\ref{clt.6}, we provide a topological criterion determining when such logarithmic term actually occurs. In complex dimension 2, this criterion is particularly simple to describe:  there is a term $x\log x$ in the asymptotic expansion of $u$ if and only if the (complex) normal bundle of $\bD$ in $\bX$ is non-trivial.  Thus, an easy example of a K\"ahler-Einstein metric with such a logarithmic term is obtained by taking $\bX=\bbC\bbP_2$ with $\bD\subset \bbC\bbP_2$ a smooth curve of degree at least 4 (see Example~\ref{clt.7} below).  This can also be used to construct an example where $\bD$ has normal crossings and the solution $u$ to the complex Monge-Amp\`ere equation is not in $\CI_{\fc}(\tX)$ (see Example~\ref{clt.8}).  It seems natural to expect such solutions to be polyhomogeneous in some reasonable sense and we hope to investigate this matter in a future work.

The paper is organized as follows.  In \S~\ref{mwc.0}, we introduce the logarithmic compactification of the quasiprojective manifold $X$.  It is followed in \S~\ref{fc.0} by a description of the type of complete metrics of finite volume we will consider on $X$.  At the same time, we take the opportunity to introduce various spaces of functions, notably $\CI_{\fc}(\tX)$ and its corresponding polyhomogeneous version.  As a prelude to the study of the Ricci flow, we obtain in \S~\ref{de.0} some decay estimates and regularity results for linear (uniformly) parabolic equations on $X$.  This is used in \S~\ref{ha.0} to prove Theorem~\ref{int.6}.  We prove the generalization of Schumacher's result  in \S~\ref{ake.0}.  We then focus on the case where the divisor $\bD$ is smooth and obtain in \S~\ref{br.0} a boundary regularity result for solutions of linear (uniformly) elliptic equations on $X$.  This is used in \S~\ref{clt.0} to obtain our boundary regularity result for the K\"ahler-Einstein metric $g_{\KE}$.

\begin{Acknowledgement}
 The first author would like to thank the University of Sydney for its hospitality. The second author was partially  supported by ARC Discovery grant DP110102654 and is grateful to Xujia Wang for inviting him to visit the Australian National University.  The authors are grateful to Rafe Mazzeo and Richard Melrose for helpful discussions.
\end{Acknowledgement}

\numberwithin{equation}{section}
\section{The logarithmic compactification of a quasiprojective manifold} \label{mwc.0}  

Let $\bX$ be a smooth projective manifold and $\bD$ be a divisor on $\bX$ with normal crossings, that is, each irreducible component of $\bD$ is smooth and the irreducible components of $\bD$ intersect transversely.  We denote by $X=\bX\setminus \bD$ the corresponding quasiprojective manifold.  Write 
\begin{equation}
      \bD= \sum_{i=1}^{\ell} \bD_{i}
\label{mwc.1}\end{equation}
where the sum runs over irreducible components of $\bD$.  For each irreducible component $\bD_{i}$, let $s_{i}: \bX\to [\bD_{i}]$ be a holomorphic section defining $\bD_{i}$, namely $s^{-1}_{i}(0)= \bD_{i}$.  Let also $\|\cdot\|_{\bD_{i}}$ be a choice of Hermitian metric for the holomorphic line bundle $[\bD_{i}]$ and consider on $\bX$ the real-valued function $\rho_{i}= \|s_{i}\|_{\bD_{i}}$.  This function is smooth away from $\bD_{i}$ and such that $\rho_{i}^{-1}(0)= \bD_{i}$.  Denote by 
\begin{equation}
    \hX= [\bX; \bD_{1}, \ldots, \bD_{\ell}]
\label{mwc.2}\end{equation}
the manifold with corners obtained from $\bX$ by successively blowing up $\bD_{1}, \ldots, \bD_{\ell}$ in the sense of Melrose \cite{MelroseMWC}.  We denote by $\widehat{\beta}:\hX\to \bX$ the corresponding blow-down map.  Since the various irreducible components of $\bD$ intersect transversely, the diffeomorphism type of the manifold $\hX$ does not depend on the order in which the blow-ups are performed.  In fact, if $\hX'$ is obtained from $\bX$ by blowing up the irreducible components of $\bD$ in a different order, then the natural identification $\hX'\setminus \pa\hX'= X= \hX\setminus \pa\hX$ extends uniquely to give a diffeomorphism $\hX'\cong \hX$, see Proposition~5.8.2 in \cite{MelroseMWC} or \cite[p.21]{Mazzeo-MelroseETA}.    

The manifold with corners $\hX$ has $\ell$ distinct boundary hypersurfaces,
\begin{equation}
     \hH_{i}= \widehat{\beta}^{-1}(\bD_{i}), \quad i\in \{1,\ldots,\ell\}.
\label{mwc.3}\end{equation}
A useful way of measuring the complexity of the manifold with corners $\hX$ is via its \textbf{depth}, 
\begin{equation}
    \depth(\hX) = \max\{ \; | I | \; | \; I\subset\{1,\ldots,\ell\}, \; \hH_I= \bigcap_{i\in I} \hH_i \ne \emptyset\}.
\label{mwc.3b}\end{equation}
It is the highest possible codimension of a boundary face of $\hX$. 

Under the blow-down map, the function $\rho_{i}$ lifts naturally to a smooth function $\rho_{i}\in \CI(\hX)$ which is a boundary defining function of $\hH_{i}$, that is, $\rho_{i}^{-1}(0)= \hH_{i}$, $\rho_{i}>0$ on $\hX\setminus \hH_{i}$ and the differential $d\rho_{i}$ is nowhere zero on $\hH_{i}$.

Consider the intersection $\bD_I= \bigcap_{i\in I} \bD_i$ for a subset $I\subset \{1,\ldots,\ell\}$.  When $\bD_I$ is non-empty, it is a complex submanifold of $\bX$ of codimension $|I|$.  It comes with a    
 natural divisor given by 
\begin{equation}
      C_{I}= \sum_{j \notin I} \bD_{I}\cap \bD_{j}.  
\label{mwc.4}\end{equation}    
By transversality of the irreducible components of $\bD$,  the divisor $\bD_{I}\cap \bD_{j}$ is a  disjoint union of smooth irreducible divisors of $\bD_{I}$.  Moreover, for distinct $j,k\notin I$, we have that
$\bD_{I}\cap \bD_{j}$ and $\bD_{I}\cap \bD_{k}$ intersect transversely in $\bD_{I}$.   Thus, if we write
\begin{equation}
  C_{I} =\sum_{j=1}^{\ell_{I}} C_{Ij},
\label{mwc.5}\end{equation}
where the sum runs over the irreducible components of $C_{I}$,  we can consider the manifold with corners
\begin{equation}
       \hD_{I}= [ \bD_{I}; C_{I1}, \ldots, C_{I \ell_{I}}]
\label{mwc.6}\end{equation} 
obtained by successively blowing up the submanifolds $C_{I1}, \ldots, C_{I\ell_{I}}$ in the sense of Melrose.  As for $\hX$, the diffeomorphism type of $\hD_{I}$ does not depend on the order in which the blow-ups are performed.  When $I=\{i\}$ consists of one element, we will use the notation $\hD_i := \hD_{\{i\}}$.  

Let $\U_{i}\to \bX$ be the unit circle bundle associated to the Hermitian line bundle $([\bD_{i}], \|\cdot\|_{\bD_{i}})$.  Let $\left.  \U_{i}\right|_{\bD_{i}}$ be its restriction to $\bD_{i}$ and consider its pull-back 
\begin{equation}
    \widehat{\U}_{i}= \widehat{\beta}_{i}^{*}\left( \left. \U_{i}\right|_{ \bD_{i}} \right)  
\label{mwc.7}\end{equation}
under the blow-down map $\widehat{\beta}_{i}: \hD_{i}\to \bD_{i}$.  More generally, for $I\subset\{1,\ldots,\ell\}$, let $\U_I\to \bX$ be the torus bundle obtained by taking the fibre product of the circle bundles $\U_i$ for $i\in I$.  Let $\left.\U_I\right|_{\bD_I}$ be the restriction of $\U_I$ to $\bD_I$ and consider its pull-back
\begin{equation}
   \widehat{\U}_I= \widehat{\beta}_I^* ( \left.\U_I\right|_{\bD_I})
\label{mwc.7b}\end{equation}
under the blow-down map $\widehat{\beta}_I: \hD_I\to \bD_I$.

\begin{lemma}
When non-empty, the boundary face $\hH_{I}= \bigcap_{i\in I} \hH_i \subset \hX$ is naturally diffeomorphic to the total space of the torus bundle $\widehat{U}_{I}\to \hD_{I}$.  In particular, there exists an induced torus fibration 
$\hPhi_{I}: \hH_{I}\to \hD_{I}$.  
\label{mwc.6}\end{lemma}
\begin{proof}
Changing the order in which we blow up if needed, we can assume $I=\{1,\ldots,k\}\subset \{1,\ldots,\ell\}$.  For $p\in \bD_I$, let us choose for each $i\in I$ a local holomorphic trivialization of $[\bD_i]$ in a small neighborhood $\cU$ of $p$.  The section $s_i$ can then be seen as a local holomorphic function $\zeta_i$ on $\cU$ vanishing on $\bD_i$.  Taking $\cU$ smaller if needed, we can complete $\zeta=(\zeta_1,\ldots,\zeta_k)$ by holomorphic functions $z=(z_{k+1}, \ldots, z_n)$ such that 
$(\zeta,z)$ form holomorphic coordinates on $\cU$.  In these coordinates, blowing up $\bD_i$ amounts to introduce polar coordinates $(\rho_i',\theta_i)$ such that $\zeta_i= \rho_i' e^{\sqrt{-1}\theta_i}$.  Thus, if we denote by $\halpha_I$ the blow-down map for the manifold with corners
$\hX_I =[\bX;\bD_1,\ldots,\bD_k]$, then on the open set $\halpha_I^{-1}(\cU)$, we can use the coordinates  $(\rho_1',\theta_1,\ldots,\rho_k',\theta_k,z)$, in which case the blow-down map $\halpha_I$ is locally given by:
\begin{equation}
   (\rho_1',\theta_1,\ldots,\rho_k',\theta_k,z)\mapsto (\rho_1e^{\sqrt{-1}\theta_1}, \ldots, \rho_k' e^{\sqrt{-1}\theta_k},z).
\label{tf.1a}\end{equation}
On the boundary face $\halpha_I^{-1}(\bD_I)$, we can use the coordinates $(\theta_1,\ldots,\theta_k,z)$.  In these coordinates, the restriction of the blow-down map $\halpha_I$ to $\halpha_I^{-1}(\bD_I)$ is given by
\begin{equation}
  (\theta_1,\ldots,\theta_k,z)\mapsto z.
\label{tf.1}\end{equation}
If we choose different local holomorphic trivializations of $[\bD_i]$ for $i\in I$ on an open set $\cV$ with $\cV\cap \bD_I\ne\emptyset$, then the sections $s_i$ will give a different tuple $\zeta'=(\zeta_1',\ldots,\zeta_k')$ of holomorphic functions.  On the intersection $\cU\cap \cV$, we can assume, taking the support of $\cU$ and $\cV$ to be closer to $\bD_I$ if needed, that $(\zeta',z)$ still form holomorphic coordinates.  In particular, this will induce coordinates $(\theta_1',\ldots,\theta_k',z)$ on $\halpha_I^{-1}(\bD_I)$ with the restriction of the blow-down map $\halpha_I$ taking the form
\begin{equation}
   (\theta_1',\ldots, \theta_k',z) \mapsto z. 
\label{tf.2}\end{equation} 
If $\zeta_i'= \zeta_i f_i(\zeta,z)$, for $f_i$ a non-vanishing holomorphic function, is the change of coordinates from $(\zeta,z)$ to $(\zeta',z)$, then on $\halpha_I^{-1}(\bD_I)$, the corresponding change of coordinates is given by
\begin{equation}
  e^{\sqrt{-1}\theta_i'}= e^{\sqrt{-1}\theta_i} \frac{f_i(0,z)}{|f_i(0,z)|}, \quad i\in I.
\label{tf.3}\end{equation} 
The function $ \frac{f_i(0,z)}{|f_i(0,z)|}$ takes value on the unit circle and can be interpreted as the transition function between local trivializations of the circle bundle $\left.\U_i\right|_{\bD_I}$.  This means the fibration $\halpha_I: \halpha_I^{-1}(\bD_I)\to \bD_I$ is a torus bundle isomorphic to the torus bundle $\left. \U_I\right|_{\bD_I}\to \bD_I$.

Now, when we blow up the lifts of $\bD_{k+1}, \ldots, \bD_n$ in $\hX_I$ to obtain $\hX$, the boundary face $\halpha_I^{-1}(\bD_I)\subset \hX_I$ lifts to give the boundary face $\hH_I\subset \hX$.  Since the lifts of $\bD_{k+1},\ldots,\bD_n$  to $\hX_I$ are transversal to $\halpha_I^{-1}(\bD_I)$, we see using local coordinates as in \eqref{tf.1a} that the torus bundle $\halpha_I: \halpha_I^{-1}(\bD_I)\to \bD_I$ lifts to a fibration
\[
    \hPhi_I: \hH_I\to \hD_I
\]
which is just the pull-back of the torus bundle $\halpha^{-1}_I(\bD_I)\to \bD_I$ under the blow-down map $\widehat{\beta}_I: \hD_I\to \bD_I$.
\end{proof}

Recall from \S~5.14 in \cite{MelroseMWC} that the \textbf{logarithmic blow-up} $[\hX;\hH_{i}]_{\log}$ of a boundary hypersurface $\hH_{i}$ with boundary defining function $\rho_{i}$ is the manifold with corners $\hX$ with $\CI$-structure generated by $\CI(\hX)$ and the new boundary defining function
\begin{equation}
  x_{i}= \frac{1}{\log \left( \frac{1}{\rho_{i}}\right)},
\label{mwc.7c}\end{equation}
where we assume without loss of generality that $\rho_{i}<1$ everywhere on $\hX$.
As shown in \cite{MelroseMWC}, the definition of $[\hX;\hH_{i}]_{\log}$ does not depend on the choice of the boundary defining function $\rho_{i}$ and is such that the identity map $[\hX;\hH_{i}]_{\log}\to \hX$ is smooth.  More generally, we can perform such a logarithmic blow-up at each boundary hypersurface with the different blow-ups commuting.  We get in this way the total logarithmic blow-up 
\begin{equation}
  \hX_{\log}= \hX_{\ell}, \quad \mbox{with} \;\;\hX_{0}= \hX \; \;\mbox{and} \; \; \hX_{j}= [\hX_{j-1};\hH_{j}]_{\log}, \; j\in\{1,\ldots,\ell\},
\label{mwc.8}\end{equation}
whose diffeomorphism type is independent of the order in which the logarithmic blow-ups are performed.  The identity map plays the role of a blow-down map $\widehat{\beta}_{\log}: \hX_{\log}\to \hX$.  Under this map, the pull-back of polynomially decaying functions on $\hX$ gives rapidly decaying functions on $\hX_{\log}$.  Moreover, the  boundary face $\tH_{I}= \widehat{\beta}_{\log}^{-1}(\hH_{I})$ of $\hX_{\log}$ is canonically diffeomorphic to the total logarithmic blow-up of the boundary hypersurface $\hH_{I}$.  The smooth torus fibration $\hPhi_{I}: \hH_{I}\to \widehat{D}_{I}$ induces  a corresponding smooth torus fibration $\widetilde{\Phi}_{I}: \tH_{I}\to \widetilde{D}_{I}$, where $\tD_I= (\hD_I)_{\log}$. 
\begin{definition}
The \textbf{logarithmic compactification} of the quasiprojective manifold $X=\bX\setminus \bD$ is the manifold with corners $\tX:= \hX_{\log}$.  It has a natural blow-down map 
\[
        \widetilde{\beta}=  \widehat{\beta}\circ \widehat{\beta}_{\log}: \tX\to \overline{X}.
\]
For each boundary hypersurface $\tH_{i}= \widetilde{\beta}^{-1}(\bD_{i})$, a choice of boundary defining function is given by
\[
          x_{i}= \frac{1}{\log \left( \frac{1}{\rho_{i}} \right) } = \frac{-1}{\log \| s_{i}\|_{\bD_{i}} }\in \CI(\tX).
\] 
\label{mwc.9}\end{definition}

The boundary defining function $x_{i}$ induces a non-vanishing section $\left. dx_{i} \right|_{\tH_{i}}$ of the conormal bundle of $\tH_{i}$ in $\tX$.  As the next lemma shows, this section, which induces a trivialization of the conormal bundle, is canonical in the sense that it does not depend on the choice of the boundary defining function $\rho_{i}\in \CI(\hX)$. 

\begin{lemma}
Let $\|\cdot \|_{\bD_{i}}'$ be another choice of Hermitian metric for the bundle $[\bD_{i}]$ and suppose that the corresponding boundary defining function $\rho_{i}'= \|s_{i}\|'_{\bD_{i}}$ 
for the boundary hypersurface $\hH_{i}$ is such that  $0\le \rho_{i}'<1$.  Then the new boundary function
\[
     x_{i}'= \frac{1}{\log \frac{1}{\rho_{i}'} }\in \CI(\tX)    
\]
for the boundary hypersurface $\tH_{i}$ in $\tX$ is such that $x_{i}'= x_{i}(1+x_{i}\bb_{i}+x_{i}^{2}\tb_{i})$ for some smooth functions $\bb_{i}\in \CI(\bX)$ and $\tb_{i}\in \CI(\tX)$.  In particular,
\[
       \left. dx_{i}'\right|_{\tH_{i}}= \left. dx_{i}\right|_{\tH_{i}}.
\]
\label{mwc.10}\end{lemma}
\begin{proof}
By assumption, we have that $\rho_{i}'= \rho_{i} \phi_{i}$ for some function $\phi_{i}\in \CI(\bX)$ such that $0 < \phi_{i}$ everywhere on $X$. Since $0\le \rho_i< 1$ and $0\le \rho_i' <1$, notice that
\[
    1- x_i \log \phi_i = 1+ \frac{\log \phi_i}{\log \rho_i} = \frac{\log \rho_i +\log \phi_i}{\log \rho_i} = \frac{\log \rho_i'}{\log \rho_i} >0.  
\]
Thus, we have,
\begin{equation}
\begin{aligned}
x_{i}' &= \frac{-1}{\log \rho_{i} + \log \phi_{i}} = x_{i}\left( \frac{1}{ 1-x_{i}\log\phi_{i}}   \right)  \\
&=  x_{i} \left( 1+ \frac{x_{i} \log \phi_{i}}{ 1- x_{i}\log\phi_{i}}  \right),  \\
&= x_{i}\left( 1 + x_{i}\log\phi_{i} \left( \frac{1-x_{i}\log\phi_{i} + x_{i}\log\phi_{i}}{1-x_{i}\log \phi_{i}}\right) \right) \\
&= x_{i}\left(  1+ x_{i}\log\phi_{i} + x_{i}^{2}\left( \frac{(\log\phi_{i})^{2}}{1-x_{i}\log\phi_{i}}\right) \right),
\end{aligned}
\label{mwc.11}\end{equation}
and the result follows by taking $\bb_{i}=\log\phi_{i}$ and $\tb_{i}= \frac{(\log\phi_{i})^2}{1-x_{i}\log\phi_{i}}$.
\end{proof}

\section{Polyfibred cusp metrics on quasiprojective manifolds} \label{fc.0}

On $\tX$, consider the subspace of smooth vector fields
\begin{multline}
  \cV_{\tPhi}(\tX)= \{ \xi\in \CI(\tX;T\tX) \quad | \quad \xi x_{i} \in x_{i}^{2}\CI(\tX) \\
     \mbox{and} \; (\tPhi_{i})_{*} (\left.\xi\right|_{\tH_{i}})=0 \quad \forall i\in \{1,\ldots,\ell\} \}.
\label{fc.1}\end{multline}
When the divisor $\bD$ is smooth, this is a special case of the Lie algebra of vector fields introduced by Mazzeo and Melrose in \cite{Mazzeo-MelrosePhi}.
As can be checked directly, the subspace $\cV_{\tPhi}(\tX)$ is closed under Lie bracket, so forms a Lie subalgebra of $\CI(\tX;T\tX)$.  Thanks to Lemma~\ref{mwc.10}, the condition 
\[
       \xi x_{i} \in x_{i}^{2}\CI(\tX)    
\]
does not depend on the choice of the boundary defining function $\rho_{i}\in \CI(\hX)$ used to define $x_{i}$.   This means the Lie subalgebra $\cV_{\tPhi}(\tX)$ is canonically associated to the logarithmic compactification $\tX$.  

To describe $\cV_{\tPhi}(\tX)$ in local coordinates,  let $I\subset \{1,\ldots,\ell\}$ be a non-empty subset such that the boundary face
\begin{equation}
    \tH_{I}= \bigcap_{i\in I} \tH_{i}
\label{fc.2}\end{equation}
is non-empty.  Relabelling the boundary hypersurfaces if needed, we can assume $I= \{1,\ldots,k\}$ for some $k\le \ell$.  Given a point $p\in \tH_{I}$,  consider the image $\widetilde{\beta}(p)$ of $p$ in $\bX$.  By considering a local trivialization of $[\bD_{i}]$ near $\widetilde{\beta}(p)$ for $i\in I$, we can regard the section $s_{i}$ as a holomorphic function $\zeta_{i}$ near $\widetilde{\beta}(p)$.  We complete these functions by holomorphic functions $z_{k+1},\ldots,z_{n}$ to get holomorphic coordinates 
\begin{equation}
   (\zeta,z)=(\zeta_{1},\ldots,\zeta_{k},z_{k+1},\ldots,z_{n})
\label{fc.16}\end{equation}  
near $\widetilde{\beta}(p)$.  Without loss of generality, we will assume that the open set $\cU$ where these coordinates are defined only intersects with $\bD_i$ for $i\in I$.  In terms of the polar coordinates 
\begin{equation}
        \rho_{i}'= |\zeta_{i}|, \; \theta_{i}= \arg(\zeta_{i}), \quad \zeta_{i}= \rho_{i}'e^{\sqrt{-1}\theta_{i}}, 
\label{fc.17}\end{equation}
we can consider $\rho_{i}'= |\zeta_{i}|$ instead of $\rho_{i}= \|s_{i}\|_{\bD_{i}}$ as a (local) boundary defining function for $\hH_{i}$ in $\hX$.  Writing $z_{j}=u_{j}+\sqrt{-1}v_{j}$ where $u_{j}$ and $v_{j}$ are real-valued functions, and setting 
\[
w=(w_1, \ldots, w_{2q}):=(u_{k+1}, v_{k+1},\ldots,u_n,v_n),  \quad  q=n-k,
\]
we obtain in this way real coordinates near $p$ 
\begin{equation}
  (x_{1}',\theta_{1},\ldots, x_{k}',\theta_{k}, w_1,\ldots, w_{2q}),
\label{fc.18}\end{equation}
with $x_{i}'= \frac{-1}{\log \rho_{i}'}$ for $i\in\{1,\ldots,k\}$.     Let $\tU=\tbeta^{-1}(\cU)\subset \tX$ be the open set where these coordinates are defined.  By our choice of $\tU$, we   have that $\tH_i\cap \tU=\emptyset$ unless $i\in I$. 

The open set $\tU$ is such that for each $i\in \{1,\ldots,k\}$, the fibration $\tPhi_{i}: \tH_{i}\to \tD_{i}$ restricts to be trivial on $\tU\cap \tH_{i}$.  In fact, for each $i\in I$, 
\[
(x_1',\ldots\widehat{x_{i}'}, \ldots , x_k',\theta_{1}, \ldots, \widehat{\theta_{i}},\ldots,\theta_{k}, w)
\]
 forms a coordinate system on $\tD_{i}$ such that the projection $\tPhi_{i}$ is given by
\begin{equation}
 (x_1',\ldots\widehat{x_{i}'}, \ldots , x_k',\theta_{1},\ldots,\theta_{k}, w) \mapsto  (x_1',\ldots\widehat{x_{i}'}, \ldots , x_k',\theta_{1}, \ldots, \widehat{\theta_{i}},\ldots,\theta_{k}, w).  
 \label{fc.4}\end{equation} 
Here, the notation \ $\widehat{}$ \ above a variable means it is omitted.  By Lemma~\ref{mwc.10}, the coordinate $x_i'$ is also a valid choice of local boundary defining function for $\tH_i$.  Thus, in the coordinate system \eqref{fc.18}, a vector field $\xi\in \cV_{\tPhi}(\tX)$ is of the form
\begin{equation}
  \xi= \sum_{i=1}^{k} \left( a_{i} x_{I}'x_{i}' \frac{\pa}{\pa x_{i}'} + \frac{b_{i} x_{I}'}{x_{i}'} \frac{\pa}{\pa \theta_{i}}\right) + \sum_{j=1}^{2q} c_{j} x_{I}' \frac{\pa}{\pa w_{j}} ,
\label{fc.5}\end{equation}
where $x_{I}'= \prod_{i\in I} x_{i}'$ and $a_{i},b_{i}, c_{j} \in \CI(\tX)$.

Since $\cV_{\tPhi}(\tX)$ is a $\CI(\tX)$-module, there exists a smooth vector bundle ${}^{\tPhi}T\tX\to \tX$ and a natural map $\iota_{\tPhi}: {}^{\tPhi}T\tX\to T\tX$ which restricts to an isomorphism on $\tX \setminus \pa \tX$ with the property that 
\begin{equation}
     \cV_{\tPhi}(\tX)= \iota_{\tPhi} (\CI (\tX;{}^{\tPhi}T\tX)).
\label{fc.6}\end{equation}
At a point $p\in \tX$, the fibre of ${}^{\tPhi}T\tX$ above $p$ can be defined by 
\begin{equation}
   {}^{\tPhi}T_{p}\tX= \cV_{\tPhi}(\tX)/ \cI_{p} \cV_{\tPhi}(\tX),
\label{fc.7}\end{equation}
where $\cI_{p} \subset \CI(\tX)$ is the ideal of all smooth functions vanishing at $p$.  
\begin{definition}
  A \textbf{polyfibred boundary metric} on $\tX$ is a smooth metric $g_{\tPhi}$ for the real vector bundle ${}^{\tPhi}T\tX\to \tX$.  Under the canonical identification between $T\tX$ and ${}^{\tPhi}T\tX$ over $X= \tX\setminus \pa \tX$, a polyfibred boundary metric induces a Riemannian metric on $X= \tX\setminus \pa \tX$ which we also refer to as a polyfibred boundary metric.    
\label{fc.8}\end{definition}

In the local coordinates \eqref{fc.18}, an example of polyfibred boundary metric is given by 
\begin{equation}
   g_{\tPhi}= \sum_{i=1}^{k} \left( a_{i} \frac{d(x_{i}')^{2}}{(x_{i}'x_{I}')^{2}} + \frac{b_{i}(x_{i}')^{2}}{(x_{I}')^{2}} d\theta_{i}^{2} \right) + \sum_{\alpha=1}^{2q} \sum_{\beta=1}^{2q} c_{\alpha \beta} \frac{dw_{\alpha}\otimes dw_\beta}{(x_{I}')^{2}},
\label{fc.10}\end{equation}
where $a_{i}, b_{i} \in \CI(\tX)$ are positive smooth functions and $c_{\alpha \beta}\in \CI(\tX)$ gives the coefficients of a positive definite symmetric matrix.  If $g_{\tPhi}$ is any choice of polyfibred boundary metric, then notice that the Lie algebra $\cV_{\tPhi}(\tX)$ admits the following alternative description,
\begin{equation}
  \cV_{\tPhi}(\tX)= \{ \xi\in \CI(\tX, T\tX) \quad | \quad 
      \exists c>0 \; \mbox{such that} \; g_{\tPhi}(\xi_{p},\xi_{p})< c \; \forall p\in X\}. 
\label{fc.10b}\end{equation}

The Riemannian metrics we are interested in are not polyfibred boundary metrics, but metrics conformal to them.  

\begin{definition}
A \textbf{polyfibred cusp metric} on $X$ is a Riemannian metric of the form
\[
           g_{\pfc}= x^{2} g_{\tPhi}
\]
where $x= \prod_{i=1}^{\ell} x_{i}$ and $g_{\tPhi}$ is a polyfibred boundary metric.  
\label{fc.9}\end{definition}

 When the divisor $\bD$ is smooth so that the logarithmic compactification $\tX$ is a manifold with fibred boundary, polyfibred cusp metrics constitute a special kind of $d$-metrics, a notion introduced by Vaillant in \cite{Vaillant}. In the local coordinates \eqref{fc.18}, an example of polyfibred cusp metric is given by
\begin{equation}
g_{\pfc}= \sum_{i=1}^{k} \left( a_{i} \frac{d(x_{i}')^{2}}{(x_{i}')^{2}} + b_{i}(x_{i}')^{2} d\theta_{i}^{2} \right) +\sum_{\alpha=1}^{2q} \sum_{\beta=1}^{2q} c_{\alpha \beta} dw_{\alpha}\otimes dw_\beta.
\label{fc.11}\end{equation}
More generally, a polyfibred cusp metric may also involve mixed terms of the form 
\begin{equation}
  \frac{dx_{i}'}{x_{i}'}\otimes \frac{dx_{j}'}{x_{j}'}, \quad x_{i}'d\theta_{i}\otimes x_{j}'d\theta_{j}, \quad  \frac{dx_{i}'}{x_{i}'} \otimes x_{j}' d\theta_{j}, \quad \frac{dx_{i}'}{x_{i}'} \otimes dw_{l}, \quad  x_{j}'d\theta_{j}\otimes dw_{l}.
\label{fc.12}\end{equation}
However, there is an unambiguous way to say that a polyfibred cusp metric does not asymptotically involve such mixed terms.  

\begin{definition}
An \textbf{exact} polyfibred cusp metric is a polyfibred cusp metric $g_{\fc}$ such that for any $I\subset \{1,\ldots,\ell\}$ and for any $p\in \tH_I\setminus\pa \tH_I$,  
\begin{equation}
g_{\fc} =\sum_{i=1}^{k} \left( a_{i} \frac{d(x_{i}')^{2}}{(x_{i}')^{2}} + b_{i}(x_{i}')^{2} d\theta_{i}^{2} \right) +  \pr_I^* h_I + \sum_{i\in I} E_i
\label{ex.4b}\end{equation}
in local coordinates as in \eqref{fc.18}, where $a_{i}$ and $b_i$ are positive constants,  $E_i\in x_i\CI(\tX;{}^{\fc}T^*\tX\otimes {}^{\fc}T^*\tX)$, $h_{I}$ is a metric on $\bD_{I}\cap \cU$ (\ie in the coordinates $w$) and
$\pr_{I}: \tU \to \bD_I\cap\cU$ is the projection
\[
    (x_1',\theta_1,\ldots, x_k',\theta_k,w) \mapsto w.  
\] 
\label{ex.4}\end{definition}
For this definition to be sensible, we need a polyfibred cusp metric locally of the form \eqref{ex.4b} to stay of this form under changes of coordinates as in \eqref{fc.18}.  This is indeed the case.  
\begin{lemma}
Let 
$(\hat{\zeta}_1,\ldots,\hat{\zeta}_{k}, \hat{z}_{k+1},\ldots, \hat{z}_n)$ be another choice of coordinates near $\tbeta(p)\in\bX$ with $\left.\hat{\zeta}_{i}\right|_{\bD_i}=0$.  Let $(\hat{x}_1',\ldots,\hat{x}_{k}', \hat{\theta}_1,\ldots, \hat{\theta}_{k},\hat{w})$ be the corresponding real coordinates as in \eqref{fc.18}.   Then a metric $g_{\fc}$ of the form \eqref{ex.4b} in the coordinates $(\hat{x}',\hat{\theta}, \hat{w})$ is also of this form in the coordinates $(x', \theta,w)$.  
\label{excc.1}\end{lemma}
\begin{proof}
By Lemma~\ref{mwc.10}, we have $\hat{x}_{i}'= x_{i}'(1+ \overline{b}_{i}x_{i}'+ \widetilde{b}_{i}(x_{i}')^{2})$ for some $\bb_{i}\in \CI(\bX)$ and $\tb_{i}\in \CI(\tX)$.  Since $d\bb_{i}\in \CI(\bX;T^{*}\bX)\subset \CI(\tX;{}^{\fc}T^{*}\tX)$, a direct computation shows that 
\begin{equation}
      \frac{d\hat{x}_{i}'}{\hat{x}_{i}'} = \frac{dx_{i}'}{x_{i}'} + F_{i}, \quad F_{i}\in x_{i}\CI(\tX;{}^{\fc}T^*\tX).
\label{ex.6}\end{equation}
Since we have $\hat{\zeta_i}= \zeta_if_i(\zeta,z)$ for some local non-vanishing holomorphic function $f_i$, we have
\begin{equation}
\begin{aligned}
   \hat{\theta}_i &= \frac{1}{\sqrt{-1}} \log \left( \frac{\hat{\zeta}_i}{|\hat{\zeta}_i|}\right)=  \frac{1}{\sqrt{-1}} \log \left( \frac{\zeta_i}{|\zeta_{i}|}\right)+ \frac{1}{\sqrt{-1}} \log \left( \frac{f_i}{|f_i|}\right) \\
       & = \theta_i + \frac{1}{\sqrt{-1}} \log \left( \frac{f_i}{|f_i|}\right).
\end{aligned}
\label{ex.6b}\end{equation}   
Writing $\frac{f_i}{|f_i|}= a_0+ \rho_{i}a_1$ with $a_0$ and $a_1$ smooth local functions on $\bX$ with $a_0$ non-vanishing and independent of $\zeta_i, \overline{\zeta}_i$, we see that 
\begin{equation}
  d\hat{\theta}_i= d\theta_i + d\log a_0 + d\log \left( 1 +\frac{\rho_i a_1}{a_0}\right),
\label{ex.6c}\end{equation}
so that $\hat{x}_{i}'d\hat{\theta}_i = x_{i}'d\theta_i$ modulo $x_i\CI(\tX;{}^{\fc}T^{*}\tX)$.  Finally, from the change of coordinates $\hat{z}_k= g_k(\zeta,z)$ for a holomorphic function $g_k$, we see that
\begin{equation}
         d\hat{z}_k= \sum_{i=1}^{k} \frac{\pa g_k}{\pa \zeta_i} d\zeta_i + \sum_{j=k+1}^n \frac{\pa g_k}{\pa z_j} dz_j .
\label{ex.6d}\end{equation}
Since 
\begin{equation}
\begin{aligned}
d\zeta_{i} &= e^{\sqrt{-1}\theta_{i}} d\rho_{i}' + \sqrt{-1} \rho_{i}' e^{\sqrt{-1}\theta_{i}} d\theta_{i}\\
&=  \rho_{i}' e^{\sqrt{-1}\theta_{i}} \left(  \frac{dx_{i}'}{(x_{i}')^{2}}+ \sqrt{-1}d\theta_{i}, \right),
\end{aligned} 
\label{fc.23}\end{equation}
it follows from the fact that $\rho_{i}'= e^{-\frac{1}{x_{i}'}}$ decays exponentially fast as $x_{i}'$ tends to $0$ that 
\begin{equation}
    \hat{\pr}_I^{*}h_I = \pr_I^* h_I + \sum_{i\in I} G_i, \quad G_i \in x_i\CI(\tX;{}^{\fc}T^{*}\tX\otimes{}^{\fc}T^* \tX), 
\label{ex.6e}\end{equation}
where $\hat{\pr}_I$ is the projection $(\hat{x}',\hat{\theta},\hat{w})\mapsto \hat{w}$.  Thus, combining \eqref{ex.6}, \eqref{ex.6c} and \eqref{ex.6e}, the result follows.  
\end{proof}

As for polyfibred boundary metrics, there is a natural subspace of smooth vector fields associated to a polyfibred cusp metric $g_{\pfc}$,   
\begin{equation}
 \cV_{\pfc}(\tX)= \{ \xi\in \CI(\tX, T\tX) \quad | \quad 
      \exists c>0 \; \mbox{such that} \; g_{\pfc}(\xi_{p},\xi_{p})<c \; \forall p\in X\}. 
      \label{fc.12b}\end{equation}
Clearly, we have that
\begin{equation}
      \cV_{\pfc}(\tX) = \frac{1}{x} \cV_{\tPhi}(\tX),
\label{fc.12c}\end{equation}
so that the definition of $\cV_{\pfc}(\tX)$ does not depend on the choice of $g_{\pfc}$.  In contrast to $\cV_{\tPhi}(\tX)$, notice however that $\cV_{\pfc}(\tX)$ is not closed under the Lie bracket.  Nevertheless, since it is a $\CI(\tX)$-module, there exists a smooth vector bundle ${}^{\pfc}T\tX\to \tX$ and a natural map $\iota_{\pfc}: {}^{\pfc}T\tX\to T\tX$ restricting to an isomorphism on $\tX\setminus \pa\tX$ with the property that 
\begin{equation}
   \cV_{\pfc}(\tX)= \iota_{\pfc}( \CI(\tX; {}^{\pfc}T\tX)).
\label{fc.12d}\end{equation}
The fibre of ${}^{\pfc}T\tX$ above a point $p\in \tX$ is defined by
\begin{equation}
  {}^{\pfc}T_{p}\tX=  \cV_{\pfc}(\tX)/ \cI_{p}\cV_{\pfc}(\tX),
\label{fc.12e}\end{equation}
where $\cI_{p}\subset \CI(\tX)$ is the ideal of all smooth functions vanishing at $p$.  In terms of the bundle ${}^{\fc}T\tX$, a polyfibred cusp metric can be seen as a Euclidean structure for the 
real vector bundle ${}^{\fc}T\tX$.  

\begin{lemma}
If $g_{\pfc}$ and $g_{\pfc}'$ are two polyfibred cusp metrics, then there exists a constant $C>0$ such that  
\[
       \frac{g_{\pfc}}{C} \le g_{\pfc}' \le Cg_{\pfc},
\]
that is to say, $g_{\pfc}$ and $g_{\pfc}'$ are bi-Lipschitz equivalent.
\label{fc.12f}\end{lemma}
\begin{proof}  
Seen as Euclidean structures on the real vector bundle ${}^{\pfc}T\tX$, we can certainly find such a constant $C>0$ since $\tX$ is compact.  We get the corresponding result for the Riemannian metrics by restricting to $X= \tX\setminus \pa \tX$.  
\end{proof}
Polyfibred cusp metrics also satisfy the following properties.  
\begin{proposition}
Any polyfibred cusp metric $g_{\pfc}$ is complete and has finite volume.
\label{fv.1}\end{proposition}  
\begin{proof}
Let $g_{\pfc}$ be a polyfibred cusp metric.  For each $i\in\{1,\ldots,\ell\}$, we can find a small positive constant $c_{i}$ such that
\[
      g_{\pfc} > c^{2}_{i} \frac{dx_{i}}{x_{i}}\otimes \frac{dx_{i}}{x_{i}}
\]
seen as sections of $\CI(\tX; {}^{\pfc}T\tX\otimes {}^{\pfc}T\tX)$.  If $d_{g_{\pfc}}(\cdot,\cdot)$ denote the distance function on $X$ defined by $g_{\pfc}$, then this implies that for any points $p,q\in X$,
\[
  d_{g_{\pfc}}(p,q)\ge \left| \int_{x_{i}(p)}^{x_{i}(q)} c_{i}\frac{dx_{i}}{x_{i}}\right| = 
  c_{i} \left|  \log\left( \frac{x_{i}(q)}{x_{i}(p)}\right) \right|.
\]
Thus, for $p$ fixed, this distance tends to infinity as $q$ is approaching $\tH_{i}$.  More precisely, for any point $p\in X$ and any $r>0$, the ball 
\[
B_{d_{g_{\pfc}}}(p,r)= \{ q\in X \quad | \quad d_{g_{\pfc}}(p,q)\le r\}
\]
is a compact subset of $\tX\setminus \pa \tX$, which shows that the metric space $(X, d_{g_{\pfc}})$ is complete.  To see that the metric $g_{\pfc}$ has finite volume, it suffices to notice that in the local coordinates \eqref{fc.18}, the $x_{i}'$ factors of $\frac{dx_{i}'}{x_{i}'}$ and of $x_{i}'d\theta_{i}$ cancel, so that the volume form of $g_{\pfc}$ is canonically an element of $\CI(\tX; |\Omega(\tX)|)$, where 
$|\Omega(\tX)|$ is the density bundle of $T\tX$.  In particular, the integral over $\tX$, and hence, over $X$, is finite.

\end{proof}

There are natural spaces of functions associated to polyfibred cusp metrics.  To describe them, 
fix a polyfibred cusp metric $g_{\pfc}$ and let $\nabla$ be its Levi-Civita connection.  This induces a corresponding Euclidean metric and Euclidean connection (also denoted $g_{\pfc}$ and $\nabla$) on the vector bundle
\begin{equation}
   T^{r}_{s}X = \underbrace{TX\otimes \cdots \otimes TX}_{r \; \mbox{times}} \otimes \underbrace{T^{*}X\otimes \cdots \otimes T^{*}X}_{s \; \mbox{times}}
   \label{fs.1}\end{equation}
of $(r,s)$-tensors.  More generally, if $V\to X$ is a smooth Euclidean vector bundle with a Euclidean connection, then we have a corresponding Euclidean metric and Euclidean connection on the vector bundle $T^{r}_{s}X\otimes V$ for all $r,s\in \bbN_{0}$.  For such a vector bundle $V$ and for $k\in \bbN_0$, we denote by $\cC^{k}_{\pfc}(X;V)$ the space of continuous sections $f$ of $V$ such that $\nabla^{j}f\in \cC^{0}(X; T^{0}_{j}X\otimes V)$ with
\begin{equation}
    \sup_{p\in X} |\nabla^{j}f(p)| < \infty, \quad \forall j\in \{0,\ldots, k\},
\label{fs.2}\end{equation}
where $|\cdot |$ is the fibrewise norm given by the Euclidean metric on $T^{r}_{s}X\otimes V$.  The space $\cC^{k}_{\fc}(X;V)$ is naturally a Banach space with norm given by
\begin{equation}
      \| f\|_{k} = \sum_{j=0}^{k} \sup_{p\in X} |\nabla^{j}f(p)|.
\label{fs.3}\end{equation}
The intersection of these spaces for $k\in \bbN$, 
\begin{equation}
    \cC^{\infty}_{\pfc}(X;V) = \bigcap_{k\in \bbN} \cC^{k}_{\fc}(X;V),
\label{fs.4}\end{equation}
is a Fr\'echet space with semi-norms given by \eqref{fs.3} for $k\in \bbN$.  In the local coordinates \eqref{fc.18}, a function $f\in \CI_{\fc}(X)$ is a function such that
\begin{equation}
\sup_{\tU\cap X } \left| \prod_{i=1}^k \left( \left( x_i'\frac{\pa}{\pa x_i'} \right)^{\alpha_i}  \left( \frac{1}{x_i'} \frac{\pa}{\pa \theta_i}  \right)^{\beta_i} \right) \prod_{j=1}^{2q} \left(  \frac{\pa}{\pa w_j} \right)^{\gamma_j} f \right| <\infty, \quad \forall\; \alpha,\beta\in \bbN_0^{k}, \; \gamma\in \bbN_0^{2q}.
\label{fs.4b}\end{equation}  

When $V= T^{r}_{s}X$, it is tacitly assumed that the Euclidean metric and connection taken on $T^{r}_{s}X$ are those induced by $g_{\pfc}$ and its Levi-Civita connection.  By Lemma~\ref{fc.12f}, the definition of $\cC^{k}_{\fc}(X;T^{r}_{s}X)$ does not depend on the choice of the polyfibred cusp metric $g_{\pfc}$.  For the trivial real line bundle $T^{0}_{0}X= \underline{\bbR}$, we obtain a corresponding space of functions
\begin{equation}
        \cC^{k}_{\fc}(X)= \cC^{k}(X;T^{0}_{0}X).  
\label{fs.5}\end{equation}
In the notation of \cite{Wu06}, the Fr\'echet space $\CI_{\fc}(X)$ corresponds to the Cheng-Yau H\"older ring $\mathcal{R}(X)$.  For $k\in \bbN_0$ and $\alpha\in (0,1)$, one can also consider the H\"older space $\cC^{k,\alpha}_{\fc}(X)$ of \cite{Kobayashi} using good quasi-coordinates (see also \cite{Wu06}).

To study the asymptotic behavior of solutions to the Ricci flow, it is also useful to introduce the parabolic version of the space $\cC^{k}_{\fc}(X;V)$.  Denoting also by $V$ the pull-back of $V$ under the projection $[0,T]\times X\to X$, we define 
$\cC^{k,\frac{k}{2}}_{\fc}([0,T]\times X;V)$ to be the space of sections $f$ of $V$ such that for all
$j,l\in \bbN_0$ with $2j+l\le k$, 
\begin{equation}
  \pa_t^j \nabla^l f \in \cC^0([0,T]\times X;T^0_l X\otimes V) \quad \mbox{and} \quad 
     \sup_{t\in [0,T]} \sup_{p\in X} |\pa_t^j \nabla^l f (t,p)|< \infty.
\label{fs.5b}\end{equation} 
The space $\cC^{k,\frac{k}{2}}_{\fc}([0,T]\times X;V)$ is naturally a Banach space with norm given by 
\begin{equation}
  \| f\|_{k,\frac{k}{2}}= \sum_{2j+l\le k} \sup_{t\in [0,T]} \sup_{p\in X} | \pa_t^j \nabla^l f(t,p) |.
\label{fs.5c}\end{equation}
The intersection of these spaces, 
\begin{equation}
  \CI_{\fc}([0,T]\times X; V)= \bigcap_{k\in \bbN} \cC_{\fc}^{k,\frac{k}{2}}([0,T]\times X;V),
\label{fs.5d}\end{equation}
is a Fr\'echet space with semi-norms given by \eqref{fs.5c}.

Suppose now that the bundle $V\to X$ is the restriction of a smooth vector bundle $\tV\to \tX$.   An example to keep in mind is $V= T^{r}_{s}X$,  in which case we can take $\tV= {}^{\pfc}T^{r}_{s}X$.  Equip $\tV$ with a Euclidean metric and a Euclidean connection $\widetilde{\nabla}$.  Similarly, choose a Euclidean metric     
on $T\tX\to \tX$ and let also $\widetilde{\nabla}$ denote its Levi-Civita connection.  We can then define
$\cC^{k}(\tX;\tV)$ to be the space of continuous sections $f\in \cC^{0}(\tX;\tV)$ such that 
\begin{equation}
\widetilde{\nabla}^{j}f\in \cC^{0}(\tX;T^{0}_{j}\tX\otimes \tV), \quad \forall j\in \{0,\ldots,k\}.
\label{fs.6}\end{equation}
  The space $\cC^{k}(\tX;\tV)$ is a Banach space with norm given by
\begin{equation}
     \| f\|_{k} = \sum_{j=0}^{k} \sup_{p\in \tX} | \widetilde{\nabla}^{j}f(p)|,
\label{fs.7}\end{equation}  
where $|\cdot |$ is the pointwise norm of the corresponding Euclidean metric.  The space $\cC^k(\tX;\tV)$ also has a parabolic version $\cC^{k,\frac{k}{2}}([0,T]\times \tX;\tV)$ which consists of sections $f\in \cC^0([0,T]\times \tX;\tV)$ such that for all $2j+l\le k$, 
\begin{equation}
   \pa_t^j \widetilde{\nabla}^{l}f\in \cC^0([0,T]\times \tX; T^0_l \tX \otimes \tV).  
\label{fs.7b}\end{equation}

We are interested in the subspace given by
\begin{equation}
  \cC^{k}_{\fc}(\tX;\tV) = \cC^{k}(\tX;\tV)\cap \cC^{k}_{\fc}(X;V),  
\label{fs.8}\end{equation}
as well as its parabolic counterpart
\begin{equation}
 \cC^{k,\frac{k}{2}}_{\fc}([0,T]\times \tX;\tV) = \cC^{k,\frac{k}{2}}([0,T]\times \tX;\tV)\cap \cC^{k,\frac{k}{2}}_{\fc}([0,T]\times X;V). 
 \label{fs.8b}\end{equation}
They are  Banach spaces with norms respectively given by the sum of the norms of $\cC^{k}(\tX;\tV)$ and
$\cC^{k}_{\pfc}(X;V)$  and the sum of the norms of $\cC^{k,\frac{k}{2}}([0,T]\times \tX;\tV)$ and 
$\cC^{k,\frac{k}{2}}_{\fc}([0,T]\times X; V)$.  We have also the corresponding Fr\'echet spaces
\begin{equation}
  \CI_{\pfc}(\tX;\tV)= \bigcap_{k\in \bbN} \cC^{k}_{\pfc}(\tX;\tV) \quad\mbox{and} \quad 
    \CI_{\fc}([0,T]\times \tX;\tV)= \bigcap_{k\in\bbN} \cC^{k,\frac{k}{2}}_{\fc}(\tX;\tV).
\label{fs.9}\end{equation}

For the trivial line bundle $\tV=\underline{\bbR}$ over $\tX$, we will use the notation $\CI_{\fc}(\tX)= \CI_{\fc}(\tX;\underline{\bbR})$. A function $f\in \CI_{\fc}(\tX)$ is in particular in $\CI(\tX)$, so it has a Taylor series at the boundary hypersurface $\tH_i$.
However, the fact that the function $f$ is also in $\CI_{\fc}(X)$ imposes some restrictions on the coefficients of its Taylor series.  
\begin{proposition}
If $f\in\CI_{\fc}(\tX)$, then in the local coordinates \eqref{fc.18}, it has a Taylor series at $\tH_{i}$ of the form 
\begin{equation}
     f\sim \sum_{k=0}^{\infty} (\tPhi_{i}^{*}a_k)  (x_{i}')^{k}, \quad a_k\in \CI_{\fc}(\tD_i).  
\label{ext.1}\end{equation}
\label{ext.2}\end{proposition}
\begin{proof}
From the fact  that $f\in \CI(\tX)$, we know that  it has a Taylor series at $\tH_i$ of the form
\[
          f\sim \sum_{k=0}^{\infty} \frac{1}{k!} \left( \left.  \frac{\pa^{k} f}{\pa (x_i')^{k}}\right|_{\tH_i}\right) (x_i')^k.
\] 
Since $f$ is also in $\CI_{\fc}(X)$, we have uniform control on its derivatives with respect to the vector field $\frac{1}{x_i'} \frac{\pa}{\pa{\theta_i}}$.  This means that 
\[
  \left( \frac{1}{x_i'}\frac{\pa}{\pa \theta_i}\right)^{k+1}\frac{\pa^{k} f}{\pa (x_i')^{k}}=        \frac{1}{(x_i')^{k+1}}  \frac{\pa^{k+1}}{\pa \theta_i^{k+1}}\frac{\pa^{k} f}{\pa (x_i')^{k}}
\]
is uniformly bounded, which implies that 
\[
\left.  \frac{\pa^{k+1}}{\pa \theta_i^{k+1}}\frac{\pa^{k} f}{\pa (x_i')^{k}}\right|_{\tH_i}=0.
\]
Integrating $k+1$ times in $\theta_i$, we find that 
$\left. \frac{\pa^{k} f}{\pa (x_i')^{k}}\right|_{\tH_i}$ 
must be constant in $\theta_i$ in order to be globally defined on each circle of the fibration $\tPhi_i$. 
Thus, the coefficients of the Taylor series are constant on the fibres of the fibration $\tPhi_i$ and are therefore given by the pull-back of functions $a_k$ on $\tD_i$.  From the fact $f\in \CI_{\fc}(\tX)$, we then conclude that each $a_k$ must be in $\CI_{\fc}(\tD_i)$.    
\end{proof}
We infer from this proposition that there is a well-defined restriction map
\begin{equation}
      r_{i}: \CI_{\fc}(\tX)\to \CI_{\fc}(\tD_i) \quad \mbox{such that} \; \left. f\right|_{\tH_i}= \tPhi_i^* (r_i(f)).
\label{ext.3}\end{equation}
In fact, more generally, for $I\subset \{1,\ldots,\ell\}$ with $\tH_{I}\ne \emptyset$,  we have a well-defined restriction map
\begin{equation}
  r_{I}: \CI_{\fc}(\tX) \to \CI_{\fc}(\tD_{I}) \quad \mbox{such that} \;  \left. f\right|_{\tH_I}= \tPhi_I^* (r_I(f)).
 \label{ext.4}\end{equation}
Conversely, it is also possible to construct an extension map, that is, a right inverse for the restriction map $r_i$.  We will construct one by proceeding locally with the help of a partition of unity.  First, in $\bX$, cover $\bD_i$ by finitely many open sets $\cU_{1}, \ldots, \cU_{k_i}$ such that on each of them, we have holomorphic coordinates as in \eqref{fc.16}.  On $\bD_i$, let $\overline{\phi}_{q}\in \CI(\bD_i)$ be a partition of unity subordinate to the cover
$\overline{V}_{q}= \left. \cU_{q}\right|_{\bD_i}$.  We can lift these to obtain an open cover
$\tV_i= \tbeta_i^{-1}(\overline{V}_i)$ of $\tD_i$ with partition of unity $\tphi_q= \tbeta_i^* \overline{\phi}_q$.  Since each $\tphi_q$ is the pull-back of a smooth function on $\bD_i$, we have automatically that $\tphi_q\in \CI_{\fc}(\tD_i)$.  Similarly, the cover $\cU_q$ can be lifted to an open cover $\tcU_q$ of $\tH_i$ in $\tX$ with each open set having coordinates as in \eqref{fc.18}.    In particular, the fibration $\tPhi_i$ restricts to give a circle fibration
\begin{equation}
      \tPhi_i: \left. \tcU_q \right|_{\tH_i} \to \tV_i. 
\label{ext.5}\end{equation} 
There is also a natural projection 
\begin{equation}
   \pr_q: \tcU_q\to \left. \tcU_q \right|_{\tH_i}
\label{ext.5}\end{equation}
obtained by removing the coordinate $x_i'$.  Finally, for $\epsilon_i>0$, let $\chi_i\in \CI_c([0,\infty))$ be a cut-off function such that $\chi_i(t)=1$ for $t<\frac{\epsilon_i}{2}$ and $\chi_i(t)=0$ for $t>\epsilon_i$.  Provided we choose $\epsilon_i$ small enough, we can insure that 
\begin{equation}
  \chi_i(x_i) \pr_q^{*}\tPhi_i^{*}\tphi_q \in \CI_{c}(\tU_q), \quad \forall q.   
\label{ext.6}\end{equation}  
With all this data, we can define an extension map 
\begin{equation}
     \Xi_i: \CI_{\fc}(\tD_i)\to \CI_{\fc}(\tX), \quad\mbox{by} \quad \Xi_i(f)= \sum_q \chi_i(x_i) \pr_q^{*}\tPhi_i^* (\tphi_q f). 
\label{ext.7}\end{equation}

Proposition~\ref{ext.2} has also some implications on the restrictions to $\tH_I$ of the derivatives of a function $f$ in $\CI_{\fc}(\tX)$.  We will in particular be interested in the restriction of $\pa\db f$ to each of the boundary faces of $\tX$.

\begin{proposition}
Given $f\in \CI_{\fc}(\tX)$, the restriction of $\pa\db f\in \CI(\tX; \Lambda^{2}({}^{\fc}T^*\tX)\otimes \bbC)$ to $\tH_{I}$, seen as section of $\Lambda^{2}({}^{\fc}T^*\tX)\otimes \bbC$, is
\[
              \tPhi_I^*(\pa\db r_I(f)) \in 
              \CI(\tH_I; \left. \Lambda^{2}({}^{\fc}T^*\tX)\otimes \bbC\right|_{\tH_I}).
\] 
\label{ext.8}\end{proposition}   
\begin{proof}
Let $p\in \bD_I$ be given.  Relabelling the boundary hypersurfaces if necessary, we can assume $I=\{1,\ldots,k\}$.  Let $(\zeta_1,\ldots, \zeta_k, z_{k+1}, \ldots, z_n)$ be holomorphic coordinates near $p$ as in \eqref{fc.16}.    In the corresponding real coordinates \eqref{fc.18} on $\tX$, we have,
\begin{equation}
  \frac{\pa}{\pa \zeta_i}= \frac{1}{2} \left(    e^{-\sqrt{-1}\theta_i} \frac{\pa}{\pa\rho_i'} - \frac{\sqrt{-1}e^{-\sqrt{-1}\theta_i}}{\rho_i'} \frac{\pa}{\pa \theta_i}  \right) 
  =   \frac{x_i'}{2\zeta_i} \left( x_i' \frac{\pa}{\pa x_i'} - \frac{\sqrt{-1}}{x_i'} \frac{\pa}{\pa \theta_i}\right).
\label{ext.9}\end{equation}
Thus, for a function $f\in \CI_{\fc}(\tX)$, we have
\begin{equation}
    \frac{\pa f}{\pa \zeta_i} d\zeta_i= \left( x_i' \frac{\pa f}{\pa x_i'} - \frac{\sqrt{-1}}{x_i'} \frac{\pa f}{\pa \theta_i}\right)  \frac{x_i'd\zeta_i}{2\zeta_i}. 
\label{ext.10}\end{equation}
Now, from Proposition~\ref{ext.2}, we see that 
\[
     \left. \left( x_i' \frac{\pa f}{\pa x_i'} - \frac{\sqrt{-1}}{x_i'} \frac{\pa f}{\pa \theta_i}\right)\right|_{\tH_{I}}=0.
     \]
Since $ \frac{x_i'd\zeta_i}{\zeta_i}= \frac{dx_i'}{x_i'} + \sqrt{-1} x_i' d\theta_i$ is a local smooth non-vanishing section of ${}^{\fc}T^{*}\tX$, this shows that, as a section of ${}^{\fc}T^*\tX$, 
\[
         \left.  \frac{\pa f}{\pa \zeta_i} d\zeta_i \right|_{\tH_I}=0.
\]
Similarly, we have that $ \left.  \frac{\pa f}{\pa \overline{\zeta}_i} d\overline{\zeta}_i \right|_{\tH_I}=0.$  Repeating this argument, we find that as local sections of $\Lambda^{2}( {}^{\fc}T^{*}\tX)\otimes \bbC$, the forms
\[
   \frac{\pa^2 f}{\pa \zeta_i\pa \overline{\zeta}_j} d\zeta_i\wedge d\overline{\zeta}_{j}, \;
    \frac{\pa^2 f}{\pa \zeta_i\pa \overline{z}_l} d\zeta_i\wedge d\overline{z}_{l} \; \mbox{and} \;
 \frac{\pa^2 f}{\pa z_l\pa \overline{\zeta}_j} dz_l\wedge d\overline{\zeta}_{j}
 \]
 all restrict to zero on $\tH_{I}$.  Thus, the restriction of $\pa\db f$ to $\tH_{I}$ is given locally by the restriction of 
 \[
   \sum_{i,j}  \frac{\pa^2 f}{\pa z_i\pa \overline{z}_j} dz_i\wedge d\overline{z}_{j},
 \]
 from which the result follows.  
\end{proof}

For certain applications, notably the study of the asymptotic behavior of K\"ahler-Einstein metrics on quasiprojective manifolds, the space of functions $\CI_{\fc}(\tX)$ is too small and need to be enlarged to a suitable space of polyhomogeneous functions.  To introduce this space, recall from \cite{Melrose1992} that an \textbf{index set} is a countable discrete subset $E$ of $\bbC\times \bbN_0$ such that 
\begin{gather}
   \label{phg.1a} (z_j, k_j)\in E, \quad |(z_j,k_j)|\to \infty \quad \Longrightarrow \quad \Re z_j \to \infty, \\ 
 \label{phg.1b} (z,k)\in E \quad \Longrightarrow \quad (z+p, k)\in E \; \forall p\in \bbN, \; \mbox{and}  \\
 \label{phg.1c}  (z,k)\in E \quad \Longrightarrow \quad (z,p)\in E \; \forall p\in \bbN_0 \; \mbox{with}\; p\le k.   
\end{gather} 
For us, an index set will prescribe the type of asymptotic behavior a function should have near a boundary hypersurface.  Suppose first that $\bD$ is smooth so that $\tX$ is a smooth manifold with boundary.  Suppose also that $\bD$ has only one irreducible component $\bD_1$, so that $\pa\tX$ is connected.  Then given an index set $E$, recall from \cite{MelroseAPS} that we define the associated space $\cA^{E}_{\phg}(\tX)$ of polyhomogeneous functions to consist of functions $f\in \CI(X)$ having an asymptotic expansion near $\pa\tX$ of the form
\begin{equation}
  f\sim \sum_{(z,k)\in E} a_{(z,k)} x_1^{z} (\log x_1)^{k}, \quad a_{(z,k)} \in \CI(\tX),
\label{phg.2}\end{equation} 
where the symbol $\sim$ means that for all $N\in\bbN$, we have that 
\begin{equation}
   f-  \underset{\Re z\le N}{\sum_{(z,k)\in E}} a_{(z,k)} x_1^{z}(\log x_1)^{k} \in \dot{\cC}^{N}(\tX),
\label{phg.3}\end{equation}
where $\dot{\cC}^{N}(\tX)$ is the space of functions which are $N$ times differentiable on $\tX$ with all their derivatives vanishing up to order $N$ along $\pa\tX$.  An important example is when $E=\bbN_{0}\times\{0\}$, in which case $\cA^{E}_{\phg}(\tX)$ is simply the space $\CI(\tX)$.  Another important example is when $E=\emptyset$, in which case $\cA^{E}_{\phg}(\tX)= \dot{\cC}^{\infty}(\tX)$ is the space of smooth functions on $\tX$ vanishing on $\pa \tX$ together with all their derivatives.   

In our context,  we will particularly be interested in the subspace of polyhomogeneous functions $\cA^{E}_{\fc}(\tX)$ consisting of smooth functions having an asymptotic expansion near $\tX$ of the form
\begin{equation}
 f\sim \sum_{(z,k)\in E} a_{(z,k)} x_1^{z}(\log x_1)^{k}, \quad a_{(z,k)} \in \CI_{\fc}(\tX),
\label{phg.5}\end{equation} 
which means that for all $N\in \bbN$, we have that 
\begin{equation}
   f-  \underset{\Re z\le N}{\sum_{(z,k)\in E}} a_{(z,k)} x_1^{z}(\log x_1)^{k} \in x^{N}\cC^{N}_{\fc}(X).
\label{phg.6}\end{equation}    

When $\bD$ is not smooth, so that $\tX$ is a manifold with corners, the space of polyhomogeneous functions $\cA_{\fc}^{E}(\tX)$ can be generalized as follows.  First, recall from \cite{Melrose1992} that an \textbf{index family} $\cE=(E_1 ,\ldots, E_\ell )$ for the compact manifold with corners $\tX$ is the assignment of an index set $E_i$ to each boundary hypersurface $\tH_i$ of $\tX$.    Notice that if $\cE$ is an index family for $\tX$, then it naturally induces an index family $\cE_i$ for $\tD_i$ which to 
a non-empty boundary hypersurface of the form $\tPhi_i( \tH_i \cap \tH_j)\subset \tD_i$ associate the index set $E_j$.  Thus, proceeding by induction on the  depth of $\tX$, we can define the space $\cA^{\cE}_{\fc}(\tX)$ associated to the index family $\cE=(E_{1},\ldots, E_{\ell})$  to consist of functions $f$ such that for all $i\in \{1,\ldots, \ell\}$, 
\begin{equation}
    f\sim \sum_{(z,k)\in E_i} \Xi_i ( a_{(z,k)}) x_{i}^{z} (\log x_i )^{k}, \quad a_{(k,z)}\in \cA_{\fc}^{\cE_i} (\tD_i),
\label{phg.7}\end{equation}
 which means that for all $N\in \bbN$, 
\begin{equation}
  f-  \underset{\Re z\le N}{\sum_{(z,k)\in E_i}} \Xi_i ( a_{(z,k)} ) x_i^z (\log x_i)^{k} \in x_i^N \cC^{N}_{\fc}(X).  
\label{phg.8}\end{equation}
Here, the map $\Xi_i$ is defined as in \eqref{ext.7} by 
\begin{equation}
  \Xi_i (a_{(z,k)}) = \sum_q \chi_i(x_i) \pr_q^* \tPhi_i^* (\tphi_q a_{(z,k)}).
\label{phg.9}\end{equation}
In that more general sense, notice that this defines a map 
\begin{equation}
  \Xi_i : \cA^{\cE_i}_{\fc}(\tD_i) \to \cA^{\cF_i}_{\fc}(\tX), \quad \cF_i=(F_1,\ldots, F_\ell ) \quad \mbox{with} \;
    F_j= \left\{ \begin{array}{ll}
           E_j, & j \ne i, \\
           \bbN_0\times \{0\}, & j=i.
        \end{array} \right.   
\label{phg.10}\end{equation}

Notice that thanks to property~\eqref{phg.1b} of index sets, the space $\cA^{\cE}_{\fc}(\tX)$ is naturally a $\CI_{\fc}(\tX)$-module.  When $\tV\to \tX$ is a smooth vector bundle on $\tX$ for which the space $\CI_{\fc}(\tX;\tV)$ has been defined, this means we can more generally define the space of polyhomogeneous sections of $\tV$ associated to the index family $\cE$ by
\begin{equation}
   \cA_{\fc}^{\cE}(\tX;\tV)= \cA^{\cE}_{\fc}(\tX) \otimes_{\CI_{\fc}(\tX)} \CI_{\fc}(\tX;\tV).  
\label{phg.11}\end{equation}

We are interested in polyfibred cusp metrics that are K\"ahler with respect to the complex structure of $X$.  Examples of such metrics are not hard to construct.  Indeed, let $L\to \bX$ be a positive holomorphic Hermitian line bundle with Hermitian metric $\|\cdot\|_{L}$.  Locally, the curvature of its Chern connection is given by 
\begin{equation}
    \Theta_{L}= \db\pa \log \| \lambda \|^{2}_{L}
\label{fc.13}\end{equation}  
where $\lambda$ is any local holomorphic section of $L$.  Since $L$ is positive, the curvature form $\sqrt{-1}\Theta_{L}$ is a K\"ahler form on $\bX$.  To obtain a K\"ahler form on the quasiprojective manifold $X$, we consider instead for $\epsilon>0$ the closed $2$-form
\begin{equation}
\begin{aligned}
\omega &=\sqrt{-1} \Theta_{L} + \sqrt{-1}\ \db\pa \log\left( \prod_{i=1}^{\ell} (-\log \epsilon \|s_{i}\|^{2}_{\bD_{i}})^{2} \right) \\
&= \sqrt{-1}\Theta_{L}+ 2\sqrt{-1}\sum_{i=1}^{\ell} \left( \frac{ (\pa \log \epsilon \|s_{i}\|^{2}_{\bD_{i}})\wedge (\db \log \epsilon \|s_{i}\|^{2}_{\bD_{i}}) }{( \log \epsilon \|s_{i}\|^{2}_{\bD_{i}})^{2}} + \frac{ \db\pa \log \epsilon \|s_{i}\|_{\bD_{i}}^{2}}{\log\epsilon \|s_{i}\|_{\bD_{i}}^{2} }
 \right) \\
 &= \sqrt{-1}\Theta_{L} + 2\sqrt{-1}\sum_{i=1}^{\ell} \left(\frac{ \Theta_{\bD_{i}}}{\log\epsilon \|s_{i}\|_{\bD_{i}}^{2} }\right)  \\
 & \quad \quad \quad + 2\sqrt{-1}\sum_{i=1}^{\ell} \left( \frac{ (\pa \log \epsilon \|s_{i}\|^{2}_{\bD_{i}})\wedge (\db \log \epsilon \|s_{i}\|^{2}_{\bD_{i}}) }{( \log \epsilon \|s_{i}\|^{2}_{\bD_{i}})^{2}} 
 \right).\end{aligned}
\label{fc.14}\end{equation}
Since the middle term can be made arbitrarily small by taking $\epsilon>0$ sufficiently small, we see that the form $\omega$ is a K\"ahler form provided $\epsilon$ is chosen small enough.

\begin{proposition}
The K\"ahler metric $g_{\omega}$ associated to the K\"ahler form $\omega$ in \eqref{fc.14} is an exact polyfibred cusp metric on $X$ with $g_{\omega}\in \CI_{\fc}(\tX; {}^{\fc}T^{*}\tX\otimes {}^{\fc}T^{*}\tX)$.
\label{fc.15}\end{proposition}  
\begin{proof}
Away from $\pa \tX$, there is nothing to prove.  Thus, let $p\in \pa \tX$ be given.  Near $\tbeta(p)$, choose complex coordinates $(\zeta_{1},\ldots,\zeta_{k},z_{k+1},\ldots,z_{n})$ as in \eqref{fc.16} and let  $(x_{1}',\theta_{1},\ldots, x_{k}',\theta_{k}, u_{k+1}, v_{k+1}, \ldots,u_{n},v_{n})$ be the corresponding coordinates on $\tX$ as in \eqref{fc.18}.  
In these coordinates, we see by direct computation that multiplication by $\sqrt{-1}$ sends $x_{i}' \frac{\pa}{\pa x_{i}'}$ to $\frac{1}{x_{i}'}\frac{\pa}{\pa \theta_{i}}$ and 
$\frac{\pa}{\pa u_{j}}$ to $\frac{\pa}{\pa v_{j}}$.  Since $p\in \pa \tX$ is arbitrary, this shows that complex multiplication induces a smooth map
\begin{equation}
  \sqrt{-1}: \CI_{\fc}(\tX; {}^{\fc}T\tX\otimes \bbC) \to \CI_{\fc}(\tX; {}^{\fc}T\tX\otimes \bbC).
\label{fc.19}\end{equation}  
Thus, to show $g_{\omega}$ is a polyfibred cusp metric with $g_{\omega}\in \CI_{\fc}(\tX;{}^{\fc}T^{*}\tX\otimes{}^{\fc}T^{*}\tX)$, it suffices to show that 
\begin{equation}
   \omega\in \CI(\tX; \Lambda^{2}({}^{\fc}T^{*}\tX)\otimes \bbC)
\label{fc.20}\end{equation}
and that the restriction of $\omega$ to $\pa \tX$ as a section
of $  \Lambda^{2}({}^{\fc}T^{*}\tX)\otimes \bbC$ is positive definite.  This can be checked by writing \eqref{fc.14} in the local holomorphic coordinates \eqref{fc.16}.  Indeed, notice first that 
\begin{equation}
     \|s_{i}\|^{2}_{\bD_{i}}= h_{i} |\zeta_{i}|^{2}
\label{fc.21}\end{equation}
for some positive smooth function $h_{i}$ in the holomorphic coordinates \eqref{fc.16}.  Thus,
\begin{equation}
   \frac{1}{\log\epsilon\|s_{i}\|^{2}_{\bD_{i}} }= \frac{1}{\log\epsilon + \log h_{i} + \log |\zeta_{i}|^{2} }  = \frac{x_{i}'}{x_{i}'(\log\epsilon + \log h_{i}) -2}
\label{fc.22}\end{equation}
is clearly smooth as a function on $\tX$.  Since $\Theta_{L}$ and $\Theta_{\bD_{i}}$ are smooth forms on $\bX$, it follows from \eqref{fc.23} and the fact that $\rho_{i}'= e^{-\frac{1}{x_{i}'}}$ decays exponentially fast as $x_{i}'$ tends to $0$ that the first two terms in \eqref{fc.14} are in
$\CI_{\fc}(\tX; \Lambda^{2}( {}^{\fc}T^{*}\tX)\otimes \bbC)$.  For the last term in
\eqref{fc.14}, notice that 
\begin{equation}
\begin{aligned}
  \pa \log \epsilon \|s_{i}\|^{2}_{\bD_{i}} &= \pa \log h_{i} + \pa \log |\zeta_{i}|^{2} \\
    &= \pa \log h_{i} + \frac{d\zeta_{i}}{\zeta_{i}} = \pa \log h_{i} + \frac{d\rho_{i}'}{\rho_{i}'} + \sqrt{-1}d\theta_{i}.
    \end{aligned}   
\label{fc.24}\end{equation}
Since $\pa \log h_{i}$ is a smooth form on $\bX$, we see from \eqref{fc.23} and \eqref{fc.24} that the last term of \eqref{fc.14} is also in $\CI_{\fc}(\tX; \Lambda^{2}( {}^{\fc}T^{*}\tX)\otimes \bbC)$, so that \eqref{fc.20} holds.  On the boundary face $\tH_{I}$ near the point $p$, the restriction of $\omega$, as a section of
$\Lambda^{2}({}^{\fc}T^{*}\tX)\otimes \bbC$ is given by
\begin{multline}
    \sqrt{-1} \tPhi^{*}_{I} \iota^{*}_{D_{I}}\left(\Theta_{L} +2 \sum_{i\notin I} \left(\frac{\Theta_{\bD_i}}{\log \epsilon \|s_i\|_{\bD_i}}  + \frac{ (\pa \log \epsilon \|s_{i}\|^{2}_{\bD_{i}})\wedge (\db \log \epsilon \|s_{i}\|^{2}_{\bD_{i}}) }{( \log \epsilon \|s_{i}\|^{2}_{\bD_{i}})^{2}}\right) \right)  \\
    + 
      2\sqrt{-1}\sum_{i\in I} \left(\frac{x_{i}'}{2}\right)^{2}\frac{d\zeta_{i}\wedge d\overline{\zeta}_{i}}{|\zeta_{i}|^{2}},
\label{fc.25}\end{multline} 
where $\iota_{D_{I}}: D_{I}\hookrightarrow X$ is the natural inclusion 
of $D_{I}= \bigcap_{i\in I} D_{i}$ in $X$.  In particular, this restriction is clearly positive definite, which shows that $g_{\omega}$ is a polyfibred cusp metric.  It is also clear from \eqref{fc.25} that $g_\omega$ is an exact polyfibred cusp metric.     
\end{proof}

This proposition suggests the following reformulation of the notion of standard spatial asymptotics introduced in \cite{Lott-Zhang}.  As in \cite{Lott-Zhang}, to lighten the presentation, we will often not distinguish between a K\"ahler metric $g_{\omega}$ and its K\"ahler form $\omega$.    

\begin{definition}
Let $\{\omega_{I}\}$ be exact polyfibred cusp K\"ahler metrics on $\{\tD_{I}\}$  for $I\subset \{1,\ldots,\ell\}$ such that $\bigcap_{i\in I} \bD_{i}\ne \emptyset$ and let $\{c_{i}\}_{i=1}^{\ell}$ be positive numbers.  Then an exact polyfibred cusp K\"ahler metric $\omega$ has \textbf{standard spatial asymptotics} associated to $\{\omega_{I}\}$ and $\{c_{i}\}_{i=1}^{\ell}$ if for every $p\in \pa \tX$, the restriction of $\omega$ at $p$ seen as a section of $\CI(\tX;\Lambda^{2}({}^{\fc}T^{*}\tX)\otimes \bbC)$   is given by
\[
     \tPhi_{I}^{*} \omega_{I}+ \frac{\sqrt{-1}}{2} \sum_{i\in I} \frac{ c_i d\zeta_{i}\wedge d\overline{\zeta}_{i}}{|\zeta_{i}|^{2} (\log|\zeta_{i}|)^{2}},
\] 
where $I=\{i_{1}, \ldots,i_{k}\}\subset\{1,\ldots,\ell\}$ is the largest subset such that $p$ is contained in $\tH_{I}= \bigcap_{i\in I} \tH_{i}$ and $(\zeta_{i_{1}},\ldots,\zeta_{i_{k}},z_{k+1}, \ldots, z_{n})$ are holomorphic coordinates near $\widetilde{\beta}(p)\in \bX$ as in \eqref{fc.16}.  
\label{fc.26}\end{definition}

The other property of the K\"ahler metric of Proposition~\ref{fc.15}, namely that $g_{\omega}\in \CI_{\fc}(\tX;{}^{\fc}T^{*}\tX\otimes{}^{\fc}T^{*}\tX)$, will play a predominant role in our study of the 
  K\"ahler-Ricci flow. 

\begin{definition}
An \textbf{asymptotically tame} polyfibred cusp metric on the quasiprojective manifold $X$ is an exact polyfibred cusp metric $g_{\fc}$ such that 
\[g_{\fc}\in \CI_{\fc}(\tX; {}^{\fc}T^{*}\tX\otimes {}^{\fc}T^{*}\tX).
\]
\label{fc.27}\end{definition}  
Thus, the K\"ahler metric of Proposition~\ref{fc.15} is an example of asymptotically tame polyfibred cusp K\"ahler metric.  

We conclude this section by describing the asymptotic behavior of the Ricci form of an asymptotically tame polyfibred cusp K\"ahler metric.  

\begin{proposition}
Let $\omega$ be an asymptotically tame polyfibred cusp K\"ahler metric with standard spatial asymptotic associated to $\{\omega_I\}$ and $\{c_i\}$.  Then its Ricci form $\varrho$ is an element of $\CI_{\fc}(\tX; \Lambda^{2}({}^{\fc}T^{*}\tX)\otimes \bbC)$.  Furthermore, in the holomorphic coordinates \eqref{fc.16}, the restriction of $\varrho$ at $\tH_I$ seen as a section of $\CI(\tX;\Lambda^{2}({}^{\fc}T^{*}\tX)\otimes \bbC)$ is given by 
\[
       \tPhi_I^{*} \varrho_I -\frac{\sqrt{-1}}{2} \sum_{i\in I} \frac{d\zeta_i \wedge d\overline{\zeta}_{i}  }{|\zeta_i |^{2} (\log |\zeta_i |)^{2}}
\]
where $\varrho_I$ is the Ricci form of $\omega_I$.  
\label{fc.28}\end{proposition}
\begin{proof}
Let us work in the local holomorphic coordinates $(\zeta_1,\ldots,\zeta_k, z_{k+1},\ldots,z_n)$ of \eqref{fc.16}.  In these coordinates, we have a standard asymptotically tame polyfibred cusp K\"ahler metric,
\begin{equation}
 \omega_{\st}= \frac{\sqrt{-1}}{2} \sum_{i=1}^{k}  \frac{d\zeta_i \wedge d\overline{\zeta}_{i}  }{|\zeta_i |^{2} (\log |\zeta_i |)^{2}} + \frac{\sqrt{-1}}{2}
   \sum_{j=k+1}^{n} dz_j \wedge d\overline{z}_j, 
\label{fc.29}\end{equation}
with Ricci form
\begin{equation}
  \varrho_{\st}= \sum_{i=1}^{k} \sqrt{-1} \pa \db \log \left( |\zeta_i|^{2} (\log |\zeta_i |)^{2} \right) = - \frac{\sqrt{-1}}{2} \sum_{i=1}^{k}  \frac{d\zeta_i \wedge d\overline{\zeta}_{i}  }{|\zeta_i |^{2} (\log |\zeta_i |)^{2}}.
\label{fc.30}\end{equation}
Then the Ricci form of $\omega$ is given by 
\begin{equation}
  \varrho= \varrho_{\st} - \sqrt{-1} \pa \db \log \left( \frac{ \omega^{n}}{\omega_{\st}^{n}} \right).
\label{fc.31}\end{equation}
Since $\frac{\omega^n}{\omega_{\st}^n}$ is locally in $\CI_{\fc}(\tX)$ and does not vanish, we see that $f=\log \left( \frac{ \omega^{n}}{\omega_{\st}^{n}} \right)$ is locally an element of $\CI_{\fc}(\tX)$.  Furthermore, a straightforward computation shows that the restriction of $f$ to $\tH_I$ is given by
\begin{equation}
     \left. f\right|_{\tH_I} = \log \left( \frac{\omega^{n-k}_{I}}{ \left( \frac{\sqrt{-1}}{2} \sum_{j=k+1}^{n}  dz_j\wedge d\overline{z}_j   \right)^{n-k}} \right) + \sum_{i=1}^{k} \log c_i.
\label{fc.32}\end{equation}
 The result then follows from \eqref{fc.31}, \eqref{fc.32} and Proposition~\ref{ext.8}.
 
\end{proof}

\begin{remark}

Actually, the Riemannian curvature tensor of 
the above metric has a similar splitting form 
at $\tH_I$. The easiest way to see that is 
to introduce the local (in $\theta_i$) holomorphic 
coordinates $\{\log(\log\zeta_1), \cdots, \log
(\log\zeta_k), z_{k+1}, \cdots, z_{n}\}$, 
where $\log\zeta_i=-(x'_i)^{-1}+\sqrt{-1}
\theta_i$ and $(\zeta,z)$ are holomorphic coordinates as in \eqref{fc.16}. The corresponding holomorphic 
vector fields are bounded with respect to 
the metric. Hence one can make use of the 
standard form of the Riemannian curvature 
tensor for a K\"ahler metric (in Section 1.2 
of \cite{tiannotes} for example) to obtain 
the splitting. This coordinate system 
represents the bounded geometry for this 
metric as discussed in \cite{Kobayashi} (see also \cite{Wu06}). The classic Cheng-Yau's 
function spaces considered there are 
equivalent to the space $C^k_{\fc}(X)$ considered
here.  

\label{Rm_asymp} \end{remark}

\section{Decay estimates for linear uniformly  parabolic equations} \label{de.0}

A natural class of differential operators of order $m$ acting on $\cC^{\infty}_{\fc}(X;V)$ is given by the space $\Diff^{m}_{\fc}(X;V)$ of differential operators $P$ of the form
\begin{equation}
  Pf=  \sum_{j=0}^{m} a_{j} \cdot \nabla^{j}f, \quad a_j \in \cC^{\infty}_{\fc}(X;T^{j}_{0}X\otimes \End(V)),  \quad f\in \cC^{\infty}_{\fc}(X;V),
\label{fs.10}\end{equation}
where $a_{j}\cdot \nabla^{j}f$ is seen as an element of $\CI_{\fc}(X;V)$ after the natural contractions are performed and $\nabla$ is the Levi-Civita connection of some polyfibred cusp metric.  We can also consider the subspace $\Diff^{m}_{\fc}(\tX;\tV)\subset \Diff^{m}_{\fc}(X;V)$ of operators which can be written as in \eqref{fs.10}, but with
$a_{j}\in \CI_{\fc}(\tX;{}^{\pfc}T^{j}_{0}\tX\otimes \End(\tV))$ for all $j\in \{0,\ldots,m\}$.

\begin{definition}
An operator $L\in\Diff^{2}_{\fc}(X;V)$ is said to be \textbf{uniformly  elliptic} if there exist a positive constant $c$ and a polyfibred cusp metric $g_{\fc}$  such that for all $p \in X$,
\[
     |\det(\sigma_{2}(L)(\xi,\xi))| >c |\xi|^{2\rank(V)}_{g_{\fc}} \quad \forall \, \xi\in T_{p}^{*}X\setminus\{0\},
\]
where $\sigma_{2}(L)\in \CI(X; TX\otimes TX\otimes \End(V))$ is the principal symbol of $L$ and $\rank(V)$ is the dimension of the fibres of $V$ .  In particular, an operator $L\in \Diff^{2}_{\fc}(X)$ with negative principal symbol (\eg the negative Laplacian) is uniformly elliptic if there exists a positive constant $c$ such that 
\[
        \sigma_{2}(L)< -cg^{*}_{\fc},
\] 
where $g^{*}_{\fc}$ is the metric dual to $g_{\fc}$ on the cotangent bundle.  \label{de.1}\end{definition}

\begin{definition}
 For $t\in [0,T]$, let $t\mapsto  L_{t}\in \Diff^{2}_{\fc}(X;V)$ be a smooth family of operators.  Then the operator 
 \[
       \frac{\pa}{\pa t} - L_{t}
 \]
 acting on $\CI_{\fc}([0,T]\times X)$ is said to be \textbf{uniformly  parabolic} if there exists a constant $c>0$ and a polyfibred cusp metric $g_{\fc}$ such that for all $p\in X$,
 \[
    |\det(\sigma_{2}(L_{t})(\xi,\xi))|> c |\xi|^{2\rank(V)}_{g_{\fc}}, \quad \forall\; t\in[0,T], \quad \forall\; \xi \in T^{*}_{p}X\setminus \{0\}.
 \]
 Alternatively, we will say the family $L_{t}$ is \textbf{uniformly elliptic}.  
 \label{de.2}\end{definition}
 
To study the asymptotic behavior of solutions to a uniformly parabolic equation, we need some preparation.  

\begin{proposition}
For $t\in [0,T]$, let $t\mapsto L_{t}\in \Diff^{2}_{\fc}(X)$ be a smooth family of operators with negative principal symbols such that
$\pa_{t}- L_{t}$ is uniformly  parabolic.  Suppose that $u\in \cC^{2,1}_{\fc}([0,T]\times X)$ is  a solution to the initial value problem
\[
    \frac{\pa u}{\pa t} - L_{t}u \le E, \quad u(0,\cdot)= u_{0}(\cdot),
\]
with $u_{0}\in x^{\gamma}\cC^{2}_{\fc}(X)$ and $E\in x^{\gamma}\cC_{\fc}^{0,0}([0,T]\times X)$, where $x^{\gamma} = \prod_{i=1}^{\ell} x_i^{\gamma_i}$ and $\gamma=(\gamma_1,\ldots,\gamma_\ell)$ with $\gamma_i\ge 0$. Then there exist positive constants $K$ and $c$ such that for all $t\in [0,T]$,
\[
     u\le  Ke^{ct} x^{\gamma}. 
\]
\label{de.4}\end{proposition}
\begin{proof}
The result will follow by applying the maximum principle.  Consider the new function 
\[
\psi= xu,
\]
where we recall that $x=\prod_{i=1}^{\ell} x_{i}$.  Its evolution equation is given by
\begin{equation}
  \frac{\pa\psi}{\pa t}- \widetilde{L}_{t} \psi \le xE, \quad \psi(0,\cdot)= xu_{0}(\cdot),
\label{de.5}\end{equation}
where $t\mapsto \widetilde{L}_{t}\in \Diff^{2}_{\fc}(X)$ is the smooth family of operators defined by 
\[
        \widetilde{L}_{t}= x\circ L_{t}\circ x^{-1}.  
\]
Since $\sigma_{2}(\widetilde{L}_{t})= \sigma_{2}(L_{t})$, the family $\widetilde{L}_{t}$ is also uniformly elliptic.  For some positive constants $K$ and $c$, consider then the barrier function
\begin{equation}
    v= Ke^{ct} x x^{\gamma}.
\label{de.6}\end{equation}
Provided the constants $K$ and $c$ are large enough, we will have that for all
$t\in [0,T]$, 
\begin{equation}
  \frac{\pa v}{\pa t} \ge \widetilde{L}_{t} v + x E,  \quad v(0,\cdot)\ge \psi(0,\cdot).
\label{de.7}\end{equation}
This means that 
\begin{equation}
     \frac{\pa}{\pa t}(v-\psi) \ge \widetilde{L}_{t}(v-\psi),  \quad (v-\psi)(0,\cdot)\ge 0.
\label{de.8}\end{equation}
Since, thanks to the factor $x$, $(v-\psi)$ tends to zero as one approaches $\pa\tX$, we can apply the maximum principle to \eqref{de.8} to conclude that 
\[
              \psi\le v = Ke^{ct} x x^{\gamma},
\]
from which the result follows.

\end{proof}

\begin{corollary}
For $t\in [0,T]$, let $t\mapsto L_{t}\in \Diff^{2}_{\fc}(X)$ be a smooth family of operators with negative principal symbols such that
$\pa_{t}- L_{t}$ is uniformly  parabolic.  Suppose that $u\in \cC^{2,1}_{\fc}([0,T]\times X)$ is a solution to the initial value problem
\[
    \frac{\pa u}{\pa t} - L_{t}u = E, \quad u(0,\cdot)= u_{0}(\cdot),
\]
with $u_{0}\in x^{\gamma}\cC^{2}_{\fc}(X)$ and $E\in x^{\gamma}\cC_{\fc}^{0,0}([0,T]\times X)$ for some $\gamma=(\gamma_1,\ldots,\gamma_{\ell})$ with $\gamma_i\ge 0$.  Then we have that
\[
     u \in x^{\gamma}\cC^{0,0}_{\fc}([0,T]\times X).
\]
\label{de.9}\end{corollary}
\begin{proof}
It suffices to apply Proposition~\ref{de.4} to $u$ and $-u$.  
\end{proof}

To get a similar estimate on the derivatives of $u$, we need to know their corresponding evolution equations.  
\begin{lemma}
Let $V\to X$ be a smooth Euclidean vector bundle and let $t\mapsto L_{t}\in \Diff^{2}_{\fc}(X;V)$ be a uniformly elliptic smooth family of operators.  Let $\nabla^{V}$ be a choice of Euclidean connection for the bundle $V\to X$.  If a section
$u\in \cC^{k,\frac{k}{2}}_{\fc}([0,T]\times X; V)$ satisfies the evolution equation
\[
    \frac{\pa u}{\pa t}=  L_{t}u +E,  \quad E\in \cC^{k-2,\frac{k}{2}-1}_{\fc}([0,T]\times X;V),
\]  
then the evolution equation of $\nabla^{V}u$ is given by
\[
   \frac{\pa}{\pa t} \nabla^{V} u= \widetilde{L}_{t} \nabla^{V}u + \nabla^{V} E + f u,
\]
where $t\mapsto \widetilde{L}_{t}\in \Diff^{2}_{\fc}(X;T^{*}X\otimes V)$ is a uniformly elliptic smooth family of operators and 
$f\in \CI_{\fc}(X; T^{*}X\otimes \End(V))$. 
\label{de.10}\end{lemma}
\begin{proof}
Let $g_{\fc}$ be a choice of polyfibred cusp metric.  Using its Levi-Civita connection and the connection $\nabla^{V}$ of $V$, we have an induced connection on any tensor product of copies of $V$, $TX$ and their duals.  To simplify the notation, we will denote all these induced connections by $\nabla$.  In terms of these connections, the family of operators $L_{t}$ is of the form
\begin{equation}
         L_{t} u= a\cdot \nabla\nabla u + b\cdot\nabla u + cu
\label{de.10b}\end{equation}
where $a\in \CI_{\fc}(X; TX\otimes TX\otimes \End(V))$, 
$b\in \CI_{\fc}(X; TX\otimes \End(V))$ and $c\in \CI_{\fc}(X;\End(V))$.  From that perspective, the operator $L_{t}$ not only acts on sections of $V$, but also on sections of any bundle given by the tensor product of copies of $V$, $TX$ and their duals.  It suffices then to notice that  
\[
     \nabla L_{t} u= L_{t} \nabla u + [\nabla, L_{t}]u,
\]
and consequently that
\[
     \frac{\pa}{\pa t} \nabla u = L_{t}\nabla u + [\nabla, L_{t}] u + \nabla E.
\]

Since in local coordinates where the various bundles are trivialized, the covariant derivative $\nabla_i$ is the same as $\pa_i$ (the trivial covariant derivative) modulo terms of order zero, we see from the identity $[\pa_i, \pa_j]=0$ that the term $[\nabla, L_{t}] u$ involves at most two derivatives of $u$.  On the other hand, $L_{t}$, seen as an element of $\Diff^{2}_{\fc}(X; V\otimes T^{*}X)$, is still uniformly elliptic, so the result follows.
\end{proof}

With this lemma, we can now obtain the following decay estimate.

\begin{proposition}
For $t\in [0,T]$, let $t\mapsto L_{t}\in \Diff^{2}_{\fc}(X)$ be a smooth family of operators with negative principal symbols such that
$\pa_{t}- L_{t}$ is uniformly  parabolic.  Suppose that $u\in \cC^{k,\frac{k}{2}}_{\fc}([0,T]\times X)$ is  a solution to the initial value problem
\[
    \frac{\pa u}{\pa t} - L_{t}u = E, \quad u(0,\cdot)= u_{0}(\cdot),
\]
with $u_{0}\in x^{\gamma}\cC^{k}_{\fc}(X)$ and $E\in x^{\gamma}\cC_{\fc}^{k-2,\frac{k}{2}-1}([0,T]\times X)$ for some $k\ge 2$ and $\gamma=(\gamma_1,\ldots,\gamma_\ell)$ with $\gamma_i\ge 0$.  Then it follows that
\[
     u \in x^{\gamma}\cC^{k-2,\frac{k}{2}-1}_{\fc}([0,T]\times X).
\]
\label{de.3}\end{proposition}
\begin{proof}
We need to show that  
\begin{equation}
    \sup_{[0,T]\times X} \frac{| \pa_t^j \nabla^l u|}{x^\gamma} <\infty  
\label{de.3b}\end{equation}
for all  $j,l\in \bbN_0$ such that $2j+l\le k-2$.  We will proceed in two steps. \\

\noindent \textbf{Step 1:} The estimate \eqref{de.3b} holds for $j=0$ and $l\le k-2$.  \\

Our strategy is to proceed by induction on $k$ using Proposition~\ref{de.4} and Lemma~\ref{de.10}.  The case $k=2$ is given by Corollary~\ref{de.9}.  Thus, assume the result is true for $k=m$.  We need to prove it holds for $k=m+1$.  

Fix a polyfibred cusp metric $g_{\fc}$ and let $\nabla$ denote the corresponding Levi-Civita connection.  Proceeding by recurrence using Lemma~\ref{de.10}, we know that the evolution equation of the $(m-1)$th covariant derivative of $u$,
\[
    \nabla^{m-1}u= \underbrace{\nabla\cdots \nabla}_{m-1 \; \mbox{times}} u,
\]  
is given by
\begin{equation}
   \frac{\pa}{\pa t} \nabla^{m-1}u = L_{t}^{m-1} \nabla^{m-1}u + \nabla^{m-1}E +
     \sum_{j=0}^{m-2} b_{j} \nabla^{j}u,
\label{de.11}\end{equation}
where $L_{t}^{m-1}$ is a uniformly elliptic family of operators and 
\[
b_j \in\CI_{\fc}(X; \End( T^j_0X, T^{m}_0X )).
\]
Thus, we have
\begin{equation}
\begin{aligned}
   \frac{\pa}{\pa t} | \nabla^{m-1} u |^{2} &= 2\langle \frac{\pa}{\pa t} \nabla^{m-1}u,  \nabla^{m-1}u\rangle  \\
   &= 2\langle L_{t}^{m-1} \nabla^{m-1}u + \nabla^{m-1}E + \sum_{j=0}^{m-2}b_{j}\nabla^{j}u, \nabla^{m-1}u\rangle,
\end{aligned}
\label{de.12}\end{equation}
where the symbol $\langle\cdot,\cdot\rangle$ denote the inner product between tensors induced by the polyfibred cusp metric $g_{\fc}$.  Writing the operator $L_{t}^{m-1}$ as in \eqref{de.10b},
\begin{equation}
   L_{t}^{m-1} v= a^{\alpha\beta}\cdot \nabla_{\alpha}\nabla_{\beta} v + b^{\alpha}\cdot \nabla_{\alpha} v + c v,
\label{de.13}\end{equation}
where the sum is taken over repeated indices,   
we see using the ellipticity of $L_{t}^{m-1}$ that,
\begin{equation}
\begin{aligned}
  L_{t}^{m-1}| \nabla^{m-1}u|^{2} &= 2a^{\alpha\beta}\cdot \langle\nabla_{\alpha}\nabla^{m-1}u,\nabla_{\beta}\nabla^{m-1}u\rangle + 2\langle L_{t}^{m-1}\nabla^{m-1}u, \nabla^{m-1}u\rangle  \\
&\ge   2\langle L_{t}^{m-1}\nabla^{m-1}u, \nabla^{m-1}u\rangle.
\end{aligned}
\label{de.14}\end{equation}
Combining \eqref{de.12} with \eqref{de.14} and using the Cauchy-Schwarz inequality, we obtain
\begin{equation}
\begin{aligned}
\frac{\pa}{\pa t} | \nabla^{m-1} u |^{2}&\le L^{m-1}_{t} | \nabla^{m-1} u |^{2} +
2\langle \nabla^{m-1}E + \sum_{j=0}^{m-2}b_{j}\nabla^{j}u, \nabla^{m-1}u\rangle    \\
&\le L^{m-1}_{t} | \nabla^{m-1} u |^{2} +  |\nabla^{m-1}E|^{2} + |\nabla^{m-1}u|^{2} \\
&\quad \quad + \sum_{j=0}^{m-2} \left(| b_{j}\nabla^{j}u|^{2} + |\nabla^{m-1}u|^{2}\right).
\end{aligned}
\label{de.15}\end{equation}
Since we assume by induction that the estimate \eqref{de.3b} holds for $j=0$ and $l\le m-2$, we conclude that $|\nabla^{m-1}u|^{2}$ satisfies an evolution equation of the form
\begin{equation}
\frac{\pa}{\pa t} | \nabla^{m-1} u |^{2}\le L^{m-1}_{t} | \nabla^{m-1} u |^{2} + m|\nabla^{m-1} u|^2 +
  E^{m-1}, 
\label{de.16}\end{equation}  
with  $E^{m-1}\in x^{2\gamma}\cC^{0}_{\fc}([0,T]\times X)$.  
By Proposition~\ref{de.4}, this means there exist positive constants $K$ and $c$ such that
\begin{equation}
   0 \le |\nabla^{m-1} u|^{2} \le  Ke^{ct} x^{2\gamma}, \quad \forall\; t\in [0,T],
\label{de.17}\end{equation}
which completes \textbf{Step 1}. \\

\noindent \textbf{Step 2:} For fixed $j\in\bbN_0$ with $2j\le k-2$, the estimate \eqref{de.3b} holds for all $l\in \bbN_0$ such that
$l\le k-2-2j$. \\

 We will proceed by induction on $j$.  The case $j=0$ is \textbf{Step 1}.  From the evolution equation of $u$, we see that
 \begin{equation}
        \sup_{[0,T]\times X} \frac{1}{x^{\gamma}}\left| \nabla^l \frac{\pa u}{\pa t} \right| = \sup_{[0,T]\times X} \frac{\left| \nabla^l (L_t u +E) \right|}{x^{\gamma}} <\infty   \quad \forall \ l\le k-4,
 \label{de.17b}\end{equation}
which is the case $j=1$.  More generally, since $\pa_t^{j} u$ satisfies the evolution equation
\begin{equation}
  \frac{\pa}{\pa t} ( \pa^j_t u) = L_t(\pa^{j}_t u) + \pa_t^jE + 
    \sum_{q=0}^{j-1} \left( \frac{j!}{q!(j-q)!} \right) (\pa_t^{j-q}L_t) (\pa_t^q u),
\label{de.17c}\end{equation}
 we see from this equation that if \eqref{de.3b} holds for $2j=2m\le k-4$ and $l\le k-2-2m$, then
 it also holds for $j=m+1$ and $l\le k-4-2m$.  This shows that \eqref{de.3b} holds for all $j,l\in\bbN_0$ with $2j+l\le k-2$.
\end{proof}

As a prelude to our study of the Ricci flow asymptotics, we can use Proposition~\ref{de.3} to describe the asymptotic behavior of solutions to linear uniformly parabolic equations.

\begin{theorem}
For $t\in [0,T]$, let $t\mapsto L_t \in \Diff^2_{\fc}(\tX)$ be a smooth family of uniformly elliptic operators with negative principal symbols.  Then for $E\in \CI_{\fc}([0,T]\times \tX)$, the equation
\begin{equation}
    \frac{\pa u}{\pa t} = L_t u +E , \quad u(0,\cdot)=u_0\in \CI_{\fc}(\tX),
\label{de.18b}\end{equation}
has a unique solution $u$ in $\CI_{\fc}([0,T]\times \tX)$.  
\label{de.18}\end{theorem} 
\begin{proof}
Replacing $u$ by $u-u_0$ if needed, we can assume $u_0=0$.  
Proceeding by induction on the depth of $\tX$, we can assume the theorem holds on $\tD_i$ for all $i\in \{1,\ldots,\ell\}$.  Indeed, when the depth of $\tX$ is zero, that is, when $\tX$ is a closed manifold, the theorem is a standard result.    
Using Schauder theory, the differential equation \eqref{de.18b} has a solution $u\in \CI_{\fc}(X)$.  More precisely, let $\{\Omega_j\}$ be an exhaustion of $X$ by compact subdomains and let $\phi_j\in \CI([0,T]\times \Omega_i)$ be the unique solution to the equation
\[
     \frac{\pa \phi_j}{\pa t} = L_t \phi_j +E,  \quad  \phi_j(0,\cdot)= 0, \quad 
         \left. \phi_j \right|_{\pa\Omega_j} =0.
\]
By the maximum principle, we can find a constant $K>0$ depending on $T$, $\sup_{[0,T]\times X} |E|$ and $L_t$, but not on $j$, such that 
\[
                   \sup_{[0,T]\times \Omega_j} |\phi_j| < K.
\]
Combining this with the interior parabolic Schauder estimate (see for instance \cite{Krylov}) in good quasi-coordinates as in \cite{Kobayashi} or \cite{Wu06} gives uniform control on the $\cC^{k}_{\fc}$-norms of $\phi_j$, so that by Arzela-Ascoli, there exists a function $u\in \CI_{\fc}(X)$ such that $\phi_j\to u$  in $\cC^{k}_{\fc}(X)$ uniformly on each compact subset and for each $k\in \bbN_0$.  In particular, $u\in \CI_{\fc}(X)$ is a solution to \eqref{de.18b}.

We need to show that it is in fact in $\CI_{\fc}(\tX)$.  This amounts to show that for all $\gamma\in \bbN_0^{\ell}$, there exists $v_{\gamma}\in \CI_{\fc}([0,T]\times \tX)$ such that 
\begin{equation}
  u-v_{\gamma} \in x^{\gamma}\CI_{\fc}([0,T]\times X).  
\label{de.19}\end{equation}
Clearly, for $|\gamma|=0$, it suffices to take $v_\gamma =0$ for \eqref{de.19} to hold.  To find 
$v_{\gamma}$ for all $\gamma\in \bbN_0^{\ell}$, we can proceed by induction on $k=|\gamma |$.  

Thus, assume that for some $k\in \bbN_0$ and for all $\gamma\in \bbN_0^{\ell}$ with $|\gamma|\le k$, we can find $v_{\gamma}\in \CI_{\fc}([0,T]\times \tX)$ such that \eqref{de.19} holds.  Let 
$\alpha\in \bbN_0^{\ell}$ with $|\alpha|=k+1$ be given.   We need to show that we can find
$v_{\alpha}\in \CI_{\fc}([0,T]\times \tX)$ such that 
\[
               u-v_{\alpha}\in x^{\alpha} \CI_{\fc}([0,T]\times X).  
\]
Choose $\gamma,\beta\in \bbN^{\ell}_{0}$ such that $\alpha=\gamma+\beta$ with $|\gamma|=k$ and $|\beta|=1$.  By our inductive hypothesis, we can find $v_{\gamma}\in \CI_{\fc}([0,T]\times \tX)$ such that \eqref{de.19} holds.  This means the function 
\[
            w= \frac{u-v_\gamma}{x^{\gamma}}
\]
is in $\CI_{\fc}([0,T]\times X)$.  From the evolution equation of $u$, we see that the function $w$ satisfies the evolution equation
\begin{equation}
  \frac{\pa w}{\pa t} = L_t^{\gamma} w+ E^\gamma, \quad w(0,\cdot)= w_0 =\frac{u_0- v_{\gamma}(0,\cdot)}{x^{\gamma}}\in \CI_{\fc}(\tX),
\label{de.20}\end{equation}             
with $L_t^{\gamma}= x^{-\gamma}\circ L_t \circ x^{\gamma}$ and $E^{\gamma}= x^{-\gamma}(E+L_t v- \frac{\pa v_{\gamma}}{\pa t})$.  Since $\sigma_2(L_t^{\gamma})= \sigma_2(L_t)$, the family $L_t^{\gamma}\in \Diff^2_{\fc}(\tX)$ is also uniformly elliptic.  Moreover, since $\frac{\pa w}{\pa t}$ and 
$L_t^{\gamma}w$ are in $\CI_{\fc}([0,T]\times X)$, we infer from \eqref{de.20} that $E^{\gamma}\in \CI_{\fc}([0,T]\times X)$.  On the other hand, since $E$, $\frac{\pa v_{\gamma}}{\pa t}$  and $L_t v_{\gamma}$ are in $\CI_{\fc}([0,T]\times \tX)$, this means that in fact 
\[
         E^{\gamma}\in \CI_{\fc}([0,T]\times \tX).  
\]
      
Now, because $|\beta|=1$, there exists $i$ such that $\beta_i=1$ and $\beta_j=0$ for $j\ne i$.  For this $i$, we can look at the restriction of \eqref{de.20} to $\bD_i$, which is given by
\begin{equation}
   \frac{\pa w_i}{\pa t} = L^{\gamma}_{i,t} w_i +E^{\gamma}_i, \quad w_i(0,\cdot)= w_{0,i}\in \CI_{\fc}(\tD_i),
\label{de.21}\end{equation}      
where $E^{\gamma}_i = \left. E^{\gamma}\right|_{\tD_i}$ and $L^{\gamma}_{i,t}\in \Diff^2_{\fc}(\tD_i)$ is the family of differential operators such that 
\[
            L^{\gamma}_{i,t} ( \left.f\right|_{\tD_i}) = \left. L_t^{\gamma} f \right|_{\tD_i}, \quad 
            \forall f\in \CI_{\fc}(\tX).  
\]
When $\dim \tD_i >0$, this family is uniformly elliptic, while when $\dim \tD_i =0$, the operator $L^{\gamma}_{i,t}$ is of order zero and the evolution equation \eqref{de.21} is an ordinary differential equation.  

In any case, by our inductive hypothesis on the depth of $\tX$, we know that 
the evolution equation \eqref{de.21} has a unique solution $w_i\in \CI_{\fc}([0,T]\times \tD_i)$.  This suggests to consider the function $w_{\Xi_i}= \Xi_i(w_i)\in \CI_{\fc}([0,T]\times \tX)$, where $\Xi_i$ is the map introduced in \eqref{ext.7}.  Since it satisfies the evolution equation
\[
        \frac{\pa w_{\Xi_i}}{\pa t} = \Xi_i \left( L^{\gamma}_{i,t} w_i + E_i^{\gamma}\right),
\]
we see that 
\begin{equation}
    \frac{\pa}{\pa t} ( w-w_{\Xi_i})= L_t^{\gamma} ( w-w_{\Xi_i}) +F
\label{de.22}\end{equation}
with 
\[
  F= L_t w_{\Xi_i} - \Xi_i (L^{\gamma}_{i,t} w_i) + E^{\gamma} - \Xi_i ( E^{\gamma}_i ) \in \CI_{\fc}([0,T]\times \tX).
\]
From the definition of the operator $L^{\gamma}_{i,t}$, we see that the restriction of $F$ to $\tD_i$ vanishes, so that in fact $F\in x_i\CI_{\fc}([0,T]\times \tX)$.  Applying Proposition~\ref{de.3} to the evolution equation 
\eqref{de.22}, we therefore conclude that 
\[
        w-w_{\Xi_i} \in x_i\CI_{\fc}([0,T]\times \tX).  
\]
Consequently, if we pick $v_{\alpha}= v_{\gamma} + x^{\gamma}w_{\Xi_i}\in \CI_{\fc}([0,T]\times \tX)$, we have as desired that 
\[
     u-v_{\alpha}=  x^{\gamma}(w-w_{\Xi_i}) \in x^{\alpha}\CI_{\fc}([0,T]\times \tX).
\]

Finally, to show the solution is unique, we can simply apply Yau's generalized maximum principle.  Alternatively,  suppose $u'\in \CI_{\fc}(\tX)$ is another solution.  Then its restriction to $\tD_i$ satisfies the same parabolic equation as the one of $u$.  Since the theorem holds on $\tD_i$ by our inductive assumption, this means $v=u- u'$ vanishes on $\pa \tX$.  Moreover, it satisfies the evolution equation
\[
          \frac{\pa v}{\pa t} = L_t v, \quad v(0,\cdot)= 0.  
\] 
Applying the maximum principle on an exhaustion of $X$ by compact sets, we thus conclude $v\equiv 0$, establishing uniqueness.  

\end{proof}

\section{Evolution of spatial asymptotics along the Ricci flow} \label{ha.0}

Let $\omega_{0}$ be the K\"ahler form of an asymptotically tame polyfibred cusp K\"ahler metric on the quasiprojective manifold $X= \bX\setminus \bD$ with standard spatial asymptotics $\{\omega_{I,0}\}$ and $\{c_{i}\}$.  Consider the normalized K\"ahler-Ricci flow associated to this metric, 
\begin{equation}
    \frac{\pa \widetilde{\omega}_{t}}{\pa t} = - \Ric( \widetilde{\omega}_{t})- \widetilde{\omega}_{t} , \quad \widetilde{\omega}_{0}= \omega_{0}.
\label{sch.1}\end{equation}
As in \cite{Lott-Zhang}, consider the ansatz $\widetilde{\omega}_{t}= \omega_t + \sqrt{-1} \pa\db u$
with 
\begin{equation}
    \omega_{t}= -\Ric (\omega_{0}) + e^{-t}(\omega_{0}+ \Ric(\omega_{0})).
\label{sch.2}\end{equation}
Then the 
evolution equation of the potential function $u$ is given by
\begin{equation}
  \frac{\pa u}{\pa t}= \log \left(  \frac{ (\omega_{t} + \sqrt{-1} \pa \db u)^{n} }{ \omega_{0}^{n}}\right)- u
  , \quad u(0,\cdot)=0.
\label{sch.3}\end{equation}
\begin{remark}
When $n=1$, it is more effective to describe the evolution of the metric in terms of a conformal factor, since the evolution equation is then quasi-linear instead of fully nonlinear.  However, to keep a uniform treatment independent of the dimension of $\bX$, we will refrain from doing so and refer to \cite{Albin-Aldana-Rochon} for a study of the spatial asymptotics along the Ricci flow in complex dimension $1$ using conformal factors.      
\label{sch.3b}\end{remark}

In this section, we will prove the following result about the asymptotic behavior of a solution to
\eqref{sch.3}.  

\begin{theorem}
Suppose that $\omega_{0}$ is the K\"ahler form of an asymptotically tame polyfibred cusp K\"ahler metric.  If $u\in\CI_{\fc}([0,T]\times X)$ is a solution to the evolution equation \eqref{sch.3}, then in fact $u\in \CI_{\fc}([0,T]\times \tX)$.
\label{ha.1}\end{theorem}
\noindent It has the following immediate consequence.
\begin{corollary}
Suppose that $\omega_{0}$ is the K\"ahler form of an asymptotically tame polyfibred cusp K\"ahler metric.  If $\widetilde{\omega}_{t}$ is the solution to the normalized K\"ahler-Ricci flow \eqref{sch.1} on the time interval $[0,T]$, then for all $t\in [0,T]$, $\widetilde{\omega}_{t}$ is the K\"ahler form of an asymptotically tame polyfibred cusp K\"ahler metric on $X$.  
\label{ha.2}\end{corollary}

To indicate the main idea in the proof  of this result, suppose first that we have a solution
$u\in\CI_{\fc}([0,T]\times \tX)$ to \eqref{sch.3} with the desired regularity.  In that case, the restriction of $u$ to the boundary hypersurface $\tH_{i}$ makes sense.  In fact, it will be constant along the fibres of the fibration $\tPhi_{i}:\tH_{i}\to \tD_{i}$, so its restriction to $\tD_{i}$ is also well-defined.  On the other hand, by Proposition~\ref{fc.28}, $\omega_{t}$ is asymptotically tame, so it makes sense to restrict the evolution equation \eqref{sch.3} to $\tD_{i}$.  If $u_{i}$ denotes the restriction of $u$ to $\tD_{i}$, a simple computation shows that \eqref{sch.3} restricts on $\tD_{i}$ to give
\begin{equation}
   \frac{\pa u_{i}}{\pa t} = \log \left( \frac{ (\omega_{i,t} + \sqrt{-1}\pa \db u_{i})^{n-1}}{ \omega_{i,0}^{n-1}}
   \right)  - u_{i} + \log \left( \frac{1+ e^{-t}(c_{i}-1)}{c_{i}}\right), \quad u_{i}(0,\cdot)=0,
\label{ha.3}\end{equation}
where $\omega_{i,t}= -\Ric(\omega_{i,0})+ e^{-t}( \omega_{i,0} + \Ric(\omega_{i,0}))$.  Thus, except for the last term in this equation, this is basically \eqref{sch.3} on the quasiprojective manifold $\tD_{i}\setminus \pa \tD_{i}$.  
  More generally, for $I\subset\{ 1,\ldots, \ell\}$, the restriction of $u$ to
$\tH_{I}$ and $\tD_{I}$ is well-defined, as well as the restriction of \eqref{sch.3}, which gives,
\begin{equation}
     \frac{\pa u_{I}}{\pa t} = \log \left( \frac{ (\omega_{I,t} + \sqrt{-1}\pa \db u_{I})^{n-|I|}}{ \omega_{I,0}^{n-|I|}}
   \right)  - u_{I} + \sum_{i\in I} \log \left( \frac{1+ e^{-t}(c_{i}-1)}{c_{i}}\right), \ u_{I}(0,\cdot)=0,
\label{ha.4}\end{equation}
where $\omega_{I,t}= -\Ric(\omega_{I,0}) + e^{-t}(\omega_{I,0} + \Ric(\omega_{I,0}))$.  When $\bD_I$ is zero dimensional, notice the evolution equation \eqref{ha.4} is really just an ordinary differential equation,
\begin{equation}
\frac{\pa u_{I}}{\pa t} =  - u_{I} + \sum_{i\in I} \log \left( \frac{1+ e^{-t}(c_{i}-1)}{c_{i}}\right), \quad u_{I}(0,\cdot)=0.
\label{ha.3b}\end{equation}

Using Proposition~\ref{fc.28}, we can also consider the restriction of \eqref{sch.1} to $\tD_{I}$ and the normal directions, which gives
\begin{gather}
   \frac{\pa \widetilde{\omega}_{I,t}}{\pa t} = - \Ric( \widetilde{\omega}_{I,t})- \widetilde{\omega}_{I,t} , \quad \widetilde{\omega}_{I,0}= \omega_{I,0},  \label{ha.5} \\
   \frac{\pa c_{i,t}}{\pa t}= 1- c_{i,t}, \quad c_{i,0}= c_{i},  \label{ha.6}
 \end{gather}
where $\{\widetilde{\omega}_{I,t}\}$ and $\{c_{i,t}\}$ are the spatial asymptotics of $\widetilde{\omega}_{t}$.  Solving for $c_{i,t}$ we find
\begin{equation}
   c_{i,t}= 1+ e^{-t}(c_{i}-1).  
\label{ha.7}\end{equation}

To summarize, if the restriction of $u$ to $\tD_{I}$ exists, it has to be the solution to \eqref{ha.4}.   Thus, if we are given a solution $u\in \CI_{\fc}([0,T]\times X)$ to \eqref{sch.3} with a priori no known regularity at the boundary $\pa \tX$, then the solution to \eqref{ha.4} is necessarily the natural candidate for what should be the restriction of $u$ to $\tD_{I}$.  This observation is the starting point of our proof.  Our strategy will be to use a barrier function, that is, Proposition~\ref{de.3} for a suitable linear parabolic equation, to show that $u$ does indeed restrict to give the solution to \eqref{ha.3} on $\tD_{i}$.  This argument can then be repeated to give the full asymptotics of the solution at each of the boundary hypersurfaces.    

Initially however, we will only be able to establish such a result for a small time interval.  Provided we can control the size of this time interval in a uniform way, we will then be able to apply this argument a finite number of times to cover the full interval $[0,T]$.  This requires to consider a shifted version of the evolution equation~\eqref{sch.3}.  Namely, fix $t_{0}\in [0,T)$ and consider the new function
\begin{equation}
    \hu(t,\cdot)=  u(t+t_{0},\cdot)- u(t_{0},\cdot),  
\label{ha.8}\end{equation} 
satisfying the evolution equation
\begin{equation}
\frac{\pa \hu}{\pa t}= \frac{\pa u(t+t_{0},\cdot)}{\pa t}= 
  \log \left(  \frac{(\omega_{t+t_{0}} + \sqrt{-1}\pa \db u(t+t_{0},\cdot))^{n}}{\omega_{0}^{n}}
  \right) - u(t+t_{0},\cdot), \quad \hu(0,\cdot)=0.
\label{ha.9}\end{equation}
This can be rewritten in the more suggestive form
\begin{equation}
 \frac{\pa \hu}{\pa t}= \log \left( \frac{(\homega_{t}+\sqrt{-1}\pa \db \hu)^{n}}{\omega_{0}^{n} }
 \right) - \hu - u(t_{0},\cdot), \quad \hu(0,\cdot)=0,
\label{ha.10}\end{equation}
where $\homega_{t}= \omega_{t+t_{0}} + \sqrt{-1}\pa \db u(t_{0},\cdot)$.  In particular, notice that 
by assumption, 
\[
       \homega_{t}+ \sqrt{-1} \pa \db \hu= \widetilde{\omega}_{t+t_{0}}
\]
is a K\"ahler form, so that the logarithmic term in \eqref{ha.10} is well-defined.  

Suppose now that $u(t_{0},\cdot)\in \CI_{\fc}(\tX)$ and that $\widetilde{\omega}_{t_{0}}$ is an asymptotically tame polyfibred cusp metric.  Our goal will be to show that under this assumption,
 the function $\hu(t,\cdot)$ will also be in $\CI_{\fc}(\tX)$ provided $t$ is small enough.  This will require a few steps.  First, consider the restriction of the evolution equation \eqref{ha.10} to $\tD_{I}$, 
\begin{equation}
\begin{aligned}
\frac{\pa \hu_{I}}{\pa t} &= \log\left(  \frac{(\homega_{I,t}+ \sqrt{-1} \pa \db \hu_{I} )^{n-|I|}}{\homega_{I,0}^{n-|I|}}
\right)   \\
&  \quad - \hu_{I} -u_{I}(t_{0},0)+ \sum_{i\in I} \log\left(  \frac{1+ e^{-t-t_{0}}(c_{i}-1)}{c_{i}}
\right), \quad \hu_{I}(0,\cdot)=0.
\end{aligned}
\label{ha.11}\end{equation}
Even if we do not know if $\hu$ is in $\CI_{\fc}([0,T-t_{0}]\times \tX)$, the evolution equation \eqref{ha.11} still makes sense and we know that a solution exists for a short period of time.  We will in fact assume that it has a solution $\hu_{I}\in \CI_{\fc}([0,T-t_{0}]\times D_I)$ on the time interval $[0,T-t_{0}]$.  As we will see, this can be justified a posteriori.   Since we will proceed by induction on the depth of $\tX$ to prove Theorem~\ref{ha.1}, we  might as well assume $\hu_{I}$ is in fact a solution in
$\CI_{\fc}([0,T-t_{0}]\times \tD_I)$.  

Using the extension map $\Xi_i$ of \eqref{ext.7}, consider the function
\begin{equation}
   \hu_{\Xi_{i}}=  \Xi_i(\hu_{i}).
\label{ha.12}\end{equation}
Since the extension map $\Xi_i$ does not depend on time, the evolution equation of $\hu_{\Xi_{i}}$ is given by
\begin{multline}
   \frac{\pa \hu_{\Xi_i}}{\pa t} = \Xi_i  \left(\log\left(  \frac{( \homega_{i,t}+ \sqrt{-1}\pa\db \hu_{i})^{n-1}}{\omega_{i,0}^{n-1}} \right) -\hu_{i} -u_{i}(t_{0},\cdot) \right) \\
   +\Xi_i\left( \log \left( \frac{1+ e^{-t-t_{0}}(c_{i}-1)}{c_{i}}\right)\right), \quad \hu_{\Xi_i}(0,\cdot)=0.
\label{ha.13}\end{multline}
By construction, the restriction of the function $\hu_{\Xi_i}$ to $\tD_i$ is $\hu_{i}$.  Thus, if $\hu$ has the same restriction on $\tD_i$, then we expect the function 
\begin{equation}
   v_{i}= \hu- \hu_{\Xi_i}
\label{ha.13b}\end{equation}
to vanish on $\tH_i$.  From the evolution equations of $\hu$ and $\hu_{\Xi_{i}}$, the evolution equation of $v_{i}$ can be seen to be
\begin{equation}
\frac{\pa v_{i}}{\pa t} = \log\left( \frac{(\homega_{t}+\sqrt{-1}\pa\db \hu)^{n}}{\omega_{0}^{n}}\right) - \log\left( \frac{(\homega_{t}+\sqrt{-1}\pa\db \hu_{\Xi_{i}})^{n}}{\omega_{0}^{n}}\right) -v_{i} +E_{i}
\label{ha.13c}\end{equation}
with $E_{i}$ given by 
\begin{equation}
\begin{aligned}
E_{i} &= 
-\Xi_i \left(
    \log\left(  \frac{( \homega_{i,t}+ \sqrt{-1}\pa\db \hu_{i})^{n-1}}{\omega_{i,0}^{n-1}} \right) - u_{i}(t_{0},\cdot)  \right) \\
    &  \quad - \Xi_i\left( \log \left( \frac{1+ e^{-t-t_{0}}(c_{i}-1)}{c_{i}}  \right)\right) \\
    & \quad +\left( \log\left( \frac{(\homega_{t}+\sqrt{-1}\pa \db \hu_{\Xi_i})^{n}}{\omega_{0}^{n}}\right) - u(t_0,\cdot) \right). 
\end{aligned}  
\label{ha.13d}\end{equation}
For the logarithmic term involving $\hu_{\Xi_i}$ to make sense, the 2-form $\homega_t+ \sqrt{-1}\pa\db \hu_{\Xi_i}$ must be non-degenerate.  To insure this is the case, notice first that there exist positive  constants $c$ and $C$ depending on the solution $\widetilde{\omega}_{t}\in \CI_{\fc}([0,T]\times X; \Lambda^{2}(T^{*}X))$ of \eqref{sch.1} such that 
\begin{equation}
  c\omega_{0} < \widetilde{\omega}_{t}< C\omega_{0}, \quad \left\| \frac{\pa \widetilde{\omega}_{t}}{\pa t} \right\|_{\cC^{0}_{\fc}(X; \Lambda^{2}(T^{*}X))} < C, \quad \forall \, t\in [0,T].  
\label{ha.15}\end{equation} 
Taking the constant $C>0$ bigger if necessary and depending also on the solutions $u_{I}\in\CI_{\fc}(D_{I})$ to \eqref{ha.4}, we can also assume that it is such that
\begin{equation}
\left\| \frac{\pa}{\pa t} \pa\db u \right\|_{\cC^{0}_{\fc}(X; \Lambda^{2}(T^{*}X))} < C, \quad \left\| \frac{\pa}{\pa t} \pa\db u_{I} \right\|_{\cC^{0}_{\fc}(X; \Lambda^{2}(T^{*}X))} < C, \quad \forall\, t\in [0,T].
\label{ha.15b}\end{equation}  
Recalling that $\homega_{0}= \widetilde{\omega}_{t_0}$, we see that, in view of \eqref{ha.15}, we can find $\tau>0$ depending only on $c$ and $C$, so in particular independent of $t_0$, such that
\begin{equation}
     \homega_{t} > \frac{c\omega_{0}}{2}, \quad \forall t\in [0,\tau_{t_0}], \quad \tau_{t_0}:= \min\{\tau, T-t_0\}.
\label{ha.16}\end{equation}
From \eqref{ha.15b}, we see that choosing $\tau>0$ smaller if needed, but still only depending on the constants $c$ and $C$, so independent of the choice of $t_0$, we can assume $\pa\db u_{\Xi_i}$ is sufficiently small so that 
\begin{equation}
     \homega_{t} +\sqrt{-1} \pa\db \hu_{\Xi_i} > \frac{c\omega_{0}}{4}, \quad \forall t\in [0,\tau_{t_0}], \quad \tau_{t_0}=\min\{\tau, T-t_0\}.
\label{ha.16b}\end{equation}
Thus, at least for $t\in [0,\tau_{t_0}]$, we can write the evolution equation of $v_i$ as in \eqref{ha.13c}.  

\begin{lemma}
  If $\hu_{i}$ is in  $\CI_{\fc}([0,\tau_{t_0}]\times \tD_{i})$ and $\homega_{0}$ is an asymptotically tame polyfibred cusp K\"ahler metric, then the term $E_{i}$ in \eqref{ha.13d} is in 
  $x_{i}\CI_{\fc}([0,\tau_{t_0}]\times \tX)$.  
\label{ha.14}\end{lemma}
\begin{proof}
 Since $E_i\in \CI_{\fc}([0,\tau_{t_0}]\times\tX)$, it suffices to verify that its restriction to $\tH_i$ vanishes.  By the definition of the extension map $\Xi_i$, the restriction of the first two terms in \eqref{ha.13d} is given by 
 \begin{equation}
    -\log\left(  \frac{( \homega_{i,t}+ \sqrt{-1}\pa\db \hu_{i})^{n-1}}{\omega_{i,0}^{n-1}} \right) + u_{i}(t_{0},\cdot)  
    - \log \left( \frac{1+ e^{-t-t_{0}}(c_{i}-1)}{c_{i}}  \right).
\label{ext.11}\end{equation} 
For the third term, since $\hu_{\Xi_i}\in \CI_{\fc}(\tX\times [0,\tau_{t_0}])$, we know from Proposition~\ref{ext.8} that the restriction of $\pa\db \hu_{\Xi_i}$ to $\tH_i$ is given by $\tPhi_i^*(\pa\db \hu_i)$.  From the standard spatial asymptotics of $\homega_t$, we thus see that, as a section of $\Lambda^2({}^{\fc}T^{*}\tX)$, the restriction of $\homega_t+ \sqrt{-1}\pa\db \hu_{\Xi_i}$ to $\tH_i$  is given by
\begin{equation}
  \tPhi_i^{*}( \homega_{i,t} + \sqrt{-1}\pa\db \hu_i) + \frac{\sqrt{-1}}{2} \frac{   (1+ e^{-t-t_0}(c_i-1)) d\zeta_i\wedge d\overline{\zeta}_i}{|\zeta_i|^{2}(\log|\zeta_i |)^2}.
\label{ext.12}\end{equation} 
One can then conclude from this that the restriction of the third term in \eqref{ha.13d} exactly cancels the restriction \eqref{ext.11} of the first two terms.  Thus, the restriction $E_i$ to $\tH_i$ is zero.
\end{proof}

To show that $v_{i}$ vanishes on $\tH_{i}$, we want to rewrite the difference of the two logarithmic terms in \eqref{ha.13c} as a linear expression in $v_{i}$ which would allow us to use Proposition~\ref{de.3}.  From \eqref{ha.16}, we see that  for $t\in[0,\tau_{t_0}]$, we can rewrite the first logarithmic term in \eqref{ha.13c} as
\begin{equation}
\begin{aligned}
\log \left( \frac{ (\homega_{t}+ \sqrt{-1}\pa\db \hu)^{n}}{\omega_0^{n}}\right) &=
  \log \left( \frac{ (\homega_{t}+ \sqrt{-1}\pa\db \hu)^{n}}{\homega_{t}^{n}}\right)
    + \log \left( \frac{\homega^{n}_{t}}{\omega_{0}^{n}}\right) \\
    &= \log\left(  1 + \hDelta_t \hu + F(\pa\db \hu, \homega_t) \right)   + \log \left( \frac{\homega^{n}_{t}}{\omega_{0}^{n}}\right),\end{aligned}    
  \label{ha.17}\end{equation}
where $\hDelta_t$ is the $\db$-Laplacian associated to $\homega_t$ and the function $F$ is a polynomial of degree $n$ in the first variable with no constant and linear terms.  Let $G(x)$ be the smooth function such that 
\[
     \log(1+x) = x + x^{2}G(x), \quad \forall x\in (-1,\infty).  
\]
Applying this relation to \eqref{ha.17}, we obtain
\begin{equation}
\begin{aligned}
\log \left( \frac{ (\homega_{t}+ \sqrt{-1}\pa\db \hu)^{n}}{\omega_0^{n}}\right) &=
 \hDelta_t \hu + F(\pa \db \hu, \homega_t)  \\
 & \quad + (\hDelta_t \hu + F(\pa\db \hu,\homega_t))^{2} G(\hDelta_t \hu + F(\pa \db\hu,\homega_t)) + \log \left( \frac{\homega^{n}_{t}}{\omega_{0}^{n}}\right) \\
 &= \hDelta_t \hu + H(\pa\db \hu,\homega_t) + \log \left( \frac{\homega^{n}_{t}}{\omega_{0}^{n}}\right),
\end{aligned}  
\label{ha.18}\end{equation}
with $H$ a smooth function which vanishes at least quadratically in the first variable.  For the other logarithmic term, we have similarly
\begin{equation}
\log \left( \frac{ (\homega_{t}+ \sqrt{-1}\pa\db \hu_{\Xi_{i}})^{n}}{\omega_0^{n}}\right) = \hDelta_t \hu_{\Xi_i} + H(\pa\db \hu_{\Xi_i},\homega_t) + \log \left( \frac{\homega^{n}_{t}}{\omega_{0}^{n}}\right).
\label{ha.19}\end{equation}
Because $H$ vanishes quadratically in the first variable, we can rewrite the difference of the two logarithmic terms as
\begin{multline}
\log \left( \frac{ (\homega_{t}+ \sqrt{-1}\pa\db \hu)^{n}}{\omega_0^{n}}\right)- \log \left( \frac{ (\homega_{t}+ \sqrt{-1}\pa\db \hu_{\Xi_{i}})^{n}}{\omega_0^{n}}\right) = \\    
\hDelta_t v_i + K(\pa \db \hu, \pa\db \hu_{\Xi_i}, \homega_t,\pa\db v_i),
\label{ha.20}\end{multline}
where $K$ is a smooth function defined in terms of the function $H$ that can be chosen to be linear in the last variable and vanishing at least linearly in the first two variables.  Notice that the definitions of  $K$ and $H$ only depend on the function
\[
       (f, \Omega) \mapsto \log \left( \frac{ (\Omega+ \sqrt{-1}\pa\db f)^{n}}{\Omega^{n}} \right)
\]   
and not on the particular solutions $\hu$, $\hu_{\Xi_i}$ and $\homega_t$.  In light of  \eqref{ha.15} and \eqref{ha.15b},  this means that taking the constant $\tau>0$ smaller if needed, and still only depending on $c$ and $C$, we can assume that for each $i\in\{1,\ldots,\ell\}$, the family of operators $t\mapsto L_{i,t}\in \Diff^{2}_{\fc}(X)$ defined by 
\begin{equation}
  L_{i,t} f= \hDelta_t f + K(\pa\db\hu, \pa\db \hu_{\Xi_i}, \homega_{t},\pa\db f)
\label{ha.21}\end{equation}  
is uniformly elliptic for $t\in [0,\tau_{t_0}]$, say
\begin{equation}
    \sigma_{2}(L_{i,t}) < -\frac{cg_{\omega_{0}}}{4}, \quad \forall \, t\in [0,\tau_{t_0}], \quad \tau_{t_0}= \min\{\tau, T-t_0\}.  
\label{ha.22}\end{equation} 
In terms of the family of operators $L_{i,t}$, the evolution equation of the function $v_{i}$ can be rewritten
\begin{equation}
    \frac{\pa v_{i}}{\pa t} = L_{i,t} v_{i} - v_{i} +E_{i}, \quad v_{i}(0,\cdot)=0, \quad 
    E_{i}\in x_{i}\CI_{\fc}([0,\tau_{t_0}]\times \tX).
\label{ha.23}\end{equation}
\begin{lemma}
Let $u\in \CI_{\fc}([0,T]\times X)$ be a solution to \eqref{sch.3} and assume that the associated evolution equations \eqref{ha.4} have solutions $u_{I}$ in
$\CI_{\fc}([0,T]\times \tD_I)$.  Then there exists a constant $\tau>0$ depending on
$u$, $\{u_{I}\}$ and $\omega_{0}$, but not on $t_{0}\in [0,T)$, such that  if $u(t_{0},\cdot)\in \CI_{\fc}(\tX)$, then on the time interval $[t_0, t_0 +\tau_{t_0}]$ and for each $i\in\{1,\ldots,\ell\}$, 
\[
      u\in x_{i}\CI_{\fc}([t_{0},t_{0}+\tau]\times X) + \CI_{\fc}([t_{0}, t_{0}+\tau_{t_0}]\times \tX), \quad \tau_{t_0}= \min\{\tau, T-t_0\},
\]
with restriction to $\tD_i$ given by $u_{i}$.
\label{ha.24}\end{lemma}
\begin{proof}
Take $\tau>0$ as above so that \eqref{ha.16} and \eqref{ha.22} hold.  The evolution equation \eqref{ha.23} is then a uniformly parabolic equation for $t\in [0,\tau_{t_0}]$.  Applying Proposition~\ref{de.3}, we conclude that 
\[
       v_{i}\in x_{i}\CI_{\fc}([0,\tau_{t_0}]\times X),
\]
and the result follows by noticing that
\[
  u(t+t_0 ,\cdot)= v_{i}(t,\cdot) + \hu_{\Xi_i}(t,\cdot) + u(t_0 , 0)
\]
with the last two terms in $\CI_{\fc}([0,\tau_{t_0}]\times \tX)$ by assumption.
\end{proof}
The previous lemma gives us, for $t\in[0,\tau_{t_0}]$,  the first term in the Taylor series of $\hu(t,\cdot)$ at each boundary hypersurface.  Proceeding recursively, we will now construct all the higher order terms in the Taylor series of $\hu$ at each boundary hypersurface.  To this end, we need to consider the family 
$t\mapsto L^{\gamma}_{i,t}\in\Diff^{2}_{\fc}(X)$ defined for $\gamma= (\gamma_{1},\ldots, \gamma_{\ell})\in \bbN_{0}^{\ell}$ to be
\begin{equation}
L_{i,t}^{\gamma} f= (x^\gamma) ^{-1}L_{i,t} \left(
  (x^\gamma)f\right), \quad f\in \CI_{\fc}(X), \;\;
  x^\gamma = \prod_{j=1}^{\ell} x_{j}^{\gamma_{j}}.
  \label{ha.25}\end{equation}
As can be seen by a direct computation,
\begin{equation}
   \sigma_{2}(L_{i,t}^{\gamma})= \sigma_{2}(L_{i,t}),
\label{ha.26}\end{equation}
so that the family $L^{\gamma}_{i,t}$ is uniformly elliptic for $t\in [0,\tau_{t_0}]$.  Another important feature for our proof by induction is that the family of operators $L^{\gamma}_{i,t}$ has the same regularity at the boundary as the one of $\hu$, namely, for any $\beta\in \bbN_{0}^\ell$, we have that 
\begin{equation}
  \hu(t,\cdot)\in x^\beta  \CI_{\fc}(X)+ \CI_{\fc}(\tX)\quad \Longrightarrow \quad L_{i,t}^{\gamma} \in  x^\beta  \Diff^{2}_{\fc}(X)+ \Diff^{2}_{\fc}(\tX).
  \label{ha.27}\end{equation}  
In particular, Lemma~\ref{ha.24} also implies that
\begin{equation}
  L^{\gamma}_{i,t}\in x_{j}\Diff^{2}_{\fc}(X)+ \Diff_{\fc}^{2}(\tX)
\label{ha.28}\end{equation}
for $t\in [0,\tau_{t_0}]$ and $j\in \{1,\ldots,\ell\}$, and consequently, that the restriction
of $L^{\gamma}_{i,t}$ to $\tD_I$ is well-defined, namely, there exists a unique family of operators $t\mapsto L^{\gamma}_{i,I,t}\in \Diff^{2}_{\fc}(\tD_I)$ such that
for $f\in \CI_{\fc}(\tX)$ with restriction $f_{I}$ to $\tD_I$, we have (\cf Proposition~\ref{ext.8}),
\begin{equation}
  \left.\left(L^{\gamma}_{i,t} f\right)\right|_{\tD_{I}}= L^{\gamma}_{i,I,t} f_{I}.
\label{ha.29}\end{equation}
In the particular case where $\dim \tD_I=0$, the operator $L^{\gamma}_{i,I,t}$ is just multiplication by a function, but otherwise, the family of operators
$t\mapsto L^{\gamma}_{i,I,t}$ is clearly uniformly elliptic.   
\begin{proposition}
Let $u\in \CI_{\fc}([0,T]\times X)$ be a solution to \eqref{sch.3} and assume that for each $I$, the associated evolution equation \eqref{ha.4} has a solution 
$u_{I}$  in $\CI_{\fc}([0,T]\times \tD_I)$.  Then there exists a constant $\tau>0$ depending on $u$, $u_{I}$ and $\omega_{0}$, but not on $t_0 \in [0,T)$, such that if $u(t_{0},\cdot)\in \CI_{\fc}(\tX)$, then,
\[
       u\in \CI_{\fc}([t_0, t_0 + \tau_{t_0}]\times \tX), \quad \tau_{t_0}= \min\{\tau, T-t_0\},
\]
with restriction to $\tD_I$ given by $u_I$.  
\label{ha.30}\end{proposition}
\begin{proof}
We choose $\tau>0$ to be as in Lemma~\ref{ha.24}.  By this lemma, we then know that 
\[
    \hu\in x_{i}\CI_{\fc}([0,\tau_{t_0}]\times X)+ \CI_{\fc}([0,\tau_{t_0}]\times\tX)
\]
with restriction to $\tD_i $ given by $\hu_{i}$.  Since $\hu-v_i= u_{\Xi_i}\in \CI_{\fc}([0,\tau_{t_0}]\times \tX)$, it suffices to show that $v_i\in \CI_{\fc}([0,\tau_{t_0}]\times \tX)$, namely, that for all $\gamma\in \bbN_{0}^{\ell}$, there exists
$w_{i,\gamma}\in \CI_{\fc}([0,\tau_{t_0}]\times \tX)$ such that
\begin{equation}
   (v_{i}- w_{i,\gamma})\in x^{\gamma}\CI_{\fc}([0,\tau_{t_0}]\times X), \quad \mbox{where} \; x^{\gamma}= \prod_{j=1}^{\ell} x_{j}^{\gamma_{j}}. 
\label{ha.31}\end{equation}
To be able to proceed recursively, we will pair the statement \eqref{ha.31} with the requirement that the function 
\[
v_{i,\gamma}= \frac{ v_{i}- w_{i,\gamma}}{x^{\gamma}}
\]
obeys an evolution equation of the form
\begin{equation}
  \frac{\pa v_{i,\gamma}}{\pa t} = L^{\gamma}_{i,t} v_{i,\gamma}- v_{i,\gamma} + E_i^\gamma,
\label{ha.32}\end{equation}
with $E_{i}^{\gamma}$ having its regularity at the boundary determined by the one of $\hu$ as follows. For any $\alpha= \gamma+\beta$, we require that
\begin{multline}
 \hu\in x^{\alpha}\CI_{\fc}([0,\tau_{t_0}]\times X)+ \CI_{\fc}([0,\tau_{t_0}]\times \tX) \;\Longrightarrow \;  \\  E_{i}^{\gamma}\in x^\beta \CI_{\fc}([0,\tau_{t_0}]\times X) + \CI_{\fc}([0,\tau_{t_0}]\times \tX).
\label{ha.32b}\end{multline}   

Since the evolution equation of $v_{i}$ is given by \eqref{ha.23}, the statements \eqref{ha.31} and \eqref{ha.32} are trivially satisfied when $|\gamma|=0$ by taking $w_{i,\gamma} =0$.  Proceeding by induction on $|\gamma|$, assume therefore that \eqref{ha.31} and \eqref{ha.32} hold for all $i\in \{1,\ldots,\ell\}$ whenever $|\gamma|\le k$.  We want to show that these statements also hold for $|\gamma |=k+1$.     Given $\alpha\in \bbN^{\ell}_{0}$ with $|\alpha|=k+1$, write it as $\alpha= \gamma+\beta$ for some $|\gamma|=k$ and $|\beta|=1$, say $\beta_{j}= 1$ and $\beta_{m}=0$ whenever $m\ne j$.  By our induction hypothesis, 
\[
    \hu\in x^{\nu}\CI_{\fc}([0,\tau_{t_0}]\times X) + \CI_{\fc}(\tX), \quad \forall\, |\nu|=k+1.  
\]
If we also fix $i\in \{1,\ldots,\ell\}$, then this means the term $E_{i}^{\gamma}$ in the evolution equation of $v_{i,\gamma}$ is such that 
\[
     E_{i}^{\gamma} \in x_{j}\CI_{\fc}([0,\tau_{t_0}]\times X)+ \CI_{\fc}([0,\tau_{t_0}]\times \tX).
\]
Consequently, the evolution equation of $v_{i,\gamma}$ can be restricted to $\tD_j$ to give
\begin{equation}
    \frac{\pa f}{\pa t}= L^{\gamma}_{i,j,t}f -f + \left. E_{i}^{\gamma}\right|_{\tD_j}, \quad f(0,\cdot)=0,
\label{ha.33}\end{equation}
where $L^{\gamma}_{i,j,t}$ is the restriction of $L^{\gamma}_{i,t}$ to $\tD_j$.  
By Theorem~\ref{de.18}, this equation has a unique solution $f\in \CI_{\fc}([0,\tau]\times \tD_j)$ when $\dim \tD_j>0$.  When $\dim \tD_j=0$, this is still true since the evolution equation \eqref{ha.33} is just an ordinary differential equation and as such, it has a unique solution $f\in \CI([0,\tau]\times \bD_j)= \CI_{\fc}([0,\tau]\times \tD_j)$.   Consider then the function 
\begin{equation}
  f_{\Xi_j}= \Xi_j (f)\in \CI_{\fc}([0,\tau]\times \tX),
\label{ha.34}\end{equation}
defined using the extension map $\Xi_j$ of \eqref{ext.7}.
It satisfies the evolution equation
\begin{equation}
  \frac{\pa f_{\Xi_j}}{\pa t}= L^{\gamma}_{i,t}f_{\Xi_j} - f_{\Xi_j} + E_i^\gamma + Q_i^\gamma,
\label{ha.35}\end{equation}
with
\begin{equation}
Q_{i}^\gamma = \Xi_i( L^{\gamma}_{i,j,t}f-f+ \left.E_I^{\gamma}\right|_{\tD_j} ) - \left( L^{\gamma}_{i,t} f_{\Xi_j} -f_{\Xi_j} + E_{i}^\gamma \right) 
\label{ha.36}\end{equation}
in $x_{j}\CI_{\fc}([0,\tau_{t_0}]\times X)$  having the same regularity at the boundary as $E_i^{\gamma}$.  Consider the function
\[
      \bv_{i,\alpha}= v_{i,\gamma}- f_{\Xi_j}
\]
satisfying  the evolution equation
\[
     \frac{\pa \bv_{i,\alpha}}{\pa t}= L^{\gamma}_{i,t} \bv_{i,\alpha} - \bv_{i,\alpha} - Q_{i}^{\gamma}.
\] 
By Proposition~\ref{de.3}, we have that $\bv_{i,\alpha}\in x_j \CI_{\fc}([0,\tau]\times X)$.  This suggests to define $v_{i,\alpha}$ by
\begin{equation}
  v_{i,\alpha}= \frac{\bv_{i,\alpha}}{x_j}.
\label{ha.37}\end{equation} 
It then satisfy the evolution equation
\begin{equation}
  \frac{\pa v_{i,\alpha}}{\pa t} = L^{\alpha}_{i,t} v_{i,\alpha} - v_{i,\alpha} + E^\alpha_i,
    \quad E_i^\alpha= - \frac{Q_i^\gamma}{x_j}\in \CI_{\fc}([0,\tau]\times X),
\label{ha.38}\end{equation}
with $E_{i}^{\alpha}$ of the claimed regularity \eqref{ha.32b}.  Thus, it suffices to take 
\[
    w_{i,\alpha}= w_{i,\gamma}- f_{\Xi_j} x^\gamma.
\]
Since $i\in\{1,\ldots, \ell\}$ and $|\alpha|=k+1$ were arbitrary, this completes the proof by induction.
\end{proof}

We can now apply this result to prove the theorem.

\begin{proof}[Proof of Theorem~\ref{ha.1}]
Let $0<T_{0}\le T$ be  such that the evolution equation \eqref{ha.4} has a solution $u_I\in \CI_{\fc}([0,T_0]\times D_I)$ for all $I\subset \{1,\ldots,\ell\}$ such that $\bD_{I}= \bigcap_{i\in I} \bD_i\ne \emptyset$.  Since the result is trivial when $\tX$ is a closed manifold, we can proceed by induction on the depth of $\tX$ to prove Theorem~\ref{ha.1} and assume that $u_I\in \CI_{\fc}([0,T_0]\times \tD_I)$.         
             
Applying  Proposition~\ref{ha.30} with $t_{0}=0$, we then know that on the time interval $[0,\tau]$ (or $[0,T_0]$ if $\tau>T_0$), we have $u\in \CI_{\fc}([0,\tau]\times \tX)$.  Repeating this argument with $t_0= \tau, 2\tau,3\tau,$ et cetera, we can extend the result to the full interval to conclude that $u\in \CI_{\fc}([0,T_0]\times \tX)$.  To complete the proof, we need to show we can take $T_0=T$.  This follows from \cite[Remark~7.4]{Lott-Zhang}.  Alternatively, we can argue by contradiction and suppose that we cannot take $T_0=T$.  This means that the curvature tensor of one of the metrics $\omega_{I,t}$ blows up as $t$ tends to  $T_{I}$ for some $T_{I}\le T$.  Since for $t<T_{I}$, $\widetilde{\omega}_t$ is in $\CI_{\fc}(\tX; \Lambda^{2}({}^{\fc}T^{*}\tX))$ and its restriction to $\tD_{I}$ is $\widetilde{\omega}_{I}$, this means by Remark~\ref{Rm_asymp} that the curvature tensor of $\widetilde{\omega}_{t}$ also blows up as $t$ tends to $T_I$, a contradiction.  Thus, we can indeed take $T_0=T$.  
\end{proof}

\section{Asymptotics of K\"ahler-Einstein metrics on quasiprojective manifolds} \label{ake.0}

To find examples of K\"ahler-Einstein metrics on the quasiprojective manifold $X$, a natural condition to impose on the divisor $\bD$ is that 
\begin{equation}
      K_{\bX} + [\bD] >0.
\label{ake.1}\end{equation}
In this case, we can take $L= K_{\bX}+[\bD]$ to be the positive line bundle used to define the asymptotically tame polyfibred cusp K\"ahler metric of \eqref{fc.14}.  A Hermitian metric on $L$ is specified by the Hermitian metrics $\|\cdot \|_{\bD_i}$ and a choice of volume form $\Omega$ on $\bX$.  The condition \eqref{ake.1} ensures that we can choose $\Omega$ such that 
\begin{equation}
  \sqrt{-1} \Theta_{K_{\bX}+ [\bD]} = -\sqrt{1}\ \db \pa \log \left(  \frac{\Omega}{\prod_{i=1}^{\ell} \| s_i\|^{2}_{\bD_i}}\right) >0, 
\label{ake.2}\end{equation}
in which case 
\begin{equation}
\begin{aligned}
   \omega &= -\sqrt{-1} \ \db \pa \log \left(  \frac{\Omega}{\prod_{i=1}^{\ell} \| s_i\|^{2}_{\bD_i}}\right) +
   \sqrt{-1} \ \db\pa \log \left( \prod_{i=1}^{\ell} (-\log \epsilon \|s_i \|^{2}_{\bD_i})^{2}\right) \\
   &=  \sqrt{-1} \Theta_{K_{\bX}+[\bD]} + 2\sqrt{-1} \sum_{i=1}^{\ell} \left( \frac{\Theta_{\bD_i}}{ \log \epsilon \| s_i\|^{2}_{\bD_i}} \right) \\
   & \quad \quad \quad + 2\sqrt{-1} \sum_{i=1}^{\ell} 
     \left(  \frac{(\pa \log \epsilon \|s_i\|^{2}_{\bD_i}) \wedge (\db   \log \epsilon \|s_i\|^{2}_{\bD_i})}{  (\log \epsilon \|s_i\|^{2}_{\bD_i})^{2}} \right)
      \end{aligned}
\label{ake.3}\end{equation}
is a K\"ahler form provided $\epsilon$ is chosen small enough.  A K\"ahler-Einstein metric of the form 
\begin{equation}
     \omega_{\KE}= \omega + \sqrt{-1} \pa\db u
\label{ake.4}\end{equation}
is then obtained by solving the Monge-Amp\`ere equation
\begin{equation}
    \log \left( \frac{ (\omega+ \sqrt{-1} \pa \db u)^{n} }{ \omega^n } \right) -u = F= \log\left( \frac{\widetilde{\Omega}}{\omega^n}\right) \in \CI_{\fc}(\tX),
\label{ake.5}\end{equation}
where $\widetilde{\Omega}$ is the volume form
\begin{equation}
  \widetilde{\Omega}=  \frac{\Omega}{\prod_{i=1}^{\ell} \left( \|s_i\|^{2}_{\bD_i} (-\log (\epsilon \|s_i\|^{2}_{\bD_i}))^{2}\right) }  \in \CI_{\fc}(\tX; \Lambda^{2n}({}^{\fc}T^{*}\tX)). 
\label{ake.6}\end{equation}

From the work of Yau \cite{Yau1978}, Cheng and Yau \cite{Cheng-Yau}, Kobayashi \cite{Kobayashi}, Tsuji \cite{Tsuji}, Tian and Yau \cite{Tian-Yau} and Bando \cite{Bando}, we know  that the Monge-Amp\`ere equation \eqref{ake.6} has a unique solution $u\in \CI_{\fc}(X)$.  

By the adjunction formula, notice that for each $i\in\{1,\ldots,\ell\}$,
\begin{equation}
\left. \left( K_{\bX}+ [\bD] \right)\right|_{\bD_i}= K_{\bD_i}+ \sum_{j\ne i} \left.[\bD_j]\right|_{\bD_i}>0.
\label{ake.7}\end{equation}
If $\omega$ has standard spatial asymptotics associated to $\{\omega_I\}$ and $\{c_i\}_{i=1}^{\ell}$, where in this case $c_i=1$ for all $i$,  this means we can repeat the same procedure to find a complete K\"ahler-Einstein metric on $D_i$, and more generally on $D_I$.  Namely, we get a K\"ahler-Einstein metric on $D_I$ of the form
\begin{equation}
   \omega_{\KE,I} = \omega_I +\sqrt{-1} \pa \db u_I
\label{ake.8}\end{equation}
where $u_{I}\in \CI_{\fc}(D_I)$ is the unique solution to the Monge-Amp\`ere equation
\begin{equation}
 \log \left( \frac{ (\omega_I+ \sqrt{-1} \pa \db u_I)^{n-|I|} }{ \omega_I^{n-|I|} } \right) -u = F_I \in \CI_{\fc}(\tD_I),
 \label{ake.9}\end{equation}
where $F_I$ is the restriction of $F$ to $\tD_I$.

When the divisor $\bD$ is smooth, that is, when there is no normal crossing,  Schumacher showed in \cite{Schumacher} that as one approaches the irreducible component $\bD_i$, the solution $u$ to \eqref{ake.5} restricts on $\bD_i$ to give $u_i$.  When $\bD$ has normal crossings, the logarithmic compactification $\tX$ provides the right framework to generalize Schumacher's result.  To formulate this generalization, notice that proceeding as in \eqref{ext.7}, we can define an extension map 
\begin{equation}
  \Xi_{I,i}: \CI_{\fc}(D_{I\cup \{i\}} ) \to \CI(D_I)
\label{extn.1}\end{equation}
for $I\subset \{1,\ldots, \ell\}$ with $\bD_I\neq \emptyset$ and $i\notin I$ with $\bD_i\cap \bD_I \neq \emptyset$.

\begin{theorem}
Fix a K\"ahler form $\omega$ as in \eqref{ake.3}.  Then there exists $\nu>0$ such that for
each $i\in \{1,\dots,\ell\}$,  
   \[
                        u-\Xi_i( u_i) \in x_i^{\nu}\cC^{\infty}_{\fc}(X),
   \]
   where $u$ and $u_i$ are the solutions to the Monge-Amp\`ere equations \eqref{ake.5} and 
   \eqref{ake.9}.   Moreover, the constant $\nu>0$ can be chosen such that for all $I\subset \{1,\ldots,\ell\}$ with $\bD_I\neq \emptyset$ and for all $i\notin I$ with $\bD_i\cap \bD_I \neq \emptyset$, 
   \[
         u_I - \Xi_{I,i}( u_{I\cup \{i\}}) \in x_i^{\nu} \CI_{\fc}(D_I).
   \]
In particular, for all $I\subset \{1,\dots,\ell\}$ with $\bD_I\neq \emptyset$,  the restriction of $u$ to $D_I$ is well-defined and is given by $u_I$.     
\label{ake.10}\end{theorem}

When $\tX$ has depth zero, the theorem is trivial.  Thus, to prove this theorem, we can proceed by induction on the depth of $\tX$  and assume the theorem holds on $D_I$ for all $I$.  With this assumption, we can proceed recursively on $i\in\{1,\ldots,\ell\}$ to define
\[
       v_1= \Xi_1(u_1), \quad v_{k+1}= v_k + \Xi_{k+1} ( u_{k+1} - \left. v_k \right|_{D_{k+1}} ). 
\]
Then $v_{\ell}\in \CI_{\fc}(X)$ is such that its restriction to $D_I$ is well-defined for all $I$ and given by $u_I$.  Moreover, there exists $\nu>0$ such that 
\begin{equation}
      v_{\ell} - \Xi_{i}(u_i) \in x_i^{\nu} \CI_{\fc}(X) \quad \forall \ i \in \{1,\ldots,\ell\}.  
\label{blabla.1a}\end{equation}
Thus, to show the theorem holds on $X$, it suffices to show that for possibly a smaller $\nu>0$, we have that for all $i\in\{1,\ldots,\ell\}$, 
\begin{equation}
     u-v_{\ell} \in x_i^{\nu}\CI_{\fc}(X).  
\label{blabla.1}\end{equation}
To use the Monge-Amp\`ere equation, we will replace $v_{\ell}$ with a slightly better function.   

\begin{lemma}
There exists a function $v\in \CI_{\fc}(X)$ such that:
\begin{itemize}
\item[(i)] $v- v_{\ell} \in \frac{\rho^\frac{2}{m}}{x^2}\CI_{\fc}(X)$ for some $m\in \bbN$, where $\rho= \prod_{i=1}^{\ell}\rho_i$; 
\item[(ii)]  there exists a constant $K>0$ such that
\[
        \frac{\omega}{K} \le \omega + \sqrt{-1} \pa \db v \le K\omega \quad \mbox{on} \; X.
\]
\end{itemize}
\label{gsr.1}\end{lemma}  
\begin{proof}
Let $h_{L}$ denote the Hermitian metric on $L=K_{\bX}+\bD$ induced by the choice of volume form $\Omega$ and the Hermitian metrics $\|\cdot\|_{\bD_i}$.  On $X$, we can consider instead a new Hermitian metric given by
\[
      \widetilde{h}_L = \left(\prod_{i=1}^{\ell} (-\log \epsilon \|s_i\|^{2}_{\bD_i})^2 \right) h_L,
\]
and \eqref{ake.3} means that $\omega= \sqrt{-1} \Theta_{\widetilde{h}_L}$, where $\Theta_{\widetilde{h}_L}$ is the curvature of the Chern connection associated to $\widetilde{h}_L$.  

Similarly, the form $\omega+\sqrt{-1}\pa\db v_{\ell}$ is
simply given by $\sqrt{-1}\Theta_{\widetilde{h}_1}$ where $\Theta_{\widetilde{h}_1}$ is the curvature of the Chern connection associated to the Hermitian metric $\widetilde{h}_1= e^{-v_{\ell}}\widetilde{h}_L$.  This curvature is not necessarily everywhere positive on $X$.  However, in local holomorphic coordinates \eqref{fc.16} near $\bD_I$, we know from \eqref{blabla.1a} that
\[
     (\omega+ \sqrt{-1}\pa\db v_{\ell}) -\left( \omega_{\KE,I} + \frac{\sqrt{-1}}{2}\sum_{i\in I} \frac{d\zeta_i\wedge d\overline{\zeta}_i }{|\zeta_i|^2 (\log |\zeta_i |)^2}   \right) \in \sum_{i\in I}x_i^{\nu}\CI_{\fc}(X;\Lambda^2(T^* X)\otimes \bbC).  
\]   

This means there is a positive constant $C>0$  and small open neighborhood neighborhood $\overline{\cU}\subset \bX$ of $\bD$ such that in $\cU= \overline{\cU}\cap X$, we have that 
\[
          \frac{\omega}{C} \le  \omega + \sqrt{-1} \pa\db v_{\ell} \le C\omega.  
\]     

On the other hand, since the line bundle $L$ is positive, we know there exists $m>0$ and a Hermitian metric $h_{L^m}$ on $L^m$ such that both $L^m$ and $L^m- [\bD]$ are positive Hermitian line bundles.  On $\bX\setminus \bD$, we can then consider the new Hermitian metric given by 
\begin{equation}
     \widehat{h}_{L^m}=  \frac{h_{L^m}}{\|s\|_{\bD}^{2}},  \quad \mbox{with}\; s= \bigotimes_{i=1}^{\ell} s_i \in H^{0}(\bX; [\bD]).
\label{gsr.2}\end{equation}
Its curvature is positive on $\bX\setminus \bD$, since
\[
     \sqrt{-1}\Theta_{\widehat{h}_{L^m}}= \sqrt{-1}\Theta_{h_{L^m}} - \sqrt{-1}\Theta_{\bD}= \sqrt{-1} \Theta_{L^m-[\bD]} >0.  
\]
If we denote by $h_2= (\widehat{h}_{L^m})^{\frac{1}{m}}$ the induced Hermitian metric on $L$, we also have that $\sqrt{-1}\Theta_{h_2}= \frac{\sqrt{-1}}{m}\Theta_{\widehat{h}_{L^m}}$ is positive on $\bX\setminus \bD$.       
Thanks to the holomorphic section $s$ in \eqref{gsr.2}, notice that
\begin{equation}
     \frac{\widetilde{h}_1}{h_2} = \frac{e^{-v_{\ell}} h_L \prod_{i=1}^{\ell} (-\log \epsilon \|s_i \|^{2}_{\bD_i})^{2}}{h_2} \in \frac{\rho^{\frac{2}{m}}}{x^2} \CI_{\fc}(X).
\label{gsr.3}\end{equation}

For $c>0$, consider then on $X$ the Hermitian metric on $L$ given by
\[
           \widetilde{h}= \frac{1}{\frac{1}{\widetilde{h}_1}+ \frac{1}{c h_2} }= \frac{\widetilde{h}_1}{ 1+ \frac{\widetilde{h}_1}{c \widetilde{h}_2} }.
\]
From \eqref{gsr.3}, we have that 
\begin{equation}
      \frac{\widetilde{h}}{\widetilde{h}_1}-1  \in \frac{\rho^\frac{2}{m}}{x^2} \CI_{\fc}(X).  
\label{gsr.4}\end{equation}
By \cite[Lemma~3]{Schumacher}, we know that 
\[
       \frac{\sqrt{-1}\Theta_{\widetilde{h}} }{\widetilde{h}} \ge \frac{\sqrt{-1}\Theta_{\widetilde{h}_1}}{\widetilde{h}_1} + \frac{\sqrt{-1}\Theta_{h_2}}{c h_2}.
\]
Since $\sqrt{-1}\Theta_{\widetilde{h}_1}\ge \frac{\omega}{C}$   on $\cU$ and $\sqrt{-1}\Theta_{h_2}>0$ on $\bX\setminus \bD$, by taking $c>0$ sufficiently small, we can ensure that for some positive constant $K$, 
\begin{equation}
       \frac{\omega}{K} <\sqrt{-1}\Theta_{\widetilde{h}} < K\omega.
\label{gsr.4b}\end{equation}
Thus, in light of \eqref{gsr.4b}, it suffices to take $v= v_{\ell} - \log\left( \frac{\widetilde{h}}{\widetilde{h}_{1}}\right)$, for then 
$\omega+ \sqrt{-1}\pa\db v= \sqrt{-1}\Theta_{\widetilde{h}}$ and 
\[
    v-v_\ell = -  \log \left( \frac{\widetilde{h}}{\widetilde{h}_1}\right)\in \frac{\rho^\frac{2}{m}}{x^2} \CI_{\fc}(X).
\]

  \end{proof}

  With this lemma, we can then proceed to the proof of Theorem~\ref{ake.10}.  
  
  \begin{proof}[Proof of Theorem~\ref{ake.10}]
  In light of Lemma~\ref{gsr.1}, we need to show that for some $\delta>0$ and for all $i\in \{1,\ldots,\ell\}$,  $w= u-v$ is in  $x_i^{\delta}\CI_{\fc}(X)$.  In terms of $w$ and the K\"ahler form $\omega_v= \omega +\sqrt{-1}\pa\db v$, the Monge-Amp\`ere equation \eqref{ake.5} can be rewritten as
 \begin{equation}
   \log \left( \frac{(\omega_v+ \sqrt{-1} \pa\db w)^n }{\omega_v^n}\right)- w= F_v
 \label{gsr.6}\end{equation}
  with 
  \[
     F_v=  F -\left(\log\left(  \frac{\omega_v^n}{\omega^n}\right)  -v\right).
  \]
  From Lemma~\ref{gsr.1} and \eqref{blabla.1}, we see that for some $\nu>0$, we have that $v-\Xi_i(u_i)\in x_i^{\nu}\CI_{\fc}(X)$ for all $i\in\{1,\ldots,\ell\}$.  In particular, this means that 
\[
F_v= \left( \log \left( \frac{(\omega+\sqrt{-1}\pa\db u)^n}{\omega^n}\right)-u   \right)  -   \left( \log \left( \frac{(\omega+\sqrt{-1}\pa\db v)^n}{\omega^n}\right)-v   \right)     
\] 
is in $x_{i}^{\nu}\CI_{\fc}(X)$.  On the other hand, writing  $\omega_{v,t}= \omega_v + t\sqrt{-1} \pa\db w$, 
 the Monge-Amp\`ere equation \eqref{gsr.6} can be rewritten as
\begin{equation}
\begin{aligned}
 F_v +w &=  \int_0^1  \frac{\pa}{\pa t} \log \left( \frac{ (\omega_v+ t\sqrt{-1}\pa\db w)^n}{\omega_v^n}\right) dt \\
   & = \int_0^1 \left( \frac{n\omega_{v,t}^{n-1} \wedge \sqrt{-1} \pa\db w}{\omega_{v,t}^{n}}\right) dt,
\end{aligned}    
\label{gsr.7}\end{equation}  
 that is, it can be written as a linear equation
 \[
      (\Delta_{w} -1)w= F_v
 \] 
where
\[
      \Delta_{w} f=  \int_0^1 \left( \frac{n\omega_{v,t}^{n-1} \wedge \sqrt{-1} \pa\db f}{\omega_{v,t}^{n}}\right) dt = \int_0^1 \Delta_{\omega_{v,t}} f dt,
\]    
 with $\Delta_{\omega_{v,t}}$ the $\db$-Laplacian associated to $\omega_{v,t}$.  This tacitly assumes that $\omega_{v,t}$ is a K\"ahler form.  Since we can find a positive constant $C$ such that 
 \[
      \frac{\omega_v}{C} \le \omega_v+\sqrt{-1} \pa\db w \le C\omega_v,
 \]  
this is indeed the case, as we have,
 \[
      \left( \frac{t}{C} + (1-t) \right) \omega_v \le \omega_{v,t} \le (Ct+1-t)\omega_v, \quad \forall \; t\in [0,1]. 
 \]
 In particular, the operator $\Delta_w \in \Diff^{2}_{\fc}(X)$ is uniformly elliptic.  By Yau's generalized maximum principle and Schauder's theory, this means that for all $k\in \bbN_0$,  the operator $(\Delta_w -1)$ induces an isomorphism of Banach spaces
 \[
         (\Delta_w -1): \cC^{k+2,\alpha}_{\fc}(X)\to \cC^{k,\alpha}_{\fc}(X),
 \] 
 where $\cC^{k,\alpha}_{\fc}(X)$ is the H\"older space considered  in \cite{Kobayashi} (see also
 \cite{Wu06}).
 Provided $\delta>0$ is chosen small enough, we will show that it also induces an isomorphism
 \begin{equation}
 (\Delta_w -1): x_i^{\delta}\cC^{k+2,\alpha}_{\fc}(X)\to x_i^{\delta}\cC^{k,\alpha}_{\fc}(X)
 \label{gsr.7b}\end{equation}
 for all $k\in \bbN_0$.  Since $F_v\in x_i^{\delta}\CI_{\fc}(X)$ for $\delta>0$ small enough, this will imply that $w\in x_i^\delta\CI_{\fc}(X)$ as desired.  
 
 To find $\delta$, consider the new operator
 \begin{equation}
      x_i^{-\delta}\circ (\Delta_w -1)\circ x_i^{\delta}= \Delta_{w}-1 + x_i^{-\delta}[ \Delta_w , x_i^{\delta}].
 \label{gsr.8}\end{equation}
 Since the family $\delta\mapsto x^{-\delta} [\Delta_{w},x^{\delta}]\in \Diff_{\fc}^{2}(X)$ is continuous and vanishes for $\delta=0$, we see that for $\delta>0$ sufficiently small and for all $k\in \bbN_0$, the operator  \eqref{gsr.8} induces an isomorphism
 \[
 x_i^{-\delta}\circ (\Delta_w -1)\circ x_i^{\delta}:  \cC^{k+2,\alpha}_{\fc}(X)\to \cC^{k,\alpha}_{\fc}(X),  \]
 which is just another way of saying the linear map \eqref{gsr.7b} is an isomorphism.   
  \end{proof}

\section{Boundary regularity for linear uniformly elliptic equations} \label{br.0}

In this section, we will assume that the divisor $\bD$ is smooth, so that its irreducible components $\bD_1,\ldots,\bD_\ell$ do not intersect.  On the manifold with boundary $ \hX$, we consider for each boundary component $\hH_i$ a collar neighborhood 
\begin{equation}
  \hat{c}_i : \hH_i\times [0,\delta_i)\to \hX
\label{br.1}\end{equation} 
compatible with the boundary defining function $\rho_i\in \CI(\hX)$, that is, such that 
\[
      \rho_i \circ \hat{c}_i (p,t)=t, \quad \forall \; p\in \hH_i, \; t\in [0,\delta_i).
\] 
On $\tX$, this lifts to a collar neighborhood
\begin{equation}
       \tc_i : \tH_i \times [0,\epsilon_i)\to \tX, \quad \epsilon_i= \frac{-1}{\log\delta_i},
\label{br.2}\end{equation}
compatible with the boundary defining function $x_i= \frac{-1}{\log \rho_i}$.  Combining these, we get a collar neighborhood 
\begin{equation}
   \tc: \pa\tX\times [0,\epsilon) \to \tX, \quad \epsilon= \min \{\epsilon_1,\ldots,\epsilon_\ell\}. 
\label{br.3}\end{equation}
On $\tX$, consider the space $\dot{\cC}^{\infty}(\tX)$ of smooth functions on $\tX$ which vanish together with all their derivatives on the boundary $\pa\tX$.  In particular, we have that
\[
      \dot{\cC}^{\infty}(\tX)= \bigcap_{m\in\bbN} x^m \cC^{\infty}_{\fc}(X)\subset \cC^{\infty}_{\fc}(X).
\]
On the other hand, let $\tPhi: \pa \tX\to \bD$ denote the fibration which is induced by the fibration $\tPhi_i$ on the boundary component $\tH_i$.  In the collar neighborhood \eqref{br.3}, the fibration
$\tPhi:\pa\tX\to \bD$, extends to a fibration 
\begin{equation}
    \begin{array}{lccc} 
    \tPhi\times \Id: & \pa\tX\times [0,\epsilon) & \to & \bD\times [0,\epsilon) \\
                            & (p,t)  & \mapsto & (\tPhi(p),t).
    \end{array}                        
\label{br.4}\end{equation}
With respect to this fibration, another subspace of $\CI_{\fc}(X)$ of interest is
\begin{equation}
  \CI_{\tc}(X)= \{ f\in \CI_{\fc}(X) \quad | \quad \tc^{*}f = (\tPhi\times \Id)^{*}h 
  \; \mbox{for some} \; h\in \CI(\bD\times (0,\epsilon)) \}.
\label{br.5}\end{equation}

\begin{lemma}
Any function $f\in \cC^{\infty}_{\fc}(X)$ can be written as 
\[
         f= f_1 + f_2 \quad \mbox{for some} \; f_1\in \CI_{\tc}(X), \; f_2\in \dot{\cC}^{\infty}(\tX).
\]
\label{br.6}\end{lemma}
\begin{proof}
Let $f\in \CI_{\fc}(X)$ be given.  Using a cut-off function, we can assume without loss of generality that $f$ has its support in the collar neighborhood $\tX\times (0,\epsilon)$.  The fibration $\tPhi\times \Id$ is a circle bundle, which means in particular there is an underlying smooth action of $\bbS^1$ on $\pa\tX\times (0,\epsilon)$.  Let $f_{\av}$ be the average of $f$ with respect to this group action,    
\[
          f_{\av} = \frac{1}{2 \pi} \int_{\bbS^1} (\theta^* f) d\theta, \quad \theta\in \bbS^1.
\]
By construction, $f_{\av}$ is constant along the fibres of $\tPhi\times \Id$ and is an element of $\CI_{\tc}(X)$.  Thus, if we take $f_1 =f_{\av}$ and $f_2= f-f_{\av}$, it remains to show $f_2 \in \dot{\cC}^{\infty}(\tX)$.  Let $\xi$ be the infinitesimal generator of the circle action.  In terms of the Fourier decomposition of $f_2$ in each fibre of $\tPhi\times \Id$, we see that for all $k\in \bbN_0$,
\begin{equation}
     \frac{1}{2\pi}\int_{\bbS^1} \theta^{*} ( \xi^{k} f_2 ) d\theta =0.
\label{br.7}\end{equation}   
Moreover, since $f_1$ and $f$ are in $\CI_{\fc}(X)$, $f_2$ must also be an element of $\CI_{\fc}(X)$, so that 
\begin{equation}
   \xi^k f_2 \in x^{k}\CI_{\fc}(X) \quad \forall \; k\in \bbN_0, \quad \mbox{where} \; x=\prod_{i=1}^{\ell} x_i .
\label{br.8}\end{equation}
Integrating $k$ times $\xi^k f$ using the property \eqref{br.7}, we find that 
\[
      f_2\in x^{k}\CI_{\fc}(X) \; \forall \; k\in \bbN_0, 
\]
which implies as desired that $f_2 \in \dot{\cC}^{\infty}(\tX)$.  
\end{proof}

Let $\omega$ be an asymptotically tame polyfibred cusp K\"ahler metric with standard spatial asymptotics $\{\omega_i\}_{i=1}^{\ell}$ and $\{c_i\}_{i=1}^{\ell}$.  From the previous lemma, we see that when the $\db$-Laplacian $\Delta_{\omega}$ acts on $\CI_{\fc}(X)$, it is asymptotically modeled near $\tH_i$ by the $b$-differential operator
\begin{equation}
  \Delta_{\omega_i} + \frac{c_i}{2}\left( \left( x_i \frac{\pa}{\pa x_i}\right)^{2} + x_i\frac{\pa}{\pa x_i} \right).
\label{br.9}\end{equation} 
In the terminology of \cite{MelroseAPS}, the corresponding indicial family is given by 
\begin{equation}
  \Delta_{\omega_i} + \frac{c_i}{2}( -\tau^2 + \sqrt{-1} \tau), \quad \tau\in\bbC.
\label{br.10}\end{equation}
More generally, for $\lambda\in \bbC$, we will be interested to study the boundary regularity of the operator $\Delta_{\omega}-\lambda$ with associated indicial family on $\tH_i$ given by
\begin{equation}
 \hat{P}_{i,\lambda}(\tau)=
 \Delta_{\omega_i} + \frac{c_i}{2}( -\tau^2 + \sqrt{-1} \tau)- \lambda.
\label{br.11}\end{equation} 
The set
\begin{equation}
 \Spec_b(\hat{P}_{i,\lambda})= \{ \tau \in \bbC\; | \; \hat{P}_{i,\lambda}(\tau) \; \mbox{is not invertible}\}
\label{br.12}\end{equation}
will be particularly important. 
Following \cite[p.174]{MelroseAPS}, for each $\alpha\in \bbR$, we consider the set 
\begin{multline}
  E_{i,\lambda}^{+}(\alpha)= \{ (z,k)\in \bbC\times \bbN_0 \; | \; 
  \hat{P}_{i,\lambda}(\tau)^{-1} \; \mbox{has a pole at} \; \tau= -\sqrt{-1}z \\
    \mbox{of order $k+1$ and } \; \Re z > \alpha\}.
\label{br.13}\end{multline}
In our case, the poles are all of order $1$ except possibly one pole of order $2$ at $\tau= \frac{\sqrt{-1}}{2}$.  In particular, if $\alpha>-\frac12$, then all the poles are simple.  The set $E_{i,\lambda}^{+}(\alpha)$  can alternatively be described in terms of the spectrum of $\Delta_{\omega_i}$, namely, 
\begin{multline}
  E^{+}_{i,\lambda}(\alpha)= \{ (z,k)\in \bbC\times \bbN_0 \; | \; 
    (z+\frac12)^2 = \frac{2\nu + 2\lambda}{c_i} + \frac{1}{4}\; \mbox{for some} \\
      -\nu\in \Spec(\Delta_{\omega_i}), \;
       \Re z > \alpha, \; k\le 1, \; \mbox{and} \; $k=0$ \; \mbox{if} \; z\ne -\frac12 \}.
\label{br.14}\end{multline} 
The set $E^{+}_{i,\lambda}(\alpha)$ is not an index set, but we can consider the smallest index set containing it.  As explained in \cite[p.186]{MelroseAPS}, to take into account accidental multiplicities, the index set one is ultimately led to consider is slightly bigger and given by
\begin{multline}
  \hE^{+}_{i,\lambda}(\alpha)= \{ (z,k)\in \bbC\times \bbN_0 \; | \; 
  \exists \; r\in \bbN_0,  \; \Re z > \alpha+r, \\
   -\sqrt{-1}(z-r)\in \Spec_b(\hat{P}_{i,\lambda}), \; 
   k+1 \le \sum_{j=0}^{r} \ord(-\sqrt{-1}(z-j)) \},
\label{br.15}\end{multline}
where $\ord(z)$ is the order of the pole of $\hat{P}^{-1}_{i,\lambda}(\tau)$ at $\tau=z$.  
To state our result, we also need to recall from \cite{MelroseAPS} the notion of 
\textbf{extended union} for two index sets $E$ and $F$,
\begin{equation}
  E \overline{\cup} F= E\cup F \cup \{ (z,k) \; | \; \exists \; (z,\ell_1)\in E, \; (z,\ell_2)\in F \; \mbox{such that} \; k=\ell_1+\ell_2 +1\}.
\label{br.16}\end{equation}
Similarly, the extended union of two index families $\cE=(E_1 , \ldots, E_\ell)$ and $\cF=(F_1,\ldots, F_\ell)$ is given by
\begin{equation}
  \cE\overline{\cup} \cF= (E_1\overline{\cup} F_1 , \ldots, E_{\ell}\overline{\cup} F_{\ell}).
\label{br.17}\end{equation}

In the remaining of this section, we will also make use of the $b$-Sobolev spaces of \cite{MelroseAPS}.  
 If $g$ is a $b$-metric with Levi-Civita connection $\nabla$, recall that the $b$-Sobolev space $H^m_b(\bD_i\times [0,\epsilon_i])$  is defined to be the closure of $\CI_{c}(\bD_i\times (0,\epsilon_i))$ with respect to the norm
\begin{equation}
  \| f\|_{H^{m}_{b}}^2 = \sum_{j=0}^{m} \int_{\bD_i\times (0,\epsilon_i)} |\nabla^{j} f |^{2}_{g} \nu_g
\label{br.23}\end{equation}
where $\nu_g$ is the volume density of the metric $g$ and $|\cdot |_{g}$ is the natural norm induced by $g$ on tensors.

Upon identifying $[0,\epsilon_i)$ with $[0,\infty)$ via a diffeomorphism, a simple example of $b$-metric on $\bD_i\times (0,\epsilon_i)$ is given by the product metric
\begin{equation}
  g= g_{\bD_i} + \frac{dx_i^2}{x_i^2},
\label{bm.1}\end{equation}
where $g_{\bD_i}$ is a choice of Riemannian metric on $\bD_i$.  Making the change of variable $t=\log x_i$, this corresponds to the complete cylindrical metric
\begin{equation}
      g_{\bD_i} + dt^2
\label{bm.2}\end{equation}
on $\bD_i\times \bbR$.  

Besides $b$-Sobolev spaces, we will be interested in the space $\cC^{k}_{b}(\bD_i\times (0,\epsilon_i))$ of continuous functions $f$ such that 
$\nabla^{j}f\in \cC^0(\bD_i\times (0,\epsilon_i); T_j^0 (\bD_i\times (0,\epsilon_i)))$ with
\[
         \| f\|_k = \sum_{j=0}^{k} \sup_{\bD_i\times (0,\epsilon_i)} |\nabla^j f|_g <\infty.
\]
It is a Banach space with norm $\|\cdot \|_k$.  There is a corresponding Fr\'echet space obtained by taking the intersection over all these spaces,
\[
   \CI_{b}(\bD_i\times (0,\epsilon_i))= \bigcap_{k\in \bbN} \cC^{k}_{b} (\bD_i\times (0,\epsilon_i)).
\]
Using the product metric \eqref{bm.2} and the standard Sobolev embedding for the Euclidean space, it is not hard to deduce a corresponding Sobolev embedding for $b$-metrics, for instance that we have a continuous inclusion
\begin{equation}
       H^{m+n}_{b}(\bD_i\times [0,\epsilon_i]) \subset \cC^m_b(\bD_i\times (0,\epsilon_i)), \quad n=\dim_{\bbC} \bX.
\label{bm.3}\end{equation}
In particular, this means that 
\begin{equation}
      \bigcap_{m\in \bbN} H^m_b (\bD_i\times [0,\epsilon_i]) = \cC^{\infty}_b (\bD_i\times (0,\epsilon_i)).
\label{bm.4}\end{equation}

We are now ready to state our first regularity result.  Recall that the space of polyhomogeneous functions $\cA^{\cF}_{\fc}(\tX)$ associated to an index family $\cF$ was introduced in \S~\ref{fc.0}.  We will also use the notation $\bbR_+= (0,\infty)$ and $\overline{\bbR}_{+}= [0,\infty)$.     

\begin{theorem}
Suppose the divisor $\bD$ is smooth and $\omega$ is an asymptotically tame polyfibred cusp K\"ahler metric with standard spatial asymptotics $\{\omega_i\}_{i=1}^{\ell}$ and $\{c_i\}_{i=1}^{\ell}$.  
If $u\in x^{\alpha}\CI_{\fc}(X)$ for some $\alpha\in (\overline{\bbR}_{+})^{\ell}$ is such that 
\[
    (\Delta_{\omega} -\lambda)u=f \quad \mbox{for some} \; \lambda\in \bbR, \;
    f\in \cA^{\cF}_{\fc}(\tX)\cap \CI_{\fc}(X),
\] 
where $\cF=(F_1, \dots, F_\ell)$ is an index family, then 
\[
          u \in \cA_{\fc}^{\widehat{\cE}^{+}_{\lambda}(\alpha-\delta)\overline{\cup}\cF} (\tX), \quad \forall \; \delta\in (\bbR_+)^{\ell},
\]
where $\widehat{\cE}^{+}_{\lambda}(\alpha-\delta)= (\hE^{+}_{1,\lambda}(\alpha_1-\delta_1), \ldots, \hE^{+}_{\ell,\lambda}(\alpha_\ell-\delta_{\ell}))$.
\label{br.18}\end{theorem}
\begin{proof}
Let $\chi_i\in \CI_{c}([0,\infty))$ be a function such that $\chi_i (t)=1$ for $t\le \frac{\epsilon_i}{4}$ and $\chi_i(t)=0$ for $t\ge \frac{\epsilon_i }{2}$.  Looking at the equation satisfied by $\chi_i u$, we can effectively reduce to the case where $u$ and $f$ are supported in the collar neighborhood $\tc_i(\tH_i\times (0,\frac{\epsilon_i}{2}))\subset X$.  Using Lemma~\ref{br.6}, we can write 
$u=u_1+ u_2$ with $\xi u_1=0$ and $u_2 \in \dot{\cC}^{\infty}(\tX)$, where $\xi$ is the infinitesimal generator of the $\bbS^1$-action associated to the fibration
$\tPhi\times \Id$.  This means we can rewrite the equation as,
\[
     (\Delta_{\omega}-\lambda)u_1 = \tilde{f}, \quad \tilde{f}= f - (\Delta_{\omega}-\lambda)u_2 \in \cA^{\cF}_{\fc}(\tX)\cap \CI_{\fc}(X).
\]    
If we write $\tilde{f}= \tilde{f}_1+ \tilde{f}_{2}$ with $\tilde{f}_1=\tilde{f}_{\av}$ and $\tilde{f}_{2}\in \dot{\cC}^{\infty}(\tX)$, then averaging on both sides with respect to the circle action on $\tH_{i}\times (0,\epsilon_i)$, we get
\begin{equation}
  (\Delta_{\omega}^{\av}-\lambda)u_1= \tilde{f}_{1}
\label{br.19}\end{equation}
where 
\[
    \Delta^{\av}_{\omega}= \frac{1}{2\pi} \int_{\bbS^{1}} (\theta^{*}\circ \Delta_{\omega}\circ (\theta^{-1})^{*}) d\theta
\]
is the $\bbS^{1}$ invariant part of $\Delta_{\omega}$.  In particular, since $\Delta^{\av}_{\omega}$ maps $\bbS^1$-invariant functions to $\bbS^{1}$ invariants functions, there is an operator $P_i$ such that 
\begin{equation}
  \Delta^{\av}_{\omega}(\tPhi_i\times \Id)^{*}(h)= (\tPhi_i\times \Id)^{*}(P_i h), \quad 
  \forall \; h\in \CI(\bD_i\times (0,\epsilon_i)).
\label{br.20}\end{equation}  
Using local coordinates $w=(w^1,\ldots, w^{n-2})$ on $\bD_i$ over which the fibration $\tPhi_i$ is trivial, one can check that the operator $P_i$ takes the form
\begin{equation}
  P_i = a\left( x_i \frac{\pa}{\pa x_i}\right)^2 + b x_i \frac{\pa}{\pa x_i} +
  \sum_{\alpha} b^\alpha \left( x_i \frac{\pa}{\pa x_i}\right) \frac{\pa}{\pa w^{\alpha}} + \sum_{\alpha,\beta} a^{\alpha\beta} \frac{\pa^{2}}{\pa w^\alpha \pa w^\beta} + \sum_{\alpha} c^{\alpha} \frac{\pa}{\pa w^{\alpha}},
\label{br.21}\end{equation}
where $a,b, a^{\alpha\beta}, b^{\alpha}, c^{\alpha}\in \CI(\bD_i\times [0,\epsilon_i))$.  In particular, the operator $P_i$ is a $b$-differential operator 
near the boundary hypersurface $\bD_i\times \{0\}$.  Since the average of positive-definite symmetric matrices is again a positive-definite symmetric matrix, we see the operator $P_i$ is also elliptic, in fact $b$-elliptic in the sense of \cite{MelroseAPS}, that is, uniformly elliptic for the $b$-geometry.

Thus, if we write $u_1= (\tPhi_i\times \Id)^{*}\overline{u}$ and $
\tilde{f}_1= (\tPhi_i\times \Id)^{*}\overline{f}$, then equation \eqref{br.19} becomes
\begin{equation}
    (P_i-\lambda)\overline{u}= \overline{f}.
\label{br.22}\end{equation}
Changing the operator $P_i$ outside the region where $\overline{u}$ and $\overline{f}$ are supported, we can also assume that $P_i$ is a $b$-elliptic  differential operator on $\bD_i\times [0,\epsilon_i]$.  Now, from the standard spatial asymptotics of $\omega$ at $\tH_i$, we know that the indicial family of 
$P_i-\lambda$ at the boundary hypersurface $\bD_i\times \{0\}$ is given by
$\hat{P}_{i,\lambda}(\tau)$ in \eqref{br.11}.  
For later convenience, it is also useful to choose $P_i$ so that the indicial family of $P_i-\lambda$ at the other boundary hypersurface $\bD_i\times \{\epsilon_i\}$ is also given by \eqref{br.11}. 

Now, if $f\in \cA^{\cF}_{\fc}(\tX)$, this means that $\overline{f}\in \cA^{F_i,\emptyset}_{\phg}(\bD_i\times [0,\epsilon_i])$, where $F_i$ is the index set at the boundary face $\bD_i\times\{0\}$ and  $\emptyset$ is the index set at the boundary face $\bD_i\times \{\epsilon_i\}$.  On the other hand, the fact
$u\in x^{\alpha}\CI_{\fc}(X)$ certainly implies $\overline{u}\in x_{i}^{\alpha_i-\delta_i}H_{b}^{m}(\bD_i\times [0,\epsilon_i])$ for all $m\in \bbN$ and $\delta_i>0$.
 
Applying the standard boundary regularity result \cite[Proposition~5.61]{MelroseAPS} and using the fact that $\overline{u}$ is supported away from $\bD_i\times\{\epsilon_i\}$, we conclude that
\[
   \overline{u}\in \cA^{\hE^{+}_{i,\lambda}(\alpha_i-\delta_i)\overline{\cup}F_i, \emptyset}_{\phg}(\bD_i\times [0,\epsilon_i]).
\] 
In particular, we have that
\[
        (\tPhi_i\times \Id)^{*}\overline{u} \in \cA^{\hE^{+}_{i,\lambda}(\alpha_i-\delta_i)\overline{\cup}F_i,\emptyset}_{\fc}(\tH_i\times [0,\epsilon_i]),
\]
from which the result follows.  
    
\end{proof}

To study boundary regularity for the Monge-Amp\`ere equation, we need a more general version of the previous theorem where less regularity is assumed for the function $f$.  We first need to study some of the properties of the $b$-differential operator $P_i$ introduced in the proof of Theorem~\ref{br.18}.  From \cite[Theorem~5.60]{MelroseAPS}, we know that the operator $P_i$ induces a Fredholm operator 
\begin{equation}
 P_i-\lambda: x_i^{\alpha} H^{m+2}_{b}( \bD_i\times [0,\epsilon_i]) \to
                  x_i^{\alpha} H^{m}_{b}( \bD_i\times [0,\epsilon_i])\label{br.24}\end{equation}    
whenever $-\sqrt{-1}\alpha\notin \Spec_b(\hat{P}_{i,\lambda})$.  Furthermore, the index of $P_i-\lambda$ does not depend on $m\in \bbN_0$, but does depend on $\alpha$.  When $\lambda>0$, we can compute the index explicitly.  

\begin{lemma}
For $\lambda>0$, the operator $P_i-\lambda$ induces an isomorphism of Sobolev spaces 
\begin{equation}
            P_i-\lambda: H^{m+2}_{b}(\bD_i\times [0,\epsilon_i])\to H^{m}_{b}(\bD_i\times [0,\epsilon_i]).
\label{br.25b}\end{equation}
\label{br.25}\end{lemma}
\begin{proof}
Using the principal symbol of $P_i$, we can define a $b$-metric $g$ such that 
$P_i$ takes the form 
\[
        P_i u= \Delta_g u + b\cdot \nabla u, \quad b\in \CI(\bD_i\times [0,\epsilon_i];{}^{b}T(\bD_i\times [0,\epsilon_i])),
\]
where $\Delta_g$ is the (negative) Laplacian associated to the metric $g$ and $\nabla$ is its Levi-Civita connection.  To show $P_i-\lambda$ is injective,
suppose that $u\in H^{m+2}_{b}(\bD_i\times [0,\epsilon_i])$ is such that 
$(P_i-\lambda)u=0$.  Then we know from \cite[Proposition~5.61]{MelroseAPS} that $u\in \cA^{\hE^{+}_{i,\lambda}(0),\hE^{+}_{i,\lambda}(0)}_{\phg}(\bD_i\times [0,\epsilon_i])$.  In particular, $u$ vanishes on $\bD_i\times \{0\}$ and $\bD_i\times \{\epsilon_i\}$.  If $u$ is not identically zero, then replacing $u$ by $-u$ if needed, we can assume that $u$ attains a positive maximum at an interior point
$p\in \bD_i\times (0,\epsilon_i)$.  Evaluated at the point $p$, the equation
$Pu=\lambda u$ gives
\[
               \Delta_g u (p)= \lambda u(p) >0,
\]  
contradicting the fact that $u$ attains a maximum at $p$.  To avoid a contradiction, we must admit $u\equiv 0$, which establishes injectivity. 

To show the map is surjective, it suffices to show its index is zero.  Since the map \eqref{br.25b} is Fredholm for all $\lambda>0$, the index does not depend on the choice of $\lambda>0$ and we can therefore assume $\lambda$ is as large as we want for the purpose of the argument.  Now, the formal adjoint of $P_i-\lambda$ is of the form
\[
     (P^{*}_i-\lambda)u= \Delta_g u + \tilde{b}\cdot \nabla u + (c-\lambda)u, 
\]
for some $ \tilde{b}\in \CI(\bD_i\times [0,\epsilon_i], {}^{b}T( \bD_i\times [0,\epsilon_i]))$ and $c\in \CI(\bD_i\times [0,\epsilon_i])$.  In particular, taking $\lambda>0$ sufficiently large so that $c-\lambda<0$ on 
$\bD_{i}\times [0,\epsilon_i]$, we can use the maximum principle as before to show that 
\[
      u\in H^{m+2}_{b}(\bD_i\times [0,\epsilon_i]), \; (P^{*}_i-\lambda)u=0 \quad \Longrightarrow \quad u\equiv 0.
\]
This shows that $P_i-\lambda$ has index zero in this case, and therefore for all $\lambda>0$.  
 
\end{proof}

\begin{lemma}
For $\lambda>0$ and $\alpha\ge 0$ with $-\sqrt{-1} \alpha\notin \Spec_b(\hat{P}_{i,\lambda})$, the Fredholm operator
\begin{equation}
   P_i -\lambda: x_i^{\alpha} H^{m+2}_{b}(\bD_i\times [0,\epsilon_i])\to x_i^{\alpha}H^{m}_{b}(\bD_i\times [0,\epsilon_i])
\label{br.26b}\end{equation}
is injective and has a cokernel of dimension
\[
      k_{i,\lambda}^{\alpha}= \underset{0\le -\Im\tau \le \alpha}{\sum_{ \tau\in \Spec_b(\hat{P}_{i,\lambda})} } \dim\ker(\hat{P}_{i,\lambda}(\tau)).
\]
\label{br.26}\end{lemma}
\begin{proof}
Since $x_i^{\alpha}H^{m+2}_{b}(\bD_i\times [0,\epsilon_i])\subset H^{m+2}_{b}(\bD_i\times [0,\epsilon_i])$ for $\alpha\ge 0$, we know from the previous lemma that the map \eqref{br.26b} is injective.  The formula for the dimension of the cokernel then follows from the previous lemma and the relative index theorem of \cite{MelroseAPS}.
\end{proof}

The following lemma will be helpful to characterize the cokernel of the map \eqref{br.26b}.  

\begin{lemma}
Suppose $\lambda>0$.  Then for each $\tau\in \Spec_b(\hat{P}_{i,\lambda})$ and $v\in \ker(\hat{P}_{i,\lambda}(\tau))$, there exists $u\in \cA^{\hE^{+}_{i,\lambda}(\alpha-\delta),\emptyset}_{\phg}(\bD_i\times [0,\epsilon_i])$ with $\alpha=\sqrt{-1}\tau$ 
and $\delta>0$ small such that
\begin{itemize}
 \item[(i)] $(P_i -\lambda)u \in \dot{\cC}^{\infty}(\bD_i\times [0,\epsilon_i])$;
 \item[(ii)] $u-\chi(x_i) vx_i^{\alpha} \in \cA^{\hE^{+}_{i,\lambda}(\alpha),\emptyset}_{\phg}(\bD_i\times [0,\epsilon_i])$,      
\end{itemize}
where $\chi \in \CI([0,\epsilon_i])$ is a cut-off function with $\chi(t)=1$ for $t<\frac{\epsilon_i}{2}$ and $\chi(t)=0$ for $t>\frac{3\epsilon_i}{4}$.
\label{br.27}\end{lemma}
\begin{proof}
Consider the index set $G(\beta)= \hE^{+}_{i,\lambda}(\beta)\cap ( (\alpha+\bbN_0)\times \bbN_0)$ for $\beta\in\bbR$ and fix $0<\delta<1$.     
Starting with $u_0= \chi(x_i) v x_i^{\alpha}$, we will inductively construct 
$u_j\in \cA^{G(\alpha+j-\delta),\emptyset}_{\phg}(\bD_i\times [0,\epsilon])$ such that
\begin{equation}
  (P_i-\lambda)\left( \sum_{j=0}^{m} u_j\right)  \in \cA^{G(\alpha+m),\emptyset}_{\phg}(\bD_i\times [0,\epsilon_i])
\label{br.28}\end{equation}
for all $m\in \bbN_0$.
Using Borel lemma to take an asymptotic sum, it then suffices to take $u\in \cA^{G(\alpha-\delta),\emptyset}_{\phg}(\bD_i\times[0,\epsilon_i])$ such that 
\[
      u\sim \sum_{j=0}^{\infty} u_j
\]
to obtain the result.  First, notice that since $-\sqrt{-1}\alpha$ is in
$\Spec_b(\hat{P}_{i,\lambda})$, we have that
\[
         \left( \Delta_{\omega_i} + \frac{c_i}{2} \left( \left(x_i \frac{\pa}{\pa x_i}\right)^{2} + x_i \frac{\pa}{\pa x_i} \right) -\lambda\right) x_i^{\alpha}v=0.
\]  
Thus, this means that 
\[
     (P_i -\lambda)u_0 \in x_i^{\alpha+1} \CI(\bD_i\times [0,\epsilon_i]),
\]
that is, \eqref{br.28} holds for $m=0$.  Assume now that \eqref{br.28} holds for some $m\in \bbN_0$.  This means that we have
\[
    (P_i-\lambda) \sum_{j=0}^{m} u_j = \sum_{p=0}^{k} \chi(x_i)a_p x_i^{\alpha+m+1}(\log x_i)^{p} + f, \quad f\in \cA^{G(\alpha+m+1),\emptyset}_{\phg}(\bD_i\times [0,\epsilon_i]),  
\]  
for some $k\in \bbN_0$ and $a_0,\ldots, a_k\in \CI(\bD_i)$.  If $-\sqrt{-1}(\alpha+m+1)\notin \Spec_b(\hat{P}_{i,\lambda})$, we can find $b_p\in \CI(\bD_i)$ for $p=0,1,\ldots,k$ such that
\begin{multline*}
    \left( \Delta_{\omega_i} + \frac{c_i}{2} \left( \left(x_i \frac{\pa}{\pa x_i}\right)^{2} + x_i \frac{\pa}{\pa x_i} \right) -\lambda\right)  \left( \sum_{p=0}^{k} b_p x_i^{\alpha+m+1} (\log x_i)^{p} \right)=   \\
    \sum_{p=0}^{k} a_p x_i^{\alpha+m+1}(\log x_i)^{p}.
\end{multline*}
In terms of the operator $P_i$, this means that
\[
     (P_i -\lambda)\left( \sum_{p=0}^{k} \chi(x_i)b_p x_i^{\alpha+m+1} (\log x_i)^{p} \right)= \sum_{p=0}^{k}\chi(x_i) a_p x_i^{\alpha+m+1}(\log x_i)^{p}  + c_{m+1}, 
\]
where $c_{m+1} \in \cA^{G(\alpha+m+1),\emptyset}_{\phg}(\bD_i\times [0,\epsilon_i])$.  Thus, we can take 
\[
u_{m+1}= - \sum_{p=0}^{k} \chi(x_i) b_p x_i^{\alpha+m+1}(\log x_i)^{p}
\]
 in this case.  If instead $-\sqrt{-1}(\alpha+m+1)\in \Spec_b(\hat{P}_{i,\lambda})$, this means $-\nu = \lambda- \frac{c_i}{2} ((\alpha+m+1)^{2}+(\alpha+m+1))$ is an eigenvalue of the $\db$-Laplacian $\Delta_{\omega_i}$.  In this case, we can find 
 $b_k\in \CI(\bD_i)$ and $\tilde{a}_{k} \in \ker(\Delta_{\omega_i} +\nu)$ such that
 \[
     (\Delta_{\omega_i}+\nu)b_k= a_k -\tilde{a}_{k}.  
 \]  
 Moreover, we can find $\tilde{b}_{k}\in \ker(\Delta_{\omega_i}+\nu)$ such that 
 \begin{multline*}
   \left( \Delta_{\omega_i} + \frac{c_i}{2} \left( \left(x_i \frac{\pa}{\pa x_i}\right)^{2} + x_i \frac{\pa}{\pa x_i} \right) -\lambda\right)  x_i^{\alpha+m+1}\tilde{b}_k (\log x_i )^{k+1}   = \\
    x^{\alpha+m+1}_{i} \left( \tilde{a}_{k}(\log x_i)^{k} + \tilde{b}_k  (k+1)k(\log x_i)^{k-1} \right).
\end{multline*}
This means that if we take
\[
    u_{m+1}^{k}= \chi(x_i) x_i^{\alpha+m+1}\left( \tilde{b}_{k}(\log x_i)^{k+1} + b_k (\log x_i)^{k}\right),
\]
then
\[
  (P_i-\lambda)u^k_{m+1}-  \sum_{p=0}^{k}\chi(x_i) a_p x_i^{\alpha+m+1}(\log x_i)^{p} =   \sum_{p=0}^{k-1}\chi(x_i) a^{k-1}_p x_i^{\alpha+m+1}(\log x_i)^{p} +c,
\]
where $a^{k-1}_{p} \in \CI(\bD_i)$ and $c \in \cA^{G(\alpha+m+1),\emptyset}_{\phg}(\bD_i\times [0,\epsilon_i])$.  This effectively reduces $k$ by one.  Repeating this step $k$ times we can thus reduce to the case $k=0$.  Applying this step one more time then gives us the desired term $u_{m+1}$.  More precisely, proceeding recursively starting with $q=k$ and finishing with $q=0$, we can define 
\[
    u_{m+1}^{q}= \chi(x_i) x_i^{\alpha+m+1}\left( \tilde{b}_{q}(\log x_i)^{q+1} + b_q (\log x_i)^{q}\right),
\]  
for some $\tilde{b}_{q}, b_q \in \CI(\bD_i)$ 
so that 
\[
  (P_i-\lambda)\left(\sum_{q=0}^{k}u^q_{m+1}\right)=  \sum_{p=0}^{k}\chi(x_i) a_p x_i^{\alpha+m+1}(\log x_i)^{p} +d,   
\]
with $d\in  \cA^{G(\alpha+m+1),\emptyset}_{\phg}(\bD_i\times [0,\epsilon_i])$.  Thus, taking 
\[
    u_{m+1}= -\sum_{q=0}^{k} u^k_{m+1}
\]
completes the inductive step and the proof.  
\end{proof}

\begin{proposition}
Suppose that  $\lambda>0$. Then for each $\alpha\ge 0$ with $-\sqrt{-1} \alpha \notin \Spec_{b}(\hat{P}_{i,\lambda})$, there exists a subspace $W^{\alpha}_{i,\lambda}\subset \cA^{\hE^{+}_{i,\lambda}(0),\emptyset}_{\phg}(\bD_i\times [0,\epsilon_i])$ of dimension $k^{\alpha}_{i,\lambda}$ (see Lemma~\ref{br.26}) such that
$V^{\alpha}_{i,\lambda}= (P_i-\lambda)W^{\alpha}_{i,\lambda} \subset \dot{\cC}^{\infty}(\bD_i\times [0,\epsilon_i])$ and  
\[
        x_i^{\alpha}H^{m}_{b}(\bD_i\times [0,\epsilon_i])= (P_i-\lambda)(x_i^{\alpha} H^{m+2}_{b}(\bD_i\times [0,\epsilon_i])) + V^{\alpha}_{i,\lambda} \quad \forall m\in \bbN_{0}.
\]
\label{br.29}\end{proposition}
\begin{proof}
For each $0<\beta <\alpha$ such that $-\sqrt{-1}\beta\in \Spec_b(\hat{P}_{i,\lambda})$, choose a basis $v^1_{\beta},\ldots, v^{k_{\beta}}_{\beta}$ of $\ker( \hat{P}_{i,\lambda}(-\sqrt{-1}\beta))$.
Let $u_{\beta}^{1},\ldots,u_{\beta}^{k_{\beta}}\in \cA^{\hE_{i,\lambda}(\beta-\delta),\emptyset}_{\phg}(\bD_i\times [0,\epsilon_i])$ be functions as in Lemma~\ref{br.27}, that is, such that 
\begin{gather} \label{br.30b}
  (P_i -\lambda) u_{\beta}^j \in \dot{\cC}^{\infty}(\bD_i\times [0,\epsilon_i]); \\
  u_{\beta}^{j}- \chi(x_i)v_{\beta}^{j}x_i^{\beta} \in \cA_{\phg}^{\hE^{+}_{i,\lambda}(\beta),\emptyset}(\bD_i\times [0,\epsilon_i]).  \label{br.30}
\end{gather} 
We define $W^{\alpha}_{i,\lambda}$ to be the finite dimensional subspace of $\cA_{\phg}^{\hE^{+}_{i,\lambda}(0),\emptyset}(\bD_i\times [0,\epsilon_i])$ spanned by the functions $u^{j}_{\beta}$ for 
$j\in\{1,\ldots,k_{\beta}\}$ and for $0<\beta<\alpha$ such that $-\sqrt{-1}\beta \in \Spec_b(\hat{P}_{i,\lambda})$.  It is clear from condition \eqref{br.30} that the functions $u^{j}_{\beta}$ are linearly independent, so that $W^{\alpha}_{i,\lambda}$ is indeed a vector space of dimension $k^{\alpha}_{i,\lambda}$.  By the injectivity of the operator $P_i -\lambda$, the subspace $V^{\alpha}_{i,\lambda}\subset \dot{\cC}^{\infty}(\bD_i\times [0,\epsilon_i])$ defined by
\[
         V^{\alpha}_{i,\lambda}= (P_i -\lambda) W^{\alpha}_{i,\lambda}
\]
is also of dimension $k^{\alpha}_{i,\lambda}$.  From condition \eqref{br.30}, we see that
\[
         W^{\alpha}_{i,\lambda}\cap x^{\alpha}_i H^{m+2}_{b}(\bD_i\times [0,\epsilon_i])= \{0\}.
\]
This means that 
\[
              V^{\alpha}_{i,\lambda} \cap (P_i-\lambda)(x_i^{\alpha} H^{m+2}_{b}(\bD_i\times [0,\epsilon_i]))= \{0\}.
\]
Since by Lemma~\ref{br.26},  $\dim V^{\alpha}_{i,\lambda}= k^{\alpha}_{i,\lambda}$ is precisely the dimension of the cokernel of the operator \eqref{br.26b}, the result follows.

\end{proof}

We finally state and prove the following generalization of Theorem~\ref{br.18}.  

\begin{theorem}
Let $\cF=(F_1,\ldots, F_{\ell})$ be some index family and let $\alpha,\beta\in (\overline{\bbR}_+)^{\ell}$ be given.
If $u\in x^{\alpha}\CI_{\fc}(X)$  satisfies the equation
\[
    (\Delta_{\omega} -\lambda)u= f,
\]
for some $\lambda>0$ and with 
\[
       f= f_1+ f_2, \quad f_1\in x^{\alpha+ \beta}\CI_{\fc}(X), \quad
         f_2 \in \cA^{\cF}_{\fc}(\tX) \cap \CI_{\fc}(X),  
\]
then $u= u_1 + u_2$ for some functions $u_1$ and $u_2$ such that for all $\delta\in (\bbR_{+})^{\ell}$,
\[
    u_1\in x^{\alpha+\beta-\delta}\CI_{\fc}(X), \quad u_2\in \cA^{\widehat{\cE}^{+}_{\lambda}(\alpha-\delta)\overline{\cup}\cF}_{\fc}(\tX)\cap x^{\alpha}\CI_{\fc}(X).
\]
\label{br.31}\end{theorem}
\begin{proof}
As in the proof of Theorem~\ref{br.18}, we can assume $u$ and $f$ are supported in the collar neighborhood $\tc_i(\tH_i\times [0,\frac{\epsilon_i}{2}))\subset \tX$.  By averaging over the circle action in that collar neighborhood, we can also assume, in light of Lemma~\ref{br.6}, that
$u= (\tPhi_i\times \Id)^{*}\overline{u}$ and $f_j= (\tPhi_i\times \Id)^{*}\overline{f_j}$, so that the equation becomes
\[
         (P_i -\lambda)\overline{u} = \overline{f}_1 + \overline{f}_2.
\]  
Since $u\in x^{\alpha}\CI_{\fc}(X)$,  $f_1\in x^{\alpha+\beta}\CI_{\fc}(X)$ and $f_2\in \cA^{\cF}_{\fc}(\tX)$, we have that $\overline{u}\in x_{i}^{\alpha_i-\delta_i} H^{m+2}_{b}(\bD_i\times [0,\epsilon_i])$ and $\overline{f}_1\in x_{i}^{\alpha_i+\beta_i-\delta_i}H^{m}_{b}(\bD_i\times [0,\epsilon_i])$  for all $m\in \bbN$ and $\delta_i >0$, while $\overline{f}_2\in \cA^{F_i ,\emptyset}_{\phg}(\bD_i\times [0,\epsilon_i])$.   

Assuming $\delta_i>0$ is chosen small enough so that
\[
            t\in [\alpha_i+\beta_i-\delta_i, \alpha_i+\beta_i) \; \Longrightarrow \; -\sqrt{-1} t \notin \Spec_b(\hat{P}_{i,\lambda}),
\]
we can apply Proposition~\ref{br.29} to conclude there exists  $\overline{v}_{1}\in x_{i}^{\alpha_i+\beta_i-\delta_i}H^{m+2}_{b}(\bD_i\times [0,\epsilon_i])$ and   $\overline{v}_{2} \in
\cA^{\hE^{+}_{i,\lambda}(0),\emptyset}_{\phg}(\bD_i\times [0,\epsilon_i])$ such that 
\[
       (P_i -\lambda) (\overline{v}_{1} + \overline{v}_2 )= \overline{f}_1.
\] 
Moreover, both $\overline{v}_1$ and $\overline{v}_2$ can be chosen to be independent of the choice of $\delta_i>0$ and $m$.  For the function $\overline{v}_1$, this implies by the Sobolev embedding for $b$-metrics that
\[
     \overline{v}_1 \in x_{i}^{\alpha_i+\beta_i-\delta_i}\CI_{b}(\bD_i\times [0,\epsilon_i]) \quad \forall \; \delta_i>0.\]

On the other hand, by Lemma~\ref{br.25} and Theorem~\ref{br.18}, there exists a unique 
$\overline{v}_3 \in \cA^{\hE_{i,\lambda}^{+}(-\delta_i)\overline{\cup}F_i, \hE_{i,\lambda}^{+}(-\delta_i)}_{\phg}(\bD_i\times [0,\epsilon_i])$ such that 
\[
       (P_i-\lambda)\overline{v}_3= \overline{f}_2.  
\] 
Thus, setting $\overline{u}_{1}= \overline{v}_1$ and $\overline{u}_2 = \overline{v}_2 + \overline{v}_3$,  we have that
\[
        (P_i -\lambda)(\overline{u}_1 + \overline{u}_2) = \overline{f}_1 + \overline{f}_2 .
\]
By the injectivity of the operator $P_i-\lambda$, we conclude that 
\[
          \overline{u}= \overline{u}_1 + \overline{u}_2.  
\]
Since $\overline{u}$ is supported in $\bD_i\times [0,\frac{\epsilon_i}{2})$,  we also have that
\[
              \overline{u} = \chi(x_i) \overline{u}_1 + \chi(x_i)\overline{u}_2, 
\]
where   $\chi\in \CI_{c}([0,\epsilon_i))$ is a cut-off function with $\chi(t)=1$ for $t<\frac{\epsilon_i}{2}$ and $\chi(t)=0$ for $t\ge \frac{3\epsilon_{i}}{4}$.  In other words, replacing $\overline{u}_{j}$ by $\chi(x_i) \overline{u}_{j}$ if needed, we can assume both $\overline{u}_{1}$ and $\overline{u}_{2}$ are supported in $\bD_i\times [0,\frac{3\epsilon_i}{4})$.  
Since $\overline{u}_{1}=\overline{v}_1 \in x_{i}^{\alpha_i+\beta_i-\delta_i}\CI_{b}(\bD_i\times [0,\epsilon_i])$ for all  $\delta_i>0$,
the theorem follows by taking $u_1= (\tPhi_i\times \Id)^{*} \overline{u}_{1}$ and 
$u_{2}=(\tPhi_i \times \Id)^{*}\overline{u}_{2}$.  From the definition of $u_2$, we only know a priori that $u_2\in \cA_{\fc}^{\cE^{+}_{\lambda}(-\delta)\overline{\cup}\cF}(\tX)$ for some $\delta\in (\bbR_+)^{\ell}$, but since 
$u_2= u-u_1 \in x^{\alpha}\CI_{\fc}(X)$, we see that $u_2$ does have the claimed regularity at the boundary.  
       
\end{proof}

\section{Full asymptotics of K\"ahler-Einstein metrics on quasiprojective manifolds} \label{clt.0}

In this section, we suppose that the divisor $\bD$ is smooth and is such that 
$K_{\bX}+[\bD]>0$, in which case the regularity results of \S~\ref{br.0} will be useful to obtain the full asymptotics of the K\"ahler-Einstein metric \eqref{ake.4}.  To do so, it will be convenient to write the associated complex Monge-Amp\`ere equation with respect to a background K\"ahler form slightly different than the one given in \eqref{ake.3}.

First, notice that if $\bD_1,\ldots,\bD_{\ell}$ are the irreducible components of the divisor $\bD$, then the holomorphic sections $s_1, \ldots, s_{\ell}$  define
a holomorphic section $s= \bigotimes_{i=1}^{\ell} s_i \in H^{0}(\bX;[\bD])$ such that $s^{-1}(0)= \bD$.  Moreover, the Hermitian metrics on $[\bD_1], \ldots, [\bD_{\ell}]$ define a Hermitian metric $\|\cdot\|_{\bD}$ on the line bundle $[\bD]$.   Now, a volume form $\Omega$ on $\bX$ induces a Hermitian metric for the canonical line bundle $K_{\bX}$.  Thus, such a volume form together with the Hermitian metric  $\|\cdot\|_{\bD}$ induce a Hermitian metric on the line bundle $K_{\bX}+ [\bD]$.   
By the extension theorem of Schumacher \cite[Theorem~4]{Schumacher},  we can choose the volume form $\Omega$ on $\bX$  such that the induced Hermitian metric on $K_{\bX}+[\bD]$ has positive curvature and restricts on each irreducible component $\bD_i$ to be the Hermitian metric on  $\left. (K_{\bX}+[\bD])\right|_{\bD_i} = K_{\bD_i}$ induced by the volume form $n\omega_{\KE,i}^{n-1}$ on  $\bD_i$, where $\omega_{\KE,i}$ is the K\"ahler-Einstein metric of \eqref{ake.8} on $\bD_i$.    
With such a choice, the curvature form $\sqrt{-1}\Theta_{K_{\bX}+[\bD]}$ is such that its restriction to $\bD_i$ is precisely the K\"ahler form $\omega_{\KE,i}$.  

Instead of the K\"ahler form \eqref{ake.3}, consider the form
\begin{equation}
\begin{aligned}
   \omega &= -\sqrt{-1}\ \db \pa \log \left(  \frac{\Omega}{ \| s\|^{2}_{\bD}}\right) +
   \sqrt{-1} \ \db\pa \log \left( (-\log \epsilon \|s \|^{2}_{\bD})^{2}\right) \\
   &=  \sqrt{-1} \Theta_{K_{\bX}+[\bD]} + 2\sqrt{-1}  \left( \frac{\Theta_{\bD}}{ \log \epsilon \| s\|^{2}_{\bD}} \right) \\
   & \quad \quad \quad + 2\sqrt{-1} 
     \left(  \frac{(\pa \log \epsilon \|s\|^{2}_{\bD}) \wedge (\db   \log \epsilon \|s\|^{2}_{\bD})}{  (\log \epsilon \|s\|^{2}_{\bD})^{2}} \right),
      \end{aligned}
\label{fff.1}\end{equation}      
which for $\epsilon>0$ sufficiently small is K\"ahler.  As for the K\"ahler form in \eqref{ake.3}, the K\"ahler form \eqref{fff.1} induces an asymptotically tame polyfibred cusp K\"ahler metric.  Using $\omega$ as a background K\"ahler metric, we can then obtain a K\"ahler-Einstein metric $\omega_{\KE}= \omega+\sqrt{-1}\pa\db u$ by taking $u\in \CI_{\fc}(X)$ to be the solution to the complex Monge-Amp\`ere equation
\begin{equation}
    \log \left( \frac{ (\omega+ \sqrt{-1} \pa \db u)^{n} }{ \omega^n } \right) -u = F= \log\left( \frac{\widetilde{\Omega}}{\omega^n}\right) \in \CI_{\fc}(\tX),
\label{fff.2}\end{equation}
where $\widetilde{\Omega}$ is the volume form
\begin{equation}
  \widetilde{\Omega}=  \frac{\Omega}{ \left( \|s\|^{2}_{\bD} (-\log (\epsilon \|s\|^{2}_{\bD}))^{2}\right) }  \in \CI_{\fc}(\tX; \Lambda^{2n}({}^{\fc}T^{*}\tX)). 
\label{fff.3}\end{equation}
As the next lemma shows, the restriction of the function $F$ to $\bD_i$ vanishes.   

\begin{lemma}
If the volume form $\Omega$ and the Hermitian metric $\|\cdot\|_{\bD}$ are chosen as above, then the function $F$ in \eqref{fff.2} is such that
\[
     F -x_i \Xi_i\left( \frac{(n-1)\omega_{\KE,i}^{n-2}\wedge(\sqrt{-1} \left.\Theta_{\bD} \right|_{\bD_i}) }{\omega_{\KE,i}^{n-1}} \right)  \in x_{i}^{2}\CI_{\fc}(\tX).
\]
In particular, the restriction of $F$ to $\bD_i$ vanishes.  

\label{clt.1}\end{lemma}
\begin{proof}
To proceed, we will use local holomorphic coordinates $(\zeta, z_2,\ldots,z_n)$ as in \eqref{fc.16} with $\zeta$ corresponding to the section $s= \bigotimes_{j=1}^{\ell} s_j$ with respect to some choice of local trivialization of $[\bD]$.  In such local coordinates, the Hermitian metric $\|\cdot \|_{\bD}$ takes the form
\begin{equation}
     \|s\|_{\bD}^{2}= h |\zeta |^{2}
\label{clt.2}\end{equation}     
for some smooth real valued function $h$.  Similarly, in these local coordinates, the form $\Omega$ is given by
\[
   \Omega= \psi d\zeta\wedge d\overline{\zeta} \wedge dz_2\wedge d\overline{z}_2 \wedge\ldots \wedge dz_n \wedge d\overline{z}_{n},  
\]
for some smooth function $\psi$.  
Our assumption on the form $\Omega$ is then that 
\begin{equation}
  \left. \frac{\psi dz_2\wedge d\overline{z}_2 \wedge\ldots \wedge dz_n \wedge d\overline{z}_{n}}{h} \right|_{\bD_i}= n \omega_{\KE,i}^{n-1}.
   \label{clt.3}\end{equation}
Thus, in these local coordinates, the difference 
\begin{equation}
  \frac{\Omega}{(\|s\|^{2}_{\bD} (-\log \epsilon \|s\|_{\bD}^{2})^{2})}\; -\; \frac{n\omega_{\KE,i}^{n-1}}{(-\log\epsilon \|s\|^2_{\bD})^{2}} \wedge \sqrt{-1}\frac{d\zeta \wedge d\overline{\zeta}}{|\zeta|^{2}}
\label{clt.4}\end{equation}
is in $\rho_i \CI_{\fc}(\tX; \Lambda^{2n}({}^{\fc}T^{*}\tX)\otimes \bbC)$ where $\rho_i= \|s_i\|_{\bD_i}$.  On the other hand, since the restriction of $\sqrt{-1}\Theta_{K_{\bX}+ [\bD]}$ to $\bD_i$ is $\omega_{\KE,i}$,  we see from \eqref{fff.1} and using \eqref{fc.24} that 
\begin{multline}
\omega^n= \frac{n\omega_{\KE,i}^{n-1}}{(\log\epsilon \|s\|^2_{\bD})^{2}} \wedge \sqrt{-1}\frac{d\zeta \wedge d\overline{\zeta}}{|\zeta|^{2}} \\
- x_i' \frac{n(n-1) \omega_{\KE,i}^{n-2}\wedge (\sqrt{-1} \Theta_{\bD}) \wedge(\sqrt{-1} d\zeta\wedge d\overline{\zeta})}{ (\log \epsilon \|s\|_{\bD}^{2})^2 |\zeta|^2} + \nu,
\label{clt.5}\end{multline}
with $\nu \in x_i^2 \CI_{\fc}(\tX;\Lambda^{2n}({}^{\fc}T^{*}\tX)\otimes \bbC)$ and $x_i'= \frac{-1}{\log \epsilon^{\frac12}\|s\|_{\bD}}$.  By Lemma~\ref{mwc.10}, notice that at the cost of changing $\nu$,  we can replace $x_i'$ in \eqref{clt.5} by $x_i= \frac{-1}{\log \|s_i\|_{\bD_i}}$.   Combining \eqref{clt.4} and \eqref{clt.5}, we thus see that in the local coordinates $(\zeta, z)$, 
\[
   F= \log \left(\frac{\widetilde{\Omega}}{\omega^n}\right)=  x_i\left( \frac{(n-1)\omega_{\KE,i}^{n-2}\wedge(\sqrt{-1} \left.\Theta_{\bD} \right|_{\bD_i}) }{\omega_{\KE,i}^{n-1}} \right)  + \mu, 
\]
with $\mu \in x_i^2 \CI_{\fc}(\tX)$.  The result of the lemma then easily follows from this local description.  

\end{proof}

Consequently, with this choice of volume form $\Omega$ and Hermitian metric $\|\cdot\|_{\bD}$ described above, the induced K\"ahler form $\omega\in \CI_{\fc}(\tX;\Lambda^2 ({}^{\fc}T^{*}\tX)\otimes \bbC)$ defined in \eqref{ake.3} has standard spatial asymptotics $\{\omega_i\}_{i=1}^{\ell}$ and $\{c_i\}_{i=1}^{\ell}$ given by
\begin{equation}
    \omega_i =\omega_{\KE,i}, \quad c_i=1, \quad \forall i\in \{1,\ldots,\ell\}.
\label{ake.11}\end{equation}   
The solution of the complex Monge-Amp\`ere equation
\begin{equation}
 \log \left( \frac{ (\omega_i+ \sqrt{-1} \pa \db u_i)^{n-1} }{ \omega_i^{n-1} } \right) -u = F_i \in \CI(\bD_i),
\label{fff.4}\end{equation}  
 is thus simply $u_i\equiv 0$, since $\omega_i= \omega_{\KE,i}$ and the restriction $F_i$ of $F$ at $\bD_i$ vanishes.  By Theorem~\ref{ake.10}, which really in the case $\bD$ is smooth is Schumacher's results \cite[Theorem 2]{Schumacher},  we know that there exists $\nu>0$ such that the solution $u$ to the complex Monge-Amp\`ere equation \eqref{fff.2} is in $x^{\nu}\CI_{\fc}(X)$, where $x=\prod_{i=1}^{\ell} x_i$.  Using this fact and Theorem~\ref{br.31} we now proceed recursively to construct the full asymptotic of $u$ at infinity.  Let us start by identifying the first term in the asymptotic expansion of $u$.  To this end, it will be convenient to have a cut-off function $\chi\in \CI_c([0,\infty))$ with
$\chi(t)=1$ for $t<\frac{1}{2}$ and $\chi(t)=0$ for $t\ge 1$, as well as its reparametrized version  
$\chi_{r}(t)=\chi(\frac{t}{r})$ for $r>0$.  
\begin{lemma}
There exist constants $\tilde{b}_i \in \bbR$ and functions $b_i\in \CI(\bD_i)$  such that 
\[
       u-     \tc_*\left( \sum_{i=1}^{\ell} \chi_{r}(x_i)x_i ( \tilde{b}_i \log x_i + b_i) \right) \in x^{1+\delta}\CI_{\fc}(X)
\]
for some $\delta>0$ and $r>0$, where $\tc$ is the collar neighborhood map  introduced in \eqref{br.3}.
\label{ake.12}\end{lemma}  
\begin{proof}
By our choice of $\omega$, the Monge-Amp\`ere equation satisfied by $u\in x^{\nu}\CI_{\fc}(X)$ is of the form
\[
\log\left(  \frac{(\omega+ \sqrt{-1}\pa \db u)^{n}}{\omega^n } \right) -u = F,  \quad F\in x\CI_{\fc}(\tX). 
\]
Taking the exponential on both sides, this can be rewritten as
\begin{equation}
  \left( 1+\Delta_{\omega} u + \sum_{j=2}^{n} G_j^{\omega}(u) \right) e^{-u} = e^{F},
\label{ake.13}\end{equation}
with 
\[
       G_{j}^{\omega}(h)= \frac{n !}{(n-j)! j!}   \left( \frac{\omega^{n-j} \wedge(\sqrt{-1}\pa\db h)^{j}}{\omega^n}\right), \quad h\in \CI_{\fc}(X).
\]
In particular, since $u\in x^{\nu}\CI(X)$, we have that $G_j^{\omega}(u)\in x^{j\nu}\CI(X)$.  We have also that $e^{F}-1\in x\CI_{\fc}(X)$.  From \eqref{ake.13}, we therefore conclude that 
\begin{equation}
    (\Delta_{\omega}-1)u = f_1 + f_2, \quad f_1\in x^{2\nu}\CI_{\fc}(X), \quad f_2\in x\CI_{\fc}(\tX).
\label{ake.14}\end{equation}
This is a situation where Theorem~\ref{br.31} can be applied with $\lambda=1$.  In this case, the smallest $z$ such that $-\sqrt{-1} z\in \Spec_b(\hat{P}_{i,1})$ is $z=1$ for all $i\in\{0,\ldots,\ell\}$.  Moreover, in this case,
\begin{equation}
            \hat{P}_{i,1}(-\sqrt{-1})= \Delta_{\omega_i}, 
\label{ake.14b}\end{equation}
with kernel given by the constant functions.  On the other hand,  $f_2\in x\CI_{\fc}(\tX)$ means that
$f_2$ is a polyhomogeneous function in $\cA_{\fc}^{\cF}(\tX)$ with index family $\cF=(F_1,\ldots, F_\ell)$ given by $F_i=\bbN\times \{0\}$ for all $i$.  

Thus, if $2\nu <1$, we can apply Theorem~\ref{br.31} to equation
\eqref{ake.14} to conclude that
\[
     u\in x^{2\nu-\delta}\CI_{\fc}(X), \quad \forall \; \delta>0.
\]       
Repeating this argument if necessary, we can thus assume $u\in x^{\nu}\CI_{\fc}(X)$ with
$\nu> \frac{1}{2}$.  In this case, applying Theorem~\ref{br.31} to \eqref{ake.14} gives us a function
$v\in \cA^{\widehat{\cE}^{+}_1(0)\overline{\cup}\cF}_{\fc}(\tX)\cap x^{\nu-\delta}\CI_{\fc}(X)$ for all
$\delta>0$ such that 
\[
     u-v \in x^{2\nu-\delta}\CI_{\fc}(X), \quad \forall \delta>0.   
\]  
From the definition of the index set $\hE_{i,1}(0)\overline{\cup}F_i$, the only elements $(z,k)$ with $z=1$ are $(1,0)$ and $(1,1)$.  Thus, removing higher order terms in the asymptotic expansion of the function $v$, we can assume that 
\[
   v= \tc_*\left( \sum_{i}^{\ell} \chi_{\epsilon}(x_i)x_i ( \tilde{b}_i \log x_i + b_i) \right)
   \]
for some $\tilde{b}_i, b_i \in\CI(\bD_i)$.  Moreover, from the proof of Theorem~\ref{br.31} and Lemma~\ref{br.27}, we know that $\tilde{b}_i$ has to be an element of the kernel of $\hat{P}_{i,1}(-\sqrt{-1})=\Delta_{\omega_i}$, that is, it must be a constant.  
 
\end{proof}

We can now proceed by induction to improve this result.

\begin{proposition}
For each $m\in\bbN$, there are an index set $\cE^m =(E_1^m,\ldots, E_{\ell}^m)$ and a polyhomogeneous function $v_m\in \cA^{\cE^m}_{\fc}(\tX)\cap x^{1-\delta}\CI_{\fc}(X)$ for all $\delta>0$ such that 
\[
     u-v_m\in x^{m+\alpha}\CI_{\fc}(X)
\]
for some $\alpha>0$.  
\label{ake.15}\end{proposition}
\begin{proof}
We proceed by induction on $m$.  The case $m=1$ is given by Lemma~\ref{ake.12}.  Suppose that the statement of the proposition is true for $m=k$.  We need to show that it is true for $m=k+1$.  By our induction hypothesis, there is an index family $\cE^k$ and a polyhomogeneous function $v_k\in   \cA^{\cE^k}_{\fc}(\tX)\cap x^{1-\delta}\CI_{\fc}(X)$ for all $\delta>0$ such that 
\[
     u-v_k\in x^{k+\alpha}\CI_{\fc}(X)
\]
for some $\alpha>0$.  Since $v_k\in x^{\frac{1}{2}}\CI_{\fc}(X)$, by replacing $v_k$ with $\chi(\frac{x_i}{r})v_k$ if needed, where $\chi\in \CI_{c}([0,\infty))$ is a cut-off function with $\chi(t)=1$ for $t<1$, we can assume, by taking $r>0$ sufficiently small, that the norm of $v_k$ in $C^2_{\fc}(X)$ is sufficiently small so that 
\[
              \frac{\omega}{2}<  \omega + \sqrt{-1} \pa \db v_k < 2\omega.
\] 
In particular, $\omega_k= \omega+ \sqrt{1}\pa \db v_k$ is also a K\"ahler form on $X$.   

Now, by hypothesis, the function $u_k= u-v_k$ is in $x^{k+\alpha}\CI_{\fc}(X)$ for some $\alpha>0$.  From the equation satisfied by $u$, we see that $u_k$ satisfies the equation
\begin{equation}
   \left(\frac{ (\omega_k+\sqrt{-1} \pa \db u_k)^n }{\omega_k^n}\right)e^{-u_k} = e^{F_k}, \quad \mbox{where} \quad F_k=F +v_k+ \log\left( \frac{\omega^n}{\omega_k^n}\right).
\label{ake.16}\end{equation} 
Now, since $u-v_k\in x^{k+\alpha}\CI_{\fc}(X)$, we have that
\begin{equation}
\begin{aligned}
  -F_k &=  \log\left( \frac{(\omega+\sqrt{-1}\pa\db v_k)^{n}}{\omega^{n}}\right)- v_k -F  \\ 
  &=  \left(\log\left( \frac{(\omega+\sqrt{-1}\pa\db v_k)^{n}}{\omega^{n}}\right)- v_k\right)  - \left(\log\left( \frac{(\omega+\sqrt{-1}\pa\db u)^{n}}{\omega^{n}}\right)- u\right)   \end{aligned}  
\label{ake.17}\end{equation}
is also in $x^{k+\alpha}\CI_{\fc}(X)$.  On the other hand, since $v_k$ and $F$ are polyhomogeneous, this also implies that $F_k$ is polyhomogeneous.  Thus
\[
           e^{F_k}= 1+ f_k, \quad f_k\in x^{k+\alpha}\CI_{\fc}(X)\cap\cA^{\cF_k}_{\fc}(\tX),
\]  
for some index family $\cF_{k}$ and equation \eqref{ake.16} can be rewritten as
\begin{equation}
\left( 1+ \Delta_{\omega_k} u_k + \sum_{j=2}^{n} G^{\omega_k}_{j}(u_k)\right) e^{-u_k} = 1+f_k,
\label{ake.18}\end{equation}
where 
\[
   G^{\omega_k}_j(h)= \frac{n !}{(n-j)! j!}   \left( \frac{\omega_k^{n-j} \wedge(\sqrt{-1}\pa\db h)^{j}}{\omega_k^n}\right), \quad h\in \CI_{\fc}(X).
\]
Since $u_k\in x^{k+\alpha}\CI_{\fc}(X)$, this means $G_j^{\omega_k}(u)\in x^{jk+j\alpha}\CI_{\fc}(X)$.  Thus, we deduce from \eqref{ake.16} that there exists a function
$w_k\in x^{2k+2\alpha}\CI_{\fc}(X)$ such that 
\[
    (\Delta_{\omega_k} -1)u_k= w_k + f_k.
\]  
Furthermore, since $(\Delta_{\omega_k}-\Delta_{\omega})u_k\in x^{2k+2\alpha}\CI_{\fc}(X)$, we see in fact that we can write
\begin{equation}
   (\Delta_{\omega}-1)u_k = \hat{w}_k+ f_k, \quad
\label{ake.19}\end{equation}
for some $ \hat{w}_{k}\in x^{2k+2\alpha}\CI_{\fc}(X)$.  We are now in a position to apply Theorem~\ref{br.31}, which stipulates that there exists a polyhomogeneous function $
h_{k+1}\in\cA^{\widehat{\cE}_{1}(0)\overline{\cup}\cF_k}_{\fc}(\tX)\cap x^{k+\alpha-\delta}\CI_{\fc}(X)$ for all $\delta>0$ such that
\[
           u_k- h_{k+1}\in x^{2k+2\alpha-\delta}\CI_{\fc}(X) \quad \forall \; \delta>0.  
\]
Thus, the function $v_{k+1}= v_k+ h_{k+1}$ satisfies the statement of the proposition for $m=k+1$, which completes the inductive step.
\end{proof}

This finally gives the following result.  
\begin{theorem}
Suppose the divisor $\bD$ is smooth and such that $K_{\bX}+[\bD]>0$.    
Suppose also that the volume form $\Omega$ on $\bX$ used to define the K\"ahler form $\omega$ in \eqref{fff.1} is chosen so that induced Hermitian metric on $K_{\bX}+[\bD]$ restricts on
$\left.(K_{\bX}+ [\bD])\right|_{\bD_i}=K_{\bD_i}$ to the Hermitian metric induced by the volume form $n\omega^{n-1}_{\KE,i}$.  Then there exists an index family $\cE$ such that the unique solution $u$ to the Monge-Amp\`ere equation
\eqref{fff.2} is such that
\[
              u\in \cA^{\cE}_{\fc}(\tX)\cap x^{1-\delta}\CI_{\fc}(X)\quad \forall \delta>0.
\]
Furthermore, the index family $\cE=(E_1,\ldots,E_{\ell})$  can always be chosen so that $(1,k)\notin E_i$ for $k>1$.     
\label{ake.20}\end{theorem}
\begin{proof}
This is a direct consequence of Proposition~\ref{ake.15}.  
\end{proof}
This has the following consequences for the K\"ahler-Einstein metric $\omega_{\KE}$ of \eqref{ake.4}.  

\begin{corollary}
Suppose the divisor $\bD$ is smooth and such that $K_{\bX}+[\bD]>0$.  Then there exists an index family $\cE=(E_1,\ldots,E_{\ell})$ such that
\[
           \omega_{\KE}\in \cA^{\cE}_{\fc}(\tX;\Lambda^{2}({}^{\fc}T^{*}\tX)\otimes \bbC),
\] 
where $\omega_{\KE}$ is the K\"ahler-Einstein metric \eqref{ake.4} obtained by solving the Monge-Amp\`ere equation \eqref{ake.5}.  Furthermore, each index set $E_i$ can be chosen so that
\[
        (z,k)\in E_i, \; z\le 1 \quad \Longrightarrow \quad (z,k)\in \{(0,0), (1,0),(1,1)\}.
\]
Finally, $\omega_{\KE}$ has well-defined standard spatial asymptotics given by 
$\{\omega_{\KE,i}\}_{i=1}^{\ell}$ and $\{1\}_{i=1}^{\ell}$ where $\omega_{\KE,i}$ is the K\"ahler-Einstein metric given by \eqref{ake.8}. 
\label{ake.21}\end{corollary}  

The logarithmic terms in the asymptotic expansion of the K\"ahler-Einstein metric 
do not necessarily appear.  For instance, when $n=\dim_{\bbC} \bX=1$, we know from \cite{Albin-Aldana-Rochon} that $\omega_{\KE}$ is in 
$\CI_{\fc}(\tX; \Lambda^2 ({}^{\fc}T^{*}\tX)\otimes \bbC)$, so that there is no logarithmic term in that case.  This is very particular to complex dimension $1$.  As the next criterion indicates, logarithmic terms do appear when $n\ge 2$.  

\begin{theorem}
Let $u$ be the function of Theorem~\ref{ake.20}.  Then the constant $\tilde{b}_i$ such that 
\[
                     u-\tilde{b}_i x_i\log x_i \in x_i\CI_{\fc}(X)
\] 
is given by
\[
                         \tilde{b}_i =- \frac{2(n-1)}{3} \frac{\int_{\bD_i} c_1(T\bD_i)^{n-2} \cup c_1(N\bD_i)}{\int_{\bD_i} c_1(T\bD_i)^{n-1}},
\]    
where $N\bD_i$ is the normal bundle of $\bD_i$ in $\bX$.  
\label{clt.6}\end{theorem}
\begin{proof}
Choose $\Omega$ as in Theorem~\ref{ake.20}.  By Lemma~\ref{clt.1}, the coefficient of the term of order $x_i$ in the Taylor series of $e^{F}$ at $\bD_i$ is given by
\[
         \phi_i= \frac{(n-1) \omega_{\KE,i}^{n-2} \wedge (\sqrt{-1} \left.\Theta_{\bD}\right|_{\bD_i})}{\omega_{\KE,i}^{n-1}}.
\]
The term $\tilde{b}_i$ will be nonzero provided the function $\phi_i$ is not in the image of the $\db$-Laplacian $\Delta_{\omega_{\KE,i}}$.  This is the case provided 
\[
        I_i = \frac{\int_{\bD_i} \phi_i \omega_{\KE,i}^{n-1}}{\int_{\bD_i} \omega_{\KE,i}^{n-1} } \ne 0.
\]
But from the definition of $\phi_i$, we have that 
\[
         I_i =  -(n-1) \frac{ \int_{\bD_i} c_1(T\bD_i)^{n-2} \cup c_1(N\bD_i) }{\int_{\bD_i} c_1(T\bD_i)^{n-1}}.
\]
From Lemma~\ref{ake.12}, we know the constant $\tilde{b}_i$ must be such that
\[
           (\Delta_{\omega}-1)(\tilde{b} x_i\log x_i)  - I_i x_i \in x_i^{1+\delta}\CI_{\fc}(X) \quad \mbox{for some} \; \delta>0.  
\] 
From the asymptotic behavior \eqref{br.9} of $\Delta_{\omega}$ near $\bD_i$, 
we conclude $\tilde{b}_i= \frac{2 I_i}{3}$ and the result follows.

\end{proof}

\begin{example}
Take $\bX=\bbC\bbP_2$ and $\bD_1\subset \bbC\bbP_2$ to be a smooth curve of degree $d\ge 4$.  In that case, $K_{\bX}+ [\bD_1]>0$ so that there is a K\"ahler-Einstein metric on $X= \bbC\bbP_2 \setminus \bD_1$ obtained by solving the Monge-Amp\`ere equation \eqref{ake.5}.  By Theorem~\ref{clt.6}, the coefficient $\tilde{b}_1$ of the term of order $x_1\log x_1$ in the asymptotic expansion of the solution $u\in \CI_{\fc}(X)$ to the Monge-Amp\`ere equation is given by
\[
    \tilde{b}_1= \frac{2d}{3(d-3)}.
\]
\label{clt.7}\end{example}

When $\bD$ has normal crossings, this criterion can also be used to show in certain cases that the K\"ahler-Einstein metric is not an asymptotically tame polyfibred cusp metric.  Here is a simple example.

\begin{example}
Take $\bX= \bbC\bbP_3$ with $\bD=\bD_1+\bD_2$, where $\bD_1$ is a hyperplane and $\bD_2\subset \bbC\bbP^3 $ is a smooth hypersurface of degree $d\ge 4$ intersecting $\bD_1$ transversely.  In that case, $K_{\bX}+ [\bD]>0$ and there is a corresponding K\"ahler-Einstein metric $\omega_{\KE}$.  By Theorem~\ref{ake.10}, it has standard spatial asymptotics given by $\{\omega_{\KE,i}\}_{i=1}^{2}$ and $\{1\}_{i=1}^{2}$.  By the previous example, the asymptotic expansion of $\omega_{\KE,1}$ involves some logarithmic terms.  In particular, the K\"ahler-Einstein metric on $X$ cannot be an asymptotically tame polyfibred cusp metric in this case, that is, the solution $u$ to the Monge-Amp\`ere equation \eqref{ake.3} cannot be in $\CI_{\fc}(\tX)$.  
\label{clt.8}\end{example}

\bibliography{harfqm}
\bibliographystyle{amsplain}

\end{document}